\documentclass[10pt,a4paper]{article}
\usepackage[numbers,sort&compress]{natbib}
\usepackage[T1]{fontenc}
\usepackage[utf8]{inputenc}
\usepackage{authblk}
\usepackage{amsmath,amssymb,amsthm,esint,bm}%,mathrsfs
\usepackage{mathrsfs}
\usepackage{bookmark}
\usepackage{amsmath}
\allowdisplaybreaks[3]

\newtheorem{definition}{Definition}[section]
\newtheorem{theorem}[definition]{Theorem}
\newtheorem{lemma}[definition]{Lemma}

\newtheorem{corollary}[definition]{Corollary}
\theoremstyle{remark}
\newtheorem{remark}[definition]{Remark}
\numberwithin{equation}{section}

\newcommand{\R}{\mathbb{R}}

\allowdisplaybreaks[4]

\title{Universal Potential Estimates for Mixed Local and Nonlocal Nonlinear Measure Data Problems}

\author[a]{Lingwei Ma}
\author[b]{Qi Xiong}
\author[a]{Zhenqiu Zhang\thanks{Corresponding author.}}

\affil[a]{School of Mathematical Sciences and LPMC, Nankai University, Tianjin 300071, P.R. China}
\affil[b]{School of Mathematics, Southwest Jiaotong University, Chengdu 610031, Sichuan, P.R. China}
\date{\today}
\usepackage{hyperref}
\begin{document}
\maketitle
\footnotetext[1]{E-mail: malingwei@nankai.edu.cn (L. Ma), xq@swjtu.edu.cn (Q. Xiong), zqzhang@nankai.edu.cn (Z. Zhang).}

\begin{abstract}
This paper presents the nonlinear potential theory for mixed local and nonlocal $p$-Laplace type equations with coefficients and measure data, involving both superquadratic and subquadratic cases. We prove a class of universal pointwise estimates for the solution and its gradient via Riesz and Wolff potentials. These are achieved by imposing various low regularity conditions on the coefficient of the local term, while the kernel coefficient for the nonlocal term is merely assumed to be measurable. The key to these proofs lies in introducing a novel fractional maximum function that can capture both local and nonlocal features simultaneously, and in establishing pointwise estimates for such maximum operators of the solution and its gradient.
Notably, our universal potential estimates not only precisely characterize the oscillations of solutions, but also identify the borderline case that bounds their size, thereby refining the pointwise potential estimates available in earlier work.



\end{abstract}
%\renewcommand{\thepage}{\roman{page}}
%\setcounter{page}{1}

%%\tableofcontents

\section{Introduction}\label{section1}
Over the past decade, the mixed diffusion problems characterized by the coexistence of local and nonlocal operators have attracted significant theoretical interest owing to their widespread applications.
The simplest prototype of mixed operators is $-\Delta+(-\Delta)^s$,
which lacks scale invariance and arises from the simultaneous presence of a classical random walk and a L\'{e}vy flight in stochastic processes.
This type of operator can be used to model a variety of important practical phenomena, including heat transport in magnetized plasmas \cite{BdC13}, population dynamics \cite{DPV23}, optimal foraging strategies \cite{DV21} and the spread of pandemics \cite{EG07}.
To acquire a clearer understanding of these phenomena,
considerable attention has been devoted to investigating the qualitative properties and regularity theory of solutions to mixed local-nonlocal elliptic and parabolic equations with regular data, as detailed in \cite{BDVV22, BKL24, BVDV21, DM24, FSZ22, GK22, GK24, N23} and the references therein.
The present paper focuses on the following more general mixed local and nonlocal nonlinear elliptic equation with measure data
\begin{equation}\label{eq1}
-\operatorname{div}\mathcal{A}(x, Du) -\mathcal{L}_\Phi u = \mu ~~\text{in}~~\Omega,
\end{equation}
where $\Omega$ is a bounded domain in $\mathbb{R}^n$ with $n \geq 2$, and $\mu$ is a signed Borel measure with finite total mass on
$\mathbb R^{n}$, denoted by $\mathcal M (\mathbb R^{n})$.
The Carath\'{e}odory vector field $\mathcal{A}(x,z): \Omega\times\mathbb{R}^{ n}\rightarrow\mathbb{R}^{n}$ is assumed to be $C^1$-regular on $\R^n$ for $p\geq 2$ and on $\R^n\setminus \{0\}$ for $1<p<2$ with respect to the gradient variable $z$ and satisfying the following growth and ellipticity conditions
\begin{equation}\label{growth1}
  \left|\mathcal{A}(x,z)\right||z|+\left|\partial_z\mathcal{A}(x,z)\right||z|^2\leq\Lambda|z|^p,
\end{equation}
\begin{equation}\label{ellipticity}
\left\langle \partial_z\mathcal{A}(x,z)\xi,\xi\right\rangle\geq\Lambda^{-1}|z|^{p-2}|\xi|^{2},
\end{equation}
for almost all $x\in\Omega$ and every $z,\,\xi\in\mathbb{R}^{n}$ with the structure constant $\Lambda\geq 1$. The nonlocal operator $\mathcal{L}_{\Phi}$ is defined by
\begin{equation*}
-\mathcal{L}_\Phi  u(x) = \text{P.V.}\int_{\mathbb{R}^n}\Phi(u(x)-u(y)){K}(x,y) \operatorname{d}\!y,
\end{equation*}
where $\text{P.V.}$ stands for the Cauchy principle value of the integral, the continuous function $\Phi: \R\rightarrow\R$ satisfies $\Phi(0)=0$ together with the following monotonicity and growth assumptions
\begin{equation}\label{mono}
  \Lambda^{-1}|t|^{p} \leq  \Phi(t)t \leq \Lambda|t|^{p}\,\,\mbox{for all}\,\, t\in\R,
\end{equation}
and the measurable kernel $ K: \mathbb{R}^n \times \mathbb{R}^n \to [0, \infty)$ fulfills the following ellipticity and growth properties
\begin{equation}\label{ellip}
  \Lambda^{-1}|x - y|^{-n-sp} \leq  K(x, y) \leq \Lambda|x - y|^{-n-sp}
\end{equation}
for almost everywhere $(x, y) \in  \mathbb{R}^n \times \mathbb{R}^n$ with the fractional order $s \in (0,1)$. As  noted in
\cite[section 1.5]{KuMinSi}, the function $\Phi$ and the kernel $K$ in the nonlocal operator $\mathcal{L}_{\Phi}$ can be reduced to
$\Phi(t)=|t|^{p-2}t$ and $K$ is symmetric in the sense that
${K}(x, y) = {K}(y, x)$ for almost everywhere $( x, y) \in  \mathbb{R}^n \times \mathbb{R}^n$ .

An effective approach to studying the regularity theory of the classical Poisson equation $-\Delta u=\mu$ is based on a representation formula for the solution. This formula enables the solution and its gradient to be controlled pointwise via the Riesz potentials of the measure data and encompasses various regularity properties of the solution within a unified framework.
It is natural to ask whether these pointwise potential estimates remain valid for general partial differential equations that lack explicit formulae representing their solutions. The first affirmative answer can be attributed to the pioneering work \cite{KiMa, KiMa2, TrWa}, which obtained optimal pointwise estimates for solutions to the $p$-Laplace type elliptic equation $-\operatorname{div}\mathcal{A}(x, Du)  = \mu$ with measurable coefficients via the Wolff potential ${\bf{W}}^{\mu}_{1,p}(x,R)$ of measure data $\mu$.
In general, for $0<\beta<\frac{n}{p}$, the truncated Wolff potential of measure $\mu\in \mathcal M (\mathbb R^{n})$ is defined as
\begin{equation*}
  {\bf{W}}^{\mu}_{\beta,p}(x,R) := \int_{0}^{R}\left[\frac{|\mu|(B_{\rho}(x))}{\rho^{n-\beta p}}\right]^{1/(p-1)}\frac{\mathrm{d}\rho}{\rho},
\end{equation*}
where $B_{\rho}(x)$ denotes an open ball in $\mathbb{R}^{n}$ with center $x$ and radius $\rho$. Particularly, when $p=2$, the nonlinear Wolff potential ${\bf{W}}^{\mu}_{\frac{\beta}{2},2}(x,R)$ reduces to the linear Riesz potential
\begin{equation*}
  {\bf{I}}^{\mu}_{\beta}(x,R) := \int_{0}^{R}\frac{|\mu|(B_{\rho}(x))}{\rho^{n-\beta}}\frac{\mathrm{d}\rho}{\rho}.
\end{equation*}
Nearly two decades afterwards, Duzaar and Mingione \cite{DuMi2, Min} generalized the previous pointwise estimates to the gradient level using the Wolff potential ${\bf{W}}^{\mu}_{\frac{1}{p},p}(x,R)$, provided that the coefficient was Dini-continuous. Kuusi and Mingione \cite{KM13} later unexpectedly demonstrated that these pointwise gradient estimates could be improved by employing the linear Riesz potential ${\bf{I}}^{\mu}_{1}(x,R)$.
It is worth highlighting that they developed universal potential estimates for the $p$-Laplace type elliptic equation under prescribed regularity conditions on the partial map $x\mapsto\mathcal{A}(x,\cdot)$, refer to \cite{KuMi}. These pointwise estimates allow us to control the oscillations of solutions and their gradients,
thereby facilitating a unified approach to deriving both Calder\'{o}n-Zygmund estimates and continuity criteria for solutions in virtually any reasonable function space. Subsequently, both pointwise potential estimates and universal potential estimates have been extended to various types of local elliptic and parabolic problems, we
refer to \cite{BM, CiSc, CV17, DKM14, DZ24, KM14-1, KuMi3, MZ, MZ24, MZZ, PV08, PV09, XZM} and the references therein.

With respect to the potential theory for the purely nonlocal problems, innovative work \cite{KuMinSi, NOS24} showed that pointwise estimates for solutions to the linear fractional Laplace equation $(-\Delta)^s u=\mu$ have analogues in the nonlinear setting $-\mathcal{L}_\Phi u= \mu$ via the Wolff potential ${\bf{W}}^{\mu}_{s,p}(x,R)$. Furthermore, the universal potential estimates for solutions to fractional $p$-Laplace type equations were also obtained in \cite{NOS24}. In the specific instance of $p=2$ and $s\in(\frac{1}{2},1)$, the pointwise  estimates for the gradient in terms of the Riesz potential ${\bf{I}}^{\mu}_{2s-1}(x,R)$ were established in \cite{KNS22, DKLN24}. However, such pointwise gradient estimates in the general case where $p\neq2$, as well as
the universal potential estimates for the gradient of solutions to the purely nonlocal equations remain unsolved.

Regarding mixed local and nonlocal problems, Byun and Song \cite{BS23}
considered the mixed local and nonlocal elliptic equation \eqref{eq1} under the assumptions \eqref{growth1}-\eqref{ellip}, and proved the pointwise potential estimate for the solution as follows
\begin{equation}\label{mixPoten-sol}
|u(x)|
   \leq C \left[\left(\fint_{B_{R}(x)} |u|^{q_0} \operatorname{d}\! \xi\right)^{\frac{1}{q_0}}+\operatorname{Tail}(u;x,R)\right]
   +C{\bf{W}}^{\mu}_{1,p}(x,R)
\end{equation}
for the ball $ B_{R}(x) \subset \Omega$, $q_0=\max\{1,p-1\}$ and $p>2-\frac{1}{n}$, where the nonlocal tail term $\operatorname{Tail}(u;x,R)$ is defined in \eqref{Tail}.
In the specific scenario where the vector field $\mathcal{A}$ is coefficient free, Chlebicka et al. \cite{CSYZ24} extended the Riesz potential estimate to the gradient
\begin{equation}\label{mixPoten-grad}
|Du(x)|
   \leq C \left[\left(\fint_{B_{R}(x)} |Du|^{q_0} \operatorname{d}\! \xi\right)^{\frac{1}{q_0}}+R^{\frac{\tau-p}{p-1}}\operatorname{Tail}(u-(u)_{B_R(x)};x,R)\right]
   +C\left[{\bf{I}}^{\mu}_{1}(x,R)\right]^{\frac{1}{p-1}}
\end{equation}
for any $\tau\in(0,1)$.
Such pointwise gradient estimates are ascribed to the fact that the fractional $W^{s,p}$-capacity generated by the nonlocal term can be controlled by the $W^{1,p}$-capacity induced by the local term when $p > sp$.
Consequently, local diffusion is the dominant feature in such mixed model, leading to gradient regularity that resembles those of purely local equations, as detailed in \cite{DM24}.
Very recently, Ma et al. \cite{MXZ25} established both pointwise and universal potential estimates with optimal tail terms for solutions to the mixed local and nonlocal parabolic equation in the case of $p=2$.

Our contribution to this paper is to establish universal potential estimates for solutions and their gradients to the mixed local and nonlocal nonlinear elliptic equation \eqref{eq1}, provided that the partial map $x\mapsto\mathcal{A}(x,\cdot)$ satisfies low regularity conditions and the kernel coefficient $K$ is merely assumed to be measurable.
These estimates encompass both superquadratic and subquadratic scenarios, and also reflect both local and nonlocal characters of the mixed operators, undoubtedly constituting a valuable addition to the potential theory of mixed problems.
To illustrate the main results of this paper, we begin with the definition of local weak solutions to the mixed local and nonlocal elliptic  equation \eqref{eq1}.

\begin{definition} Let $\mu \in \mathcal M (\mathbb R^{n})$, $p>1$ and $s\in(0,1)$. A function
$u\in W^{1,p}_{\rm loc}(\Omega)\cap {\mathcal L}^{p-1}_{sp}(\mathbb{R}^n)$
is a weak solution to the equation \eqref{eq1}, if
\begin{eqnarray*}
\int_{\Omega}\mathcal{A}(x,Du)\cdot D\varphi\operatorname{d}\! x + \int_{\mathbb{R}^n} \int_{\mathbb{R}^n} \Phi(u(x)-u(y))(\varphi(x)-\varphi(y)) K(x,y) \operatorname{d}\!x \operatorname{d}\!y= \int_{\Omega} \varphi \operatorname{d}\!\mu
\end{eqnarray*}
holds for any test function $\varphi\in  C_0^{\infty}(\Omega)$
with compact support contained in $\Omega$, where ${\mathcal L}^{p-1}_{sp}(\mathbb{R}^n)$ represents the slowly increasing function space defined as
	$${\mathcal L}^{p-1}_{sp}(\mathbb{R}^n):= \left \{u \in L^{p-1}_{\operatorname{loc}}(\mathbb{R}^n) \mathrel{\Big|} \int_{\mathbb{R}^n} \frac{|u(y)|^{p-1}}{1+|y|^{n+sp}} \operatorname{d}\!y < \infty \right \}.$$
\end{definition}
The nonlocal tail term is indispensable when deriving regularity estimates for solutions to equations involving fractional order differential operators.
Here we define the nonlocal tail of a measurable function $f:\R^n\rightarrow\R$ as follows
\begin{equation}\label{Tail}
    {\rm Tail}(f;x_0,r) := \left(r^p\int_{\R^n\setminus B_r(x_0)}\frac{|f(x)|^{p-1}}{|x-x_0|^{n+sp}} \operatorname{d}\!x\right)^{\frac{1}{p-1}}.
\end{equation}
It is evident that if $f$ belongs to ${\mathcal L}^{p-1}_{sp}(\R^n)$, then the tail term $ {\rm Tail}(f;x_0,r)<\infty$ for any $x_0$ and $r>0$.

We are now in a position to state the first result of this paper, namely, a universal potential estimate of solutions to the mixed local and nonlocal elliptic equation \eqref{eq1} with merely measurable coefficients.
\begin{theorem}\label{OscPotenSolMea}
Let $u\in W^{1,p}_{\rm loc}(\Omega)\cap {\mathcal L}^{p-1}_{sp}(\mathbb{R}^n)$ be a weak solution to \eqref{eq1} under assumptions \eqref{growth1}--\eqref{ellip} with $s \in (0,1)$, $p>2-\frac{1}{n}$ and $q_0=\max\{1,p-1\}$.
Given $B_R(x_0)\subset\Omega$ with $R\leq R_0$ for some $R_0=R_0(n,p,s,\Lambda,\operatorname{diam}(\Omega))>0$, then for almost all $x,\,y \in B_{\frac{R}{8}}(x_0)$ and any $\tilde{\alpha}\in(0,\alpha_m)$, there exists a positive constant $C$ depending only on $n$, $p$, $s$, $\Lambda$, $\tilde{\alpha}$ and $\operatorname{diam}(\Omega)$ such that
\begin{eqnarray}\label{OscPotenSolMea-est}
    |u(x)-u(y)|
   &\leq&C \left[\left(\fint_{B_{R}(x_0)} |u|^{q_0} \operatorname{d}\! x\right)^{\frac{1}{q_0}}+\operatorname{Tail}(u-(u)_{B_R(x_0)};x_0,R)\right]\left(\frac{|x-y|}{R}\right)^{\alpha}\nonumber\\
   &&+C\left[{\bf{W}}^{\mu}_{1-\frac{\alpha(p-1)}{p},p}(x,R)+{\bf{W}}^{\mu}_{1-\frac{\alpha(p-1)}{p},p}(y,R)\right]|x-y|^\alpha
\end{eqnarray}
holds uniformly in $\alpha\in[0,\tilde{\alpha}]$, where
  $\alpha_m$ is given by Lemma \ref{v-holder}.
\end{theorem}

An interesting corollary of the proof of Theorem \ref{OscPotenSolMea} is the following pointwise estimate for the truncated Hardy–Littlewood maximal function of solutions.
\begin{corollary}\label{PotenmaxMea}
Let $u\in W^{1,p}_{\rm loc}(\Omega)\cap {\mathcal L}^{p-1}_{sp}(\mathbb{R}^n)$ be a weak solution to \eqref{eq1} under assumptions \eqref{growth1}--\eqref{ellip} with $s \in (0,1)$, $p>2-\frac{1}{n}$ and $q_0=\max\{1,p-1\}$.
Given $B_R(x)\subset\Omega$ with $R\leq R_0$ for some $R_0=R_0(n,p,s,\Lambda,\operatorname{diam}(\Omega))>0$, then there exists a positive constant $C$ depending only on $n$, $p$, $s$, $\Lambda$ and $\operatorname{diam}(\Omega)$ such that
\begin{equation*}
  {\bf M}(u)(x,R)
   \leq C \left[\left(\fint_{B_{R}(x)} |u|^{q_0} \operatorname{d}\! \xi\right)^{\frac{1}{q_0}}+\operatorname{Tail}(u-(u)_{B_R(x)};x,R)\right]
   +C{\bf{W}}^{\mu}_{1,p}(x,R),
\end{equation*}
where ${\bf M}(u)$ is the truncated Hardy-Littlewood maximal operator of $u$ defined by
$${\bf M}(u)(x,R):=\displaystyle\sup_{0<r\leq R} \fint_{B_r(x)}|u(y)|\operatorname{d}\!y.$$
\end{corollary}
\begin{remark}
We observe that the aforementioned pointwise estimate for the truncated Hardy–Littlewood maximal function of solutions is an improvement on the pointwise estimate \eqref{mixPoten-sol} for the solution itself, as established in \cite[Theorem 1.4]{BS23}.
\end{remark}
In order to increase the exponent $\tilde{\alpha}$ in Theorem \ref{OscPotenSolMea}, we require additional regularity assumptions on the averaged modulus of continuity of the coefficients, as defined below.
\begin{definition}\label{coef}
The averaged modulus of continuity of $x\mapsto\mathcal{A}(x,\cdot)$ is defined by
\begin{equation*}
\omega(r):=\left[\sup_{{\substack{z\in \R^n\backslash\{\bf 0\}\\x\in B_r(x)\subset \Omega}}} \fint_{B_{r}(x)}\left[A(z,B_r(x))(y)\right]^2
\operatorname{d}\!y\right]^{\frac{1}{2}}\,,
\end{equation*}
where the function
\begin{equation}\label{coefA}
  A(z,B_r(x))(y):=\frac{\left|\mathcal{A}(y,z)-\left(\mathcal{A}
(\cdot,z)\right)_{B_{r}(x)}\right|}{|z|^{p-1}},
\end{equation}
and
$\left(\mathcal{A}(\cdot,z)\right)_{B_{r}(x)}$ denotes the integral average of $\mathcal A$ over the ball $B_{r}(x)$, given by
\begin{equation*}
\left(\mathcal{A}(\cdot,z)\right)_{B_{r}(x)}:=\fint_{B_{r}(x)}\mathcal{A}(y,z)\operatorname{d}\!y
=\frac{1}{\left|B_{r}(x)\right|}\int_{B_{r}(x)}\mathcal{A}(y,z)\operatorname{d}\!y\,.
\end{equation*}
\end{definition}
Assuming that $\omega(r)$ is small $\operatorname{BMO}$ regular, the exponent $\tilde{\alpha}$ can be chosen arbitrarily close to $1$.
\begin{theorem}\label{OscPotenSolBMO}
Let $u\in W^{1,p}_{\rm loc}(\Omega)\cap {\mathcal L}^{p-1}_{sp}(\mathbb{R}^n)$ be a weak solution to \eqref{eq1} under assumptions \eqref{growth1}--\eqref{ellip} with $s \in (0,1)$, $p>2-\frac{1}{n}$ and $q_0=\max\{1,p-1\}$. Given $B_R(x_0)\subset\Omega$, $m\in(0,1-s)$ and $\max\{0, \frac{m(2-p)}{p-1}\}<\sigma<\min\{\frac{1}{p-1},(1-s)p'\}$, where $p'$ is the conjugate exponent of $p$. For almost all $x,\,y \in B_{\frac{R}{8}}(x_0)$ and any $\tilde{\alpha}\in(0,1)$, if
\begin{equation}\label{smallBMO}
  \displaystyle\limsup_{r\rightarrow 0}\omega(r)\leq\delta
\end{equation}
for some  $\delta\equiv\delta(n,p,s,\Lambda,\tilde{\alpha},\sigma,\operatorname{diam}(\Omega))>0$, then there exists a positive constant $C$ depending only on
$n,\,p,\,s,\,\Lambda,\,\tilde{\alpha},\,m,\,\sigma,\,\omega(\cdot)$ and $\operatorname{diam}(\Omega)$ such that
\begin{eqnarray}\label{OscPotenSolBMO-est}
    |u(x)-u(y)|
   &\leq&C \left[\left(\fint_{B_{R}(x_0)} |u|^{q_0} \operatorname{d}\! x\right)^{\frac{1}{q_0}}+R^{-\sigma}\operatorname{Tail}(u-(u)_{B_R(x_0)};x_0,R)\right]\left(\frac{|x-y|}{R}\right)^{\alpha}\nonumber\\
   &&+C\left[{\bf{W}}^{\mu}_{1-\frac{\alpha(p-1)}{p},p}(x,R)+{\bf{W}}^{\mu}_{1-\frac{\alpha(p-1)}{p},p}(y,R)\right]|x-y|^\alpha
\end{eqnarray}
holds uniformly in $\alpha\in[0,\tilde{\alpha}]$.
\end{theorem}

Furthermore, it is possible for the exponent $\tilde{\alpha}$ to be equal to $1$, provided that we impose the Dini-$\operatorname{VMO}$ regularity condition on $\omega(r)$. Under this scenario, the oscillation estimate is derived by splitting it into subquadratic and superquadratic cases.

\begin{theorem}\label{OscPotenSolDiniVMO}
Let $u\in W^{1,p}_{\rm loc}(\Omega)\cap {\mathcal L}^{p-1}_{sp}(\mathbb{R}^n)$ be a weak solution to \eqref{eq1} under assumptions \eqref{growth1}--\eqref{ellip} with $s \in (0,1)$.
Given $B_R(x_0)\subset\Omega$, $m\in(0,1-s)$, $\max\{sp'-1,0\}<a<\sigma_3$ and $\frac{1}{p-1}-a<\sigma<\min\{\frac{1}{p-1},(1-s)p'\}$. Assume that  $[\omega(r)]^{\sigma_1'}$ is Dini-$\operatorname{VMO}$ regular, that is,
\begin{equation}\label{Dinivmo}
  \int_{0}^{r}[\omega(\rho)]^{\sigma_1'}\frac{\operatorname{d}\!\rho}{\rho}<\infty\,\,\mbox{for any} \,\, r>0,
\end{equation}
where $\sigma_1'$ is given by \eqref{exponent2}.
\begin{itemize}
\item  If $2-\frac{1}{n}<p\leq2$, then there exists a positive constant $C\equiv C(n,p,s,\Lambda,m,a,\sigma,\omega(\cdot),\operatorname{diam}(\Omega))$ such that %$\varsigma$
    \begin{eqnarray}\label{OscPotenSolDiniVMOsub}
    |u(x)-u(y)|
   &\leq&C \left[\fint_{B_{R}(x_0)} |u| \operatorname{d}\! x+R^{-\sigma}\operatorname{Tail}(u-(u)_{B_R(x_0)};x_0,R)\right]\left(\frac{|x-y|}{R}\right)^{\alpha}\nonumber\\
   &&+C\left[{\bf{I}}^{\mu}_{p-\alpha(p-1)}(x,R)+{\bf{I}}^{\mu}_{p-\alpha(p-1)}(y,R)\right]^{\frac{1}{p-1}}|x-y|^\alpha
\end{eqnarray}
  holds uniformly in $\alpha\in[0,1]$ for almost all $x,\,y \in B_{\frac{R}{16}}(x_0)$.
  \item If $p>2$, then there exists a positive constant $C\equiv C(n,p,s,\Lambda,m,a,\sigma,\omega(\cdot),\operatorname{diam}(\Omega))$ such that \eqref{OscPotenSolBMO-est} holds uniformly in $\alpha\in[0,1]$ for almost all $x,\,y \in B_{\frac{R}{8}}(x_0)$, where $q_0=p-1$.
\end{itemize}
\end{theorem}

Finally, if $\omega(r)$ satisfies the Dini-H\"{o}lder condition, we establish a universal potential estimate for the gradient of solutions, expressing their local  H\"{o}lder continuity.
To a certain extent, our following oscillation estimate generalizes the previous result (see \cite[Theorem 5]{DM24} with $p=\gamma$) to a mixed local and nonlocal elliptic equation involving low regular coefficients and measure data.
\begin{theorem}\label{OscPotenGrad}
Let $u\in W^{1,p}_{\rm loc}(\Omega)\cap {\mathcal L}^{p-1}_{sp}(\mathbb{R}^n)$ be a weak solution to \eqref{eq1} under assumptions \eqref{growth1}--\eqref{ellip} with $s \in (0,1)$.
 Given $B_R(x_0)\subset\Omega$, $m\in(0,1-s)$, $\max\{sp'-1,0\}<a<\sigma_3$ and $\frac{1}{p-1}-a<\sigma<\min\{\frac{1}{p-1},(1-s)p'\}$. Assume that $[\omega(\cdot)]^{\sigma_1'}$ is Dini-H\"{o}lder regular of order $\tilde{\alpha}\in[0,\alpha_0)$, that is,
\begin{equation}\label{Diniholder}
     %\sup_{0<r\leq R}
     \int_{0}^{r}\frac{[\omega(\rho)]^{\sigma_1'}}{\rho^{\tilde{\alpha}}}\frac{\operatorname{d}\!\rho}{\rho}<\infty \,\,\mbox{for any} \,\, r>0,
   \end{equation}
where $\alpha_0:=\min\{\alpha_M,\,\sigma_2,\,\sigma_3+\sigma-\frac{1}{p-1},\,\frac{1}{p-1}-\sigma,\,p'(1-s)-\sigma\}$, $\sigma_2$ and $\sigma_2$ are defined in \eqref{exponent1}, $\sigma_1'$ and $\alpha_M$ are given by \eqref{exponent2} and Lemma \ref{exdecay-Dw-Lem}, respectively.
\begin{itemize}
  \item If $2-\frac{1}{n}<p\leq2$,
   then there exists a positive constant $C\equiv C(n,p,s,\Lambda, %$\varsigma$,
    m, a, \sigma, \tilde{\alpha}, \omega(\cdot),\operatorname{diam}(\Omega))$ such that
    \begin{eqnarray}\label{OscPotenGradsub}
    |Du(x)-Du(y)|
   &\leq&C \left[\fint_{B_{R}(x_0)} |Du| \operatorname{d}\! x+R^{-1-\sigma}\operatorname{Tail}(u-(u)_{B_R(x_0)};x_0,R)\right]\left(\frac{|x-y|}{R}\right)^{\alpha}\nonumber\\
   &&+C\left[{\bf{I}}^{\mu}_{1-\alpha}(x,R)+{\bf{I}}^{\mu}_{1-\alpha}(y,R)\right]^{\frac{1}{p-1}}|x-y|^\alpha
  \end{eqnarray}
  holds uniformly in $\alpha\in[0,\tilde{\alpha}]$ for almost all $x,\,y \in B_{\frac{R}{4}}(x_0)$.
  \item If $p> 2$, then there exists a constant $C\equiv C(n,p,s,\Lambda, m, a, \sigma, \tilde{\alpha}, \omega(\cdot),\operatorname{diam}(\Omega))>0$  such that
    \begin{eqnarray}\label{OscPotenGradsup}
    |Du(x)-Du(y)|
   &\leq&C \left[\left(\fint_{B_{R}(x_0)} |Du|^{p-1} \operatorname{d}\! x\right)^{\frac{1}{p-1}}+R^{-1-\sigma}\operatorname{Tail}(u-(u)_{B_R(x_0)};x_0,R)\right]\left(\frac{|x-y|}{R}\right)^{\alpha}\nonumber\\
   &&+C\left[{\bf{W}}^{\mu}_{1-\frac{(\alpha-1)(p-1)}{p},p}(x,R)+{\bf{W}}^{\mu}_{1-\frac{(\alpha-1)(p-1)}{p},p}(y,R)\right]|x-y|^\alpha
  \end{eqnarray}
  holds uniformly in $\alpha\in[0,\tilde{\alpha}]$ for almost all $x,\,y \in B_{\frac{R}{4}}(x_0)$.
\end{itemize}
\end{theorem}

\begin{remark}
The existence theory and compactness properties of a class of very weak solutions obtained as limits of approximations (SOLA) to the mixed local and nonlocal elliptic equation \eqref{eq1} have been proven in \cite[Theorem 1.3]{BS23}.
It is noteworthy that our main theorems remain valid for SOLA by utilizing compactness results through a standard approximation process, as we establish all the necessary estimates below the energy range.
\end{remark}

The building blocks for the proof of the aforementioned universal potential estimates are establishing pointwise estimates of the appropriate maximal operators of the solution and its gradient to the mixed equation \eqref{eq1} through the truncated fractional maximal functions and potential functions of measure data. For a measure $\mu\in\mathcal M (\mathbb R^{n})$, the truncated fractional maximal function is defined as
$${\bf M}_{\beta}(\mu)(x,R):=\sup_{0<r\leq R} r^{\beta}\frac{|\mu|(B_r(x))}{|B_r(x)|} \,\,\mbox{for}\,\, \beta\in[0,n].$$
In order to characterize the mixed phenomena rigorously, we shall modify the classical maximal operators of solutions in a way that simultaneously accounts for local behavior within small neighbourhoods and long-range interactions.
Specifically, We
define the truncated nonlocal fractional maximal function with order $\beta$ of the gradient $Du$ by
$$\mathfrak{M}_{\beta,q_0}(Du)(x,R):=\sup_{0<r\leq R}r^{\beta}\left[ \left(\fint_{B_r(x)}|Du(y)|^{q_0}\operatorname{d}\!y\right)^{\frac{1}{q_0}}
+r^{-1-\sigma}\operatorname{Tail}(u-(u)_{B_r(x)};x,r)\right],$$
where $\sigma\geq 0$.
Let $\alpha\in[0,1]$, we further define the truncated nonlocal fractional sharp maximal function of order $\alpha$ for the solution $u$ and its gradient $Du$ by
$$\mathfrak{M}^{\sharp}_{\alpha,q_0}(u)(x,R):=\sup_{0<r\leq R}r^{-\alpha}\left[ \left(\fint_{B_r(x)}|u(y)-(u)_{B_r(x)}|^{q_0}\operatorname{d}\!y\right)^{\frac{1}{q_0}}+r^{-\sigma}\operatorname{Tail}(u- (u)_{B_r(x)};x,r)\right]$$
and
$$\mathfrak{M}^{\sharp}_{\alpha,q_0}(Du)(x,R):=\sup_{0<r\leq R}r^{-\alpha}\left[ \left(\fint_{B_r(x)}|Du(y)-(Du)_{B_r(x)}|^{q_0}\operatorname{d}\!y\right)^{\frac{1}{q_0}}+r^{-1-\sigma}\operatorname{ Tail}(u- (u)_{B_r(x)};x,r)\right],$$
respectively.
A direct consequence of the Poincar\'{e}'s inequality is that there exists a positive constant $C$ depending only on $n$ such that
\begin{equation}\label{Sharpm-M}
  \mathfrak{M}^{\sharp}_{\alpha,q_0}(u)(x,R)\leq C\mathfrak{M}_{1-\alpha,q_0}(Du)(x,R).
\end{equation}
We now present the corresponding pointwise estimates of nonlocal maximal operators, which are interesting for their own sake, provided that the coefficients are measurable, small $\operatorname{BMO}$ regular,
Dini-$\operatorname{VMO}$ regular, and weakened Dini-H\"{o}lder regular, respectively.

\begin{theorem}\label{Maxmeasure}
Let $u\in W^{1,p}_{\rm loc}(\Omega)\cap {\mathcal L}^{p-1}_{sp}(\mathbb{R}^n)$ be a weak solution to \eqref{eq1} under assumptions \eqref{growth1}--\eqref{ellip} with $p > 2-\frac{1}{n}$ and $s \in (0,1)$. Given $B_R(x)\Subset\Omega$, $q_0=\max\{1,p-1\}$ and $\sigma=0$. For any $\tilde{\alpha}\in(0,\alpha_m)$, there exists a positive constant $C$ depending only on $n$, $p$, $s$, $\Lambda$, $\tilde{\alpha}$ and $\operatorname{diam}(\Omega)$ such that
  \begin{eqnarray}\label{Maxpoint-Measurable}
   && \mathfrak{M}^{\sharp}_{\alpha,q_0}(u)(x,R)+\mathfrak{M}_{1-\alpha,q_0}(Du)(x,R)
   \leq C\left[ {\bf M}_{p-\alpha(p-1)}(\mu)(x,R)\right]^{\frac{1}{p-1}}\nonumber\\
   &&\qquad\qquad\quad+CR^{1-\alpha} \left[\left(\fint_{B_{R}(x)} |Du|^{q_0} \operatorname{d}\! \xi\right)^{\frac{1}{q_0}}+\frac{1}{R}\operatorname{Tail}(u-(u)_{B_R(x)};x,R)\right]
  \end{eqnarray}
  holds uniformly in $\alpha\in[0,\tilde{\alpha}]$, where
  $\alpha_m\in(0,1)$ is defined as in Lemma \ref{v-holder}.
\end{theorem}

\begin{theorem}\label{MaxsmallBMO}
Let $u\in W^{1,p}_{\rm loc}(\Omega)\cap {\mathcal L}^{p-1}_{sp}(\mathbb{R}^n)$ be a weak solution to \eqref{eq1} under assumptions \eqref{growth1}--\eqref{ellip} with $p > 2-\frac{1}{n}$ and $s \in (0,1)$. Given $B_R(x)\Subset\Omega$, $q_0=\max\{1,p-1\}$, $m\in(0,1-s)$ and $\max\{0, \frac{m(2-p)}{p-1}\}<\sigma<\min\{\frac{1}{p-1},(1-s)p'\}$. For any $\tilde{\alpha}\in(0,1)$, if $\omega(r)$ satisfies small $\operatorname{BMO}$ regular condition \eqref{smallBMO} with $\delta\equiv\delta(n,p,s,\Lambda,\tilde{\alpha},\sigma,\operatorname{diam}(\Omega))>0$, then there exists a positive constant $C$ depending only on
$n,\,p,\,s,\,\Lambda,\,\tilde{\alpha},\,m,\,\sigma,\,\omega(\cdot)$ and $\operatorname{diam}(\Omega)$ such that
\begin{eqnarray}\label{Maxpoint-SmBMO}
   && \mathfrak{M}^{\sharp}_{\alpha,q_0}(u)(x,R)+\mathfrak{M}_{1-\alpha,q_0}(Du)(x,R)
   \leq C\left[ {\bf M}_{p-\alpha(p-1)}(\mu)(x,R)\right]^{\frac{1}{p-1}}\nonumber\\
   &&\qquad\qquad\quad+CR^{1-\alpha} \left[\left(\fint_{B_{R}(x)} |Du|^{q_0} \operatorname{d}\! \xi\right)^{\frac{1}{q_0}}+R^{-1-\sigma}\operatorname{Tail}(u-(u)_{B_R(x)};x,R)\right]
\end{eqnarray}
holds uniformly in $\alpha\in[0,\tilde{\alpha}]$.
\end{theorem}

\begin{theorem}\label{MaxDinivmo}
Let $u\in W^{1,p}_{\rm loc}(\Omega)\cap {\mathcal L}^{p-1}_{sp}(\mathbb{R}^n)$ be a weak solution to \eqref{eq1} under assumptions \eqref{growth1}--\eqref{ellip} with $s \in (0,1)$, and let $[\omega(r)]^{\sigma_1'}$ satisfy Dini-$\operatorname{VMO}$ regular condition \eqref{Dinivmo}. Given $B_R(x)\Subset\Omega$, $m\in(0,1-s)$, $\max\{sp'-1,0\}<a<\sigma_3$ and $\frac{1}{p-1}-a<\sigma<\min\{\frac{1}{p-1},(1-s)p'\}$.
\begin{itemize}
\item If $2-\frac{1}{n}<p\leq 2$, then there is a constant $C\equiv C(n, p, s, \Lambda,  m, a, \sigma, \omega(\cdot), \operatorname{diam}(\Omega))>0$ such that
    \begin{eqnarray}\label{Maxpoint-Dinivmosub}
   && \mathfrak{M}^{\sharp}_{\alpha,1}(u)(x,R)+\mathfrak{M}_{1-\alpha,1}(Du)(x,R)
   \leq C\left[{\bf{I}}^{\mu}_{p-\alpha(p-1)}(x,R)\right]^{\frac{1}{p-1}}\nonumber\\
   &&\qquad\qquad\quad+CR^{1-\alpha} \left[\fint_{B_{R}(x)} |Du| \operatorname{d}\! \xi+R^{-1-\sigma}\operatorname{Tail}(u-(u)_{B_R(x)};x,R)\right]
  \end{eqnarray}
  holds uniformly in $\alpha\in[0,1]$.
  \item If $p>2$, then there exists a positive constant $C\equiv C(n, p, $s$, \Lambda, m, a, \sigma, \omega(\cdot)\operatorname{diam}(\Omega))$ such that
    \begin{eqnarray}\label{Maxpoint-Dinivmosup}
   && \mathfrak{M}^{\sharp}_{\alpha,p-1}(u)(x,R)+\mathfrak{M}_{1-\alpha,p-1}(Du)(x,R)
   \leq C{\bf{W}}^{\mu}_{1-\frac{\alpha (p-1)}{p},p}(x,R)\nonumber\\
   &&\qquad\qquad+CR^{1-\alpha} \left[\left(\fint_{B_{R}(x)} |Du|^{p-1} \operatorname{d}\! \xi\right)^{\frac{1}{p-1}}+R^{-1-\sigma}\operatorname{Tail}(u-(u)_{B_R(x)};x,R)\right]
  \end{eqnarray}
  holds uniformly in $\alpha\in[0,1]$.
\end{itemize}
\end{theorem}

\begin{theorem}\label{MaxDiniholder}
Let $u\in W^{1,p}_{\rm loc}(\Omega)\cap {\mathcal L}^{p-1}_{sp}(\mathbb{R}^n)$ be a weak solution to \eqref{eq1} under assumptions \eqref{growth1}--\eqref{ellip} with $s \in (0,1)$.
 Given $B_R(x)\Subset\Omega$, $m\in(0,1-s)$, $\max\{sp'-1,0\}<a<\sigma_3$ and $\frac{1}{p-1}-a<\sigma<\min\{\frac{1}{p-1},(1-s)p'\}$. Assume that the coefficient $\omega(\cdot)$ satisfies
 \begin{equation}\label{Dinisup}
     \sup_{0<r\leq R}\frac{[\omega(r)]^{\sigma_1'}}{r^{\tilde{\alpha}}}\leq C_\omega
   \end{equation}
for some $\tilde{\alpha}\in[0,\alpha_0)$ and $C_\omega>0$, where $\alpha_0:=\min\{\alpha_M,\,\sigma_2,\,\sigma_3+\sigma-\frac{1}{p-1},\,\frac{1}{p-1}-\sigma,\,p'(1-s)-\sigma\}$.
\begin{itemize}
  \item If $2-\frac{1}{n}<p\leq 2$,
   then there exists a positive constant $C\equiv C(n, p, s, \Lambda,
    \tilde{\alpha},  m, a, \sigma, \omega(\cdot), \operatorname{diam}(\Omega),C_\omega)$ such that
    \begin{eqnarray}\label{MaxDiniholdersub}
   && \mathfrak{M}^{\sharp}_{\alpha,1}(Du)(x,R)
   \leq C\left[ {\bf M}_{1-\alpha}(\mu)(x,R)\right]^{\frac{1}{p-1}}+C\left[{\bf{I}}^{\mu}_{1}(x,R)\right]^{\frac{1}{p-1}}\nonumber\\
   &&\qquad\qquad\quad+CR^{-\alpha} \left[\fint_{B_{R}(x)} |Du| \operatorname{d}\! \xi+R^{-1-\sigma}\operatorname{Tail}(u-(u)_{B_R(x)};x,R)\right]
  \end{eqnarray}
  holds uniformly in $\alpha\in[0,\tilde{\alpha}]$, where
  $\alpha_M\in(0,1)$ is given by Lemma \ref{exdecay-Dw-Lem}.
  \item If $p> 2$,
   then there is a constant $C\equiv C(n, p, s, \Lambda, \tilde{\alpha},  m, a, \sigma,\omega(\cdot), \operatorname{diam}(\Omega),C_\omega)>0$ such that
    \begin{eqnarray}\label{MaxDiniholdersup}
   && \mathfrak{M}^{\sharp}_{\alpha,p-1}(Du)(x,R)
   \leq C\left[ {\bf M}_{1-\alpha(p-1)}(\mu)(x, R)\right]^{\frac{1}{p-1}}+C{\bf{W}}^{\mu}_{\frac{1}{p},p}(x,R)\nonumber\\
   &&\qquad\qquad+CR^{-\alpha} \left[\left(\fint_{B_{R}(x)} |Du|^{p-1} \operatorname{d}\! \xi\right)^{\frac{1}{p-1}}+R^{-1-\sigma}\operatorname{Tail}(u-(u)_{B_R(x)};x,R)\right]
  \end{eqnarray}
  holds uniformly in $\alpha\in[0,\tilde{\alpha}]$.
\end{itemize}
\end{theorem}

\begin{remark}
When the partial map $x\mapsto\mathcal{A}(x,\cdot)$ is Dini-$\operatorname{VMO}$ regular, we can immediately derive the following pointwise estimate for the truncated Hardy–Littlewood maximal function of the gradient in terms of Theorem \ref{MaxDinivmo} with $\alpha=1$. If $2-\frac{1}{n}<p\leq2$, then there holds that
$${\bf M}(Du)(x,R)
   \leq C\left[{\bf{I}}^{\mu}_{1}(x,R)\right]^{\frac{1}{p-1}}+C \left[\fint_{B_{R}(x)} |Du| \operatorname{d}\! \xi+R^{-1-\sigma}\operatorname{Tail}(u-(u)_{B_R(x)};x,R)\right],$$
while if $p>2$, we have
$${\bf M}(Du)(x,R)
   \leq C{\bf{W}}^{\mu}_{\frac{1}{p},p}(x,R)
   +C\left[\left(\fint_{B_{R}(x)} |Du|^{p-1} \operatorname{d}\! \xi\right)^{\frac{1}{p-1}}+R^{-1-\sigma}\operatorname{Tail}(u-(u)_{B_R(x)};x,R)\right].$$
We mention that our result extends the mixed equation discussed in \cite{CSYZ24} to a vector field $\mathcal{A}$ with coefficients, and also elevates the pointwise potential estimate \eqref{mixPoten-grad} to a pointwise estimate for the truncated Hardy–Littlewood maximal function of the gradient.
\end{remark}

The remainder of this paper is organized as follows. In Section \ref{section2},
we collect some auxiliary estimates, along with known regularity results for the homogeneous equations. Section \ref{section3} is devoted to establishing the comparison estimates between \eqref{eq1} and the corresponding desired problems, and presents the direct applications required to prove our main results.
In section \ref{section4}, we complete the proof of Theorem \ref{Maxmeasure} - Theorem \ref{MaxDiniholder}.
Finally, we prove Theorem \ref{OscPotenSolMea}, Corollary \ref{PotenmaxMea}, and Theorem \ref{OscPotenSolBMO} - Theorem \ref{OscPotenGrad} in the last section.

\section{Preliminaries}\label{section2}

In this section, we first recall several useful auxiliary estimates,
and subsequently collect a series of established regularity results for the homogeneous equations, which are pivotal for proving our main theorems.
In what follows, $C$ denotes a constant whose value may vary from line to line, and only the relevant dependencies are specified in parentheses.

\subsection{Auxiliary results}

Let us begin with the following embedding inequality, which is proven by the Sobolev embedding $W^{1,q}\hookrightarrow W^{s,q}$ and Poincar\'{e}'s inequality.
\begin{lemma}\label{embed}{\rm(cf. \cite[Lemma 2.2]{DM24})}
Let $q\in[1,+\infty)$, $s\in(0,1)$ and $B_{r}(x_0)$ be a ball in $\R^n$. If $f\in W_{0}^{1,q}(B_{r}(x_0))$, then $f\in W^{s,q}(B_{r}(x_0))$ and satisfies
\begin{equation*}
 \left( \int_{B_{r}(x_0)} \fint_{B_{r}(x_0)} \frac{|f(x)-f(y)|^{q} }{|x-y|^{n+sq}} \operatorname{d}\!x \operatorname{d}\!y\right)^{\frac{1}{q}}\leq Cr^{1-s}\left( \fint_{B_{r}(x_0)}|Df|^q\operatorname{d}\!x\right)^{\frac{1}{q}},
\end{equation*}
where the positive constant $C$ depending only on $n,\, q$ and $s$.
\end{lemma}

Moreover, the following iteration lemma is crucial for  proving the local Lipschitz regularity of solutions to the homogeneous mixed local and nonlocal equation, as well as the Caccioppoli type inequality of solutions to the equation \eqref{eq1}.

\begin{lemma}\label{iterate} {\rm(cf. \cite[Lemma 6.1]{Giu})}
Let $f$ be a bounded nonnegative function defined in the interval $[r,R]$.
Assume that for $ r\leq \rho_1<\rho_2\leq R$ we have
\begin{equation*}
  f(\rho_1)\leq \vartheta f(\rho_2)+\frac{c_1}{(\rho_2-\rho_1)^{\alpha}}+\frac{c_2}{(\rho_2-\rho_1)^{\beta}}+c_3,
\end{equation*}
where $c_1,\,c_2,\,c_3\geq 0$, $\alpha>\beta>0$, and $0\leq\vartheta<1$. Then there exists a positive constant $C\equiv C(\alpha,\vartheta)$
such that
\begin{equation*}
  f(r)\leq C\left[\frac{c_1}{(R-r)^{\alpha}}+\frac{c_2}{(R-r)^{\beta}}+c_3\right].
\end{equation*}
\end{lemma}

The subsequent result establishes an immediate relationship between fractional maximal functions and potential functions based on their respective definitions.

\begin{lemma}\label{maxf-potential}{\rm(\cite[Lemma 4.1]{KuMi})}
Let $\mu\in\mathcal M (\mathbb R^{n})$, $\gamma\in(0,1)$, $\beta\geq 0$ and $p>1$. Then there holds that
\begin{equation*}\label{m-wp1}
  \left[{\bf M}_{\beta}(\mu)(x,\gamma R)\right]^{\frac{1}{p-1}}\leq \frac{\max\{\gamma^{\frac{\beta-n}{p-1}},1\}}{(-\log\gamma)|B_1|^{\frac{1}{p-1}}}{\bf W}^{\mu}_{\frac{\beta}{p},p}(x, R),
\end{equation*}
and
\begin{equation*}\label{m-wp2}
  {\bf M}_{\beta}(\mu)(x,\gamma R)\leq \frac{\max\{\gamma^{\beta-n},1\}}{(-\log\gamma)|B_1|}{\bf I}^{\mu}_{\beta}(x, R).
\end{equation*}
\end{lemma}

\subsection{Regularity for homogeneous equations}

We first consider $v\in W^{1,p}_{\rm loc}(\Omega)\cap {\mathcal L}^{p-1}_{sp}(\mathbb{R}^n)$ being a weak solution to the homogeneous mixed local and nonlocal elliptic equation
\begin{equation}
    \label{eq:homo-mixed}
   -\operatorname{div}\left(\mathcal{A}(x, Dv)\right) +\mathcal{L}_\Phi v=0\qquad\text{in }\ \Omega\,.
\end{equation}
A combination of \cite[Lemma 3.3]{BS23}
with the definition of the tail term and H\"{o}lder's inequality yields the following local boundedness estimate of $v$.
\begin{lemma}\label{lem:bdd}
Let $v\in W^{1,p}_{\rm loc}(\Omega)\cap {\mathcal L}^{p-1}_{sp}(\mathbb{R}^n)$ be a weak solution to \eqref{eq:homo-mixed} under assumptions \eqref{growth1}--\eqref{ellip} with $p > 1$ and $s \in (0,1)$.
Then for any $B_{r}(x_0) \Subset \Omega$, $\kappa \in \R$ and $\gamma\in(0,1)$, there holds that
\begin{equation*}
  \sup_{B_{\gamma r}(x_0)} |v-\kappa| \leq \frac{C}{(1-\gamma)^{\frac{n}{p-1}}} \left[ \left(\fint_{B_{r}(x_0)} |v-\kappa|^{q_0} \operatorname{d}\!x\right)^{\frac{1}{q_0}} + \operatorname{Tail}(v-\kappa; x_0,r)\right],
\end{equation*}
where $q_0=\max\{1,p-1\}$ and the positive constant $C \equiv C(n,p,s,\Lambda,\operatorname{diam}(\Omega))$.
\end{lemma}
We need the following Caccioppoli type inequality for later use, by making a completely analogous modification to the proof in \cite[Lemma 3.1 and Lemma 3.4]{BS23} and combining the definition of the tail term with H\"{o}lder's inequality again.
\begin{lemma}\label{CCP-v}
Let $v\in W^{1,p}_{\rm loc}(\Omega)\cap {\mathcal L}^{p-1}_{sp}(\mathbb{R}^n)$ be a weak solution to \eqref{eq:homo-mixed} under assumptions \eqref{growth1}--\eqref{ellip} with $p > 1$, $s \in (0,1)$ and $q_0=\max\{1,p-1\}$.
Then for any $B_{\varrho r}(x_0) \Subset \Omega$, $\kappa \in \R$ and $\gamma\in [\frac{\varrho}{2},\varrho)$, there exist positive constants $\theta\equiv\theta(n,p)$
and $C \equiv C(n,p,s,\Lambda,\operatorname{diam}(\Omega))$ such that
\begin{eqnarray}\label{CCPv}
&&\left(\fint_{B_{\gamma r}(x_0)} |Dv|^{p} \operatorname{d}\!x \right)^{\frac{1}{p}}+\left(\int_{B_{\gamma r}(x_0)}\fint_{B_{\gamma r}(x_0)} |v(x)-v(y)|^{p}K(x,y) \operatorname{d}\!x\operatorname{d}\!y \right)^{\frac{1}{p}}\nonumber\\
&\leq& \frac{C}{(\varrho-\gamma)^\theta r}\left[\left(\fint_{B_{\varrho r}(x_0)} |v-\kappa|^{q_0} \operatorname{d}\!x\right)^{\frac{1}{q_0}} +\operatorname{Tail}(v-\kappa;x_0, \varrho r) \right].
\end{eqnarray}
\end{lemma}

The set of functions satisfying the Caccioppoli inequality is typically referred to as the De Giorgi class. Based on the De Giorgi–Nash–Moser theory, \cite[Theorem 5.1]{GK22} and \cite[Lemma 3.5]{BS23} proved the local H\"{o}lder continuity of solutions to the homogeneous mixed local and nonlocal equation \eqref{eq:homo-mixed}. Here we present a slightly modified form through the further application of the definition of the tail term and H\"{o}lder's inequality.
\begin{lemma}\label{v-holder}
Let $v\in W^{1,p}_{\rm loc}(\Omega)\cap {\mathcal L}^{p-1}_{sp}(\mathbb{R}^n)$ be a weak solution to \eqref{eq:homo-mixed} under assumptions \eqref{growth1}--\eqref{ellip} with $p > 1$, $s \in (0,1)$ and $q_0=\max\{1,p-1\}$.
Then $v$ is locally H\"{o}lder continuous in $\Omega$, and for any $k \in \R$ and $B_{\rho}(x_0)\subset B_{\frac{r}{2}}(x_0)$ with $B_{r}(x_0) \Subset \Omega$,
there exists a constant $\alpha_m \in (0,1)$ depending only on $n,\,p,\,s,$ and $\Lambda$ such that
\begin{equation}\label{mixv-holder}
\underset{B_\rho(x_0)}{\operatorname{osc}} v \leq C \left( \frac{\rho}{r} \right)^{\alpha_m} \left[ \left(\fint_{B_{r}(x_0)} |v- \kappa|^{q_0} \operatorname{d}\! x\right)^{\frac{1}{q_0}} + \operatorname{Tail}(v-\kappa;x_0, r) \right]
\end{equation}
where the positive constant $C \equiv C(n,p,s,\Lambda,\operatorname{diam}(\Omega))$.
\end{lemma}

Building upon the Caccioppoli inequality \eqref{CCPv} and the H\"{o}lder continuity result \eqref{mixv-holder}, we derive the following Morrey type growth estimate for $Dv$.
\begin{lemma}\label{Dv-morrey}
Let $v\in W^{1,p}_{\rm loc}(\Omega)\cap {\mathcal L}^{p-1}_{sp}(\mathbb{R}^n)$ be a weak solution to \eqref{eq:homo-mixed} under assumptions \eqref{growth1}--\eqref{ellip} with $p > 1$, $s \in (0,1)$ and $q_0=\max\{1,p-1\}$.
Then there exists a positive constant $C$ depending only on $n,\,p,\,s,\,\Lambda$ and $\operatorname{diam}(\Omega)$ such that
\begin{equation}\label{morreyDv}
 \left(\fint_{B_{\rho}(x_0)} |Dv|^{q_0} \operatorname{d}\! x\right)^{\frac{1}{q_0}}\leq C \left( \frac{\rho}{r} \right)^{-1+\alpha_m} \left[ \left(\fint_{B_{r}(x_0)} |Dv|^{q_0} \operatorname{d}\! x\right)^{\frac{1}{q_0}} + \frac{1}{r}\operatorname{Tail}(v-(v)_{B_r(x_0)};x_0,r) \right]
\end{equation}
holds whenever $B_{\rho}(x_0) \subset B_{r}(x_0) \Subset \Omega$.
\end{lemma}
\begin{proof}
Without loss of generality, we may assume that $0<\rho<\frac{r}{4}$, otherwise \eqref{morreyDv} holds trivially. A combination of Lemma \ref{CCP-v} with $\kappa=(v)_{B_{2\rho}(x_0)}$, Lemma \ref{v-holder} with $\kappa=(v)_{B_{r}(x_0)}$, H\"{o}lder's inequality and Poincar\'{e}'s inequality yields that
\begin{eqnarray}\label{morreyDv1}
  &&\left(\fint_{B_{\rho}(x_0)} |Dv|^{q_0} \operatorname{d}\! x\right)^{\frac{1}{q_0}} \nonumber \\
   &\leq& \frac{C}{\rho}\left[\left(\fint_{B_{2\rho}(x_0)} |v-(v)_{B_{2\rho}(x_0)}|^{q_0} \operatorname{d}\!x\right)^{\frac{1}{q_0}} +\operatorname{Tail}(v-(v)_{B_{2\rho}(x_0)};x_0, 2\rho) \right]\nonumber\\
   &\leq&\frac{C}{\rho}\left\{\left( \frac{\rho}{r} \right)^{\alpha_m} \left[ \left(\fint_{B_{r}(x_0)} |v- (v)_{B_{r}(x_0)}|^{q_0} \operatorname{d}\! x\right)^{\frac{1}{q_0}} + \operatorname{Tail}(v-(v)_{B_{r}(x_0)};x_0, r) \right] \right.\nonumber\\
   &&\left. +\operatorname{Tail}(v-(v)_{B_{2\rho}(x_0)};x_0, 2\rho) \right\}\nonumber\\
   &\leq&\frac{C}{\rho}\left\{\left( \frac{\rho}{r} \right)^{\alpha_m} \left[ r \left(\fint_{B_{r}(x_0)} |Dv|^{q_0} \operatorname{d}\! x\right)^{\frac{1}{q_0}} + \operatorname{Tail}(v-(v)_{B_{r}(x_0)};x_0, r) \right]  +\operatorname{Tail}(v-(v)_{B_{2\rho}(x_0)};x_0, 2\rho) \right\}.\nonumber\\
\end{eqnarray}
To unify the two tail terms in \eqref{morreyDv1}, we apply \cite[Lemma 3.4]{CSYZ24} to deduce that
\begin{eqnarray}\label{tailesti1}
 \operatorname{Tail}(v-(v)_{B_{2\rho}(x_0)};x_0, 2\rho)
  \leq C \left( \frac{\rho}{r} \right)^{\alpha_m}\left[ r \left(\fint_{B_{r}(x_0)} |Dv|^{q_0} \operatorname{d}\! x\right)^{\frac{1}{q_0}} + \operatorname{Tail}(v-(v)_{B_{r}(x_0)};x_0, r) \right].
\end{eqnarray}
Hence, inserting \eqref{tailesti1} into \eqref{morreyDv1}, we have
\begin{equation*}
 \left(\fint_{B_{\rho}(x_0)} |Dv|^{q_0} \operatorname{d}\! x\right)^{\frac{1}{q_0}}\leq C \left( \frac{\rho}{r} \right)^{-1+\alpha_m} \left[ \left(\fint_{B_{r}(x_0)} |Dv|^{q_0} \operatorname{d}\! x\right)^{\frac{1}{q_0}} + \frac{1}{r}\operatorname{Tail}(v-(v)_{B_r(x_0)};x_0,r) \right],
\end{equation*}
where the positive constant $C\equiv C(n,p,s,\Lambda,\operatorname{diam}(\Omega))$.
\end{proof}

Based on the improved integrability result for $Dv$ from \cite[Proposition 3.4]{BKL24}, we now present an adapted version in conjunction with Lemma \ref{CCP-v} via Poincar\'{e}'s inequality.
\begin{lemma}\label{Dv-highinteglem}
Let $v\in W^{1,p}_{\rm loc}(\Omega)\cap {\mathcal L}^{p-1}_{sp}(\mathbb{R}^n)$ be a weak solution to \eqref{eq:homo-mixed} under assumptions \eqref{growth1}--\eqref{ellip} with $p > 1$, $s \in (0,1)$ and $q_0=\max\{1,p-1\}$. Then there exists an exponent $q>p$ depending only on $n,\,p,\,s$ and $\Lambda$ such that
\begin{equation}\label{Dv-highintegest}
 \left(\fint_{B_{\frac{r}{2}}(x_0)} |Dv|^{q} \operatorname{d}\! x\right)^{\frac{1}{q}}\leq C  \left(\fint_{B_{r}(x_0)} |Dv|^{q_0} \operatorname{d}\! x\right)^{\frac{1}{q_0}} + \frac{C}{r}\operatorname{Tail}(v-(v)_{B_r(x_0)};x_0,r)
\end{equation}
holds whenever $ B_{r}(x_0) \Subset \Omega$, where the positive constant $C \equiv C(n,p,s,\Lambda,\operatorname{diam}(\Omega))$.
\end{lemma}

In fact, the exponent $q$ in Lemma \ref{Dv-highinteglem}
is only slightly larger than $p$. To obtain interior gradient estimates for the case of arbitrary $q\in(p,+\infty)$, the vector field $\mathcal A$ with respect to $x$ cannot merely be measurable, it requires an additional regularity assumption, refer to \cite[Theorem 1.3]{BKL24}. Here, we present the version that will be used in the subsequent sections.
\begin{lemma}\label{Dv-intergradientlem}
Let $v\in W^{1,p}_{\rm loc}(\Omega)\cap {\mathcal L}^{p-1}_{sp}(\mathbb{R}^n)$ be a weak solution to \eqref{eq:homo-mixed} under assumptions \eqref{growth1}--\eqref{ellip} with $p > 1$, $s \in (0,1)$ and $p<q<+\infty$. If there exists a positive constant $\delta$ depending only on $n,\,p,\,s,\,\Lambda$ and $q$ such that $\omega(r)$ satisfies small $\operatorname{BMO}$ regular condition \eqref{smallBMO}, then there holds that
\begin{equation}\label{Dv-intergradientest}
 \left(\fint_{B_{\frac{r}{2}}(x_0)} |Dv|^{q} \operatorname{d}\! x\right)^{\frac{1}{q}}\leq C  \left(\fint_{B_{r}(x_0)} |Dv|^{q_0} \operatorname{d}\! x\right)^{\frac{1}{q_0}} + \frac{C}{r}\operatorname{Tail}(v-(v)_{B_r(x_0)};x_0,r),
\end{equation}
where the constant $C=C(n,p,s,\Lambda,q)$, the ball $ B_{r}(x_0) \Subset \Omega$ and $q_0=\max\{1,p-1\}$.
\end{lemma}

In the sequel, we consider a weak solution $w\in W^{1,p}_{\rm loc}(\Omega)$ to the homogeneous local equation
\begin{equation}\label{eq:homo-local}
   -\operatorname{div}\tilde{\mathcal{A}}(Dw)=0\qquad\text{in }\ \Omega\,,
\end{equation}
where the vector field $\tilde{\mathcal{A}}(Dw):=\left(\mathcal{A}(\cdot,z)\right)_{B_{r}(x_0)}$ evidently satisfies the growth and ellipticity conditions given by \eqref{growth1} and \eqref{ellipticity}.
Under these assumptions, it can be shown that there exists a H\"{o}lder regularity exponent $\alpha_M\in(0,1)$ depending only on $n,\,p$ and $\Lambda$ such that $Dw \in C_{\rm loc}^{0,\alpha_M}(\Omega,\R^n)$ and fulfills the following decay estimates for a suitable excess functional of $Dw$, refer to \cite{L06, CSYZ24} for details.

\begin{lemma}\label{exdecay-Dw-Lem}
    Let $w \in W^{1,p}_{\rm loc}(\Omega)$ be a weak solution to \eqref{eq:homo-local} under assumptions \eqref{growth1} and \eqref{ellipticity} with $p>1$ and $q_0=\max\{1,p-1\}$.
    Then there exist an exponent $\alpha_M\in(0,1)$ and a positive constant $C$, both depending only on $n,\,p$ and $\Lambda$ such that
    \begin{equation}\label{exdecay-Dw}
       \left( \fint_{B_{\rho}(x_0)}|Dw-(Dw)_{B_{\rho}(x_0)}|^{q_0}  \operatorname{d}\!x\right)^{\frac{1}{q_0}} \leq C\left(\frac{\rho}{r}\right)^{\alpha_M}\left(\fint_{B_{r}(x_0)}|Dw-(Dw)_{B_{r}}|^{q_0}  \operatorname{d}\!x\right)^{\frac{1}{q_0}}
    \end{equation}
holds whenever $x_0\in \Omega$ and $B_{\rho}(x_0) \subset B_{r}(x_0) \Subset \Omega$. Moreover, there holds that
   \begin{equation}\label{decay-Dw}
        \left(\fint_{B_{\rho}(x_0)}|Dw|^{q_0}  \operatorname{d}\!x\right)^{\frac{1}{q_0}} \leq C\left(\fint_{B_{r}(x_0)}|Dw|^{q_0}  \operatorname{d}\!x\right)^{\frac{1}{q_0}}\,,
    \end{equation}
where the positive constant $C$ also depends only on $n,\,p$ and $\Lambda$.
\end{lemma}

We conclude this subsection with a well-known global Calder\'{o}n-Zygmund estimate for the weak solution to the Dirichlet problem
\begin{equation}\label{Commodel-wv}
\left\{\begin{array}{r@{\ \ }c@{\ \ }ll}
\operatorname{div}\tilde{\mathcal{A}}(Dw)&=&0 & \mbox{in}\ \ B_r(x_0)\,, \\[0.05cm]
w&=& v & \mbox{on}\ \  \partial B_r(x_0)\,.
\end{array}\right.
\end{equation}

\begin{lemma}{\rm(\cite[Lemma 4.1]{Min10})}\label{CZest-Dw}
Let $w \in W^{1,p}(B_r(x_0))$ be a weak solution to \eqref{Commodel-wv} under assumptions \eqref{growth1} and \eqref{ellipticity} with $p>1$. Assume that the boundary datum $v\in W^{1,q}(B_r(x_0))$ for $q\geq p$, then $w\in W^{1,q}(B_r(x_0))$ and there holds that
\begin{equation}\label{CZ-Dw}
  \|Dw\|_{L^q(B_r(x_0))}\leq C  \|Dv\|_{L^q(B_r(x_0))}\,,
\end{equation}
where the positive constant $C$ depends only on $n,\,p,\,\Lambda$ and $q$.
\end{lemma}

\section{Comparison estimates and applications}\label{section3}

This section is devoted to comparing the gradient of solutions to the mixed local and nonlocal elliptic equation \eqref{eq1} with that of the homogeneous mixed equation \eqref{eq:homo-mixed} and the homogeneous local equation \eqref{eq:homo-local}, for which prior regularity results are available.
As a direct application of these comparison estimates, we also provide several key results that are necessary for proving our main theorems.

\subsection{Comparison between $Dv$ and $Dw$}

Let $v\in W^{1,p}_{\rm loc}(\Omega)\cap {\mathcal L}^{p-1}_{sp}(\mathbb{R}^n)$ be a weak solution to the homogeneous mixed local and nonlocal elliptic equation
 \eqref{eq:homo-mixed}. For a fixed ball $B_{r}(x_0)\Subset\Omega$, we define $w\in W^{1,p}(B_{r}(x_0))$ as a weak solution to the homogeneous local Dirichlet problem \eqref{Commodel-wv} and establish a comparison estimate between $Dv$ and $Dw$.
\begin{lemma}\label{comparison:Dv-Dw}
Let $v\in W^{1,p}_{\rm loc}(\Omega)\cap {\mathcal L}^{p-1}_{sp}(\mathbb{R}^n)$ and $w\in W^{1,p}(B_{r}(x_0))$ be weak solutions to \eqref{eq:homo-mixed} and \eqref{Commodel-wv} respectively, under assumptions \eqref{growth1}--\eqref{ellip} with $p > 1$ and $s \in (0,1)$. Let us define
\begin{equation}\label{exponent1}
\sigma_1 := \begin{cases}\frac{2(q-p)}{pq}\quad& \text{if } p\geq 2\,,\\
        \frac{q-p}{q}\quad &\text{if }1<p<2\,,
    \end{cases}\quad
  \sigma_2 := \begin{cases}\frac{1-s}{p-1}\quad& \text{if } p\geq 2\,, \\
        m\quad& \text{if } 1<p<2\,,
    \end{cases}\quad
     \sigma_3 := \begin{cases}\frac{1}{p-1}\quad& \text{if } p\geq 2\, ,\\
        \frac{1-m(2-p)}{p-1}\quad &\text{if }1<p<2\,,
    \end{cases}
\end{equation}
with $m\in(0,1-s)$ being an arbitrary constant and $q>p$.
Then there exists a positive constant $C\equiv C\big(n,p,s,\Lambda,\operatorname{diam}(\Omega)\big)$ such that
\begin{eqnarray}\label{comest-Dv-Dw}
\left(\fint_{B_{r}(x_0)} |Dv-Dw|^{q_0}\operatorname{d}\! x\right)^{\frac{1}{q_0}} &\leq& C\left(\big[\omega(r)\big]^{\sigma_1}+r^{\sigma_2}\right)\left(\fint_{B_{2r}(x_0)} |Dv|^{q_0}\operatorname{d}\! x\right)^{\frac{1}{q_0}}\nonumber\\
&&+\frac{C r^{\sigma_3}}{r^{p'}}\operatorname{Tail}(v-(v)_{B_{2r}(x_0)};x_0,2r),
	\end{eqnarray}
where $q_0=\max\{1,p-1\}$ and the averaged modulus of continuity $\omega(\cdot)$ is defined in Definition \ref{coef}.
\end{lemma}

\begin{remark}\label{q-range}
In this context, $q$ is used to represent the high integrability exponent of $Dv$.
Under the assumptions of Lemma \ref{comparison:Dv-Dw}, the exponent $q=q(n,p,s,\Lambda)$ is
defined by Lemma \ref{Dv-highinteglem}. If $w(r)$ satisfies the additional small $\operatorname{BMO}$ regular condition \eqref{smallBMO}, then the exponent $q$ can be chosen arbitrarily within the interval $(p,+\infty)$ via Lemma \ref{Dv-intergradientlem}.
\end{remark}

\begin{proof}
If the boundary condition in the Dirichlet problem \eqref{Commodel-wv} is extended to $\R^n\backslash B_r(x_0)$ by setting $w=v$, then $h:=v-w$ can be chosen as a test function for \eqref{eq:homo-mixed} and \eqref{Commodel-wv}, respectively. By virtue of the ellipticity condition \eqref{ellipticity}, we have
\begin{eqnarray}\label{com-vw1}
   && \fint_{B_{r}(x_0)}\left(|Dv|^2+|Dw|^2\right)^{\frac{p-2}{2}} |Dv-Dw|^{2}\operatorname{d}\! x\nonumber \\
  &\leq&C\fint_{B_{r}(x_0)} \left(\tilde{\mathcal{A}}(Dv)-\mathcal{A}(x,Dv)\right)\cdot Dh\operatorname{d}\! x -\frac{C}{|B_{r}|} \int_{\mathbb{R}^n} \int_{\mathbb{R}^n} \Phi(v(x)-v(y))(h(x)-h(y)) K(x,y) \operatorname{d}\!x \operatorname{d}\!y\nonumber\\
  &=:&I_1+I_2.
\end{eqnarray}
As indicated in \cite[(2.16)]{KuMi}, we can bound $I_1$ as shown below
\begin{equation*}
  I_1\leq \frac{1}{4}\fint_{B_{r}(x_0)}\left(|Dv|^2+|Dw|^2\right)^\frac{p-2}{2} |Dv-Dw|^2\operatorname{d}\! x+
C\fint_{B_{r}(x_0)}[A(Dv,B_r(x_0))(x)]^2 \left(|Dv|^2+\chi_{\{p<2\}}|Dw|^2\right)^{\frac{p}{2}}\operatorname{d}\! x\,,
\end{equation*}
where the coefficient function $A(Dv,B_r(x_0))$ is defined in \eqref{coefA}. It is easy to verify that
\begin{equation}\label{A-bdd}
  A(Dv,B_r(x_0))(x)\leq2 \Lambda
\end{equation}
by using the growth condition \eqref{growth1}.
To further control the last term in the estimate of $I_1$, we combine this fact with H\"{o}lder's inequality, Lemma \ref{CZest-Dw} and Lemma \ref{Dv-highinteglem} to arrive at
\begin{eqnarray}\label{com-vw1-1}
   && \fint_{B_{r}(x_0)}[A(Dv,B_r(x_0))]^2 \left(|Dv|^2+\chi_{\{p<2\}}|Dw|^2\right)^{\frac{p}{2}}\operatorname{d}\! x\nonumber \\
   &\leq& \left(\fint_{B_{r}(x_0)}[A(Dv,B_r(x_0))]^\frac{2q }{q-p} \operatorname{d}\! x\right)^{\frac{q-p}{q}}\left(\fint_{B_{r}(x_0)} \left(|Dv|^2+\chi_{\{p<2\}}|Dw|^2\right)^{\frac{q}{2}}\operatorname{d}\! x\right)^{\frac{p}{q}}\nonumber\\
   &\leq& C (2\Lambda)^{\frac{2p}{q}}\left[\omega(r)\right]^{\frac{2(q-p)}{q}}\left(\fint_{B_{r}(x_0)} |Dv|^q\operatorname{d}\! x\right)^{\frac{p}{q}}\nonumber\\
   &\leq&C\left[\omega(r)\right]^{\frac{2(q-p)}{q}}\left\{ \left( \fint_{B_{2r}(x_0)} |Dv|^{q_0} \operatorname{d}\! x\right)^{\frac{p}{q_0}}+ \left[\frac{1}{r}\operatorname{Tail}(v-(v)_{B_{2r}(x_0)};x_0,2r)\right]^p\right\}.
\end{eqnarray}
Hence, it follows that
\begin{eqnarray*}
   I_1&\leq& \frac{1}{4}\fint_{B_{r}(x_0)}\left(|Dv|^2+|Dw|^2\right)^\frac{p-2}{2} |Dv-Dw|^2\operatorname{d}\! x\\
   &&+ C\left[\omega(r)\right]^{\frac{2(q-p)}{q}}\left\{ \left( \fint_{B_{2r}(x_0)} |Dv|^{q_0} \operatorname{d}\! x\right)^{\frac{p}{q_0}}+ \left[\frac{1}{r}\operatorname{Tail}(v-(v)_{B_{2r}(x_0)};x_0,2r)\right]^p\right\}.
\end{eqnarray*}

Now we turn our attention to estimating $I_2$.
Since the support of the test function $h$ contains in $B_r(x_0)$ and the kernel $K$ is symmetric, we divide the integration area of $I_2$ into two subsets as
\begin{eqnarray*}
  I_2 &=&-C \int_{B_{\frac{3r}{2}}(x_0)} \fint_{B_{\frac{3r}{2}}(x_0)} |v(x)-v(y)|^{p-2}(v(x)-v(y))(h(x)-h(y)) K(x,y) \operatorname{d}\!x \operatorname{d}\!y \nonumber \\
   &&-2C \int_{\mathbb{R}^n\backslash B_{\frac{3r}{2}}(x_0)} \fint_{B_{r}(x_0)} |v(x)-v(y)|^{p-2}(v(x)-v(y))h(x) K(x,y) \operatorname{d}\!x \operatorname{d}\!y\nonumber\\
   &=:&I_{2,1}+I_{2,2}\,.
\end{eqnarray*}
With respect to the estimate of $I_{2,1}$, we apply the assumption \eqref{ellip} on $K$, H\"{o}lder's inequality and the Caccioppoli type inequality \eqref{CCPv} with $\varrho=2$, $\gamma=\frac{3}{2}$ and $\kappa=(v)_{B_{2r}(x_0)}$, together with Poincar\'{e}'s inequality and the embedding Lemma \ref{embed}, to derive that
\begin{eqnarray*}
  |I_{2,1}| &\leq &C\left( \int_{B_{\frac{3r}{2}}(x_0)} \fint_{B_{\frac{3r}{2}}(x_0)} \frac{|v(x)-v(y)|^{p} }{|x-y|^{n+sp}} \operatorname{d}\!x \operatorname{d}\!y\right)^{\frac{p-1}{p}}\left( \int_{B_{\frac{3r}{2}}(x_0)} \fint_{B_{\frac{3r}{2}}(x_0)} \frac{|h(x)-h(y)|^{p} }{|x-y|^{n+sp}} \operatorname{d}\!x \operatorname{d}\!y\right)^{\frac{1}{p}}  \\
   &\leq& Cr^{1-s}\left[\left(\fint_{B_{2r}(x_0)}|Dv|^{q_0}\operatorname{d}\!x\right)^{\frac{1}{q_0}}
   +\frac{1}{r} \operatorname{Tail}(v-(v)_{B_{2r}(x_0)};x_0, 2r) \right]^{p-1} \left(\fint_{B_{r}(x_0)}|Dv-Dw|^p\operatorname{d}\!x\right)^{\frac{1}{p}}.
\end{eqnarray*}
Combining Lemma \ref{v-holder} with Poincar\'{e}'s inequality, we
argue similarly as in the proof of \cite[Lemma 3.1]{CSYZ24} to deduce that
\begin{equation*}
  |I_{2,2}|\leq C\left[r^{(1-s)p'}\left(\fint_{B_{2r}(x_0)}|Dv|^{q_0}\operatorname{d}\!x\right)^{\frac{1}{q_0}}
   +\frac{1}{r} \operatorname{Tail}(v-(v)_{B_{2r}(x_0)};x_0, 2r) \right]^{p-1} \left(\fint_{B_{r}(x_0)}|Dv-Dw|^p\operatorname{d}\!x\right)^{\frac{1}{p}}.
\end{equation*}
Substituting the estimates of $I_1$ and $I_2$ into \eqref{com-vw1}, we obtain
\begin{eqnarray}\label{com-vw2}
 && \fint_{B_{r}(x_0)}\left(|Dv|^2+|Dw|^2\right)^{\frac{p-2}{2}} |Dv-Dw|^{2}\operatorname{d}\! x\nonumber  \\
   &\leq& C\left[r^{\frac{1-s}{p-1}}\left(\fint_{B_{2r}(x_0)}|Dv|^{q_0}\operatorname{d}\!x\right)^{\frac{1}{q_0}}
   +\frac{1}{r} \operatorname{Tail}(v-(v)_{B_{2r}(x_0)};x_0, 2r) \right]^{p-1} \left(\fint_{B_{r}(x_0)}|Dv-Dw|^p\operatorname{d}\!x\right)^{\frac{1}{p}}\nonumber  \\
 &&+
   C\left[\omega(r)\right]^{\frac{2(q-p)}{q}}\left\{  \left(\fint_{B_{2r}(x_0)} |Dv|^{q_0} \operatorname{d}\! x\right)^{\frac{p}{q_0}} + \left[\frac{1}{r}\operatorname{Tail}(v-(v)_{B_{2r}(x_0)};x_0,2r)\right]^p\right\}.
\end{eqnarray}

For the case of $p\geq2$, then a direct calculation using Young's inequality for \eqref{com-vw2} reveals that
\begin{eqnarray*}
   && \fint_{B_{r}(x_0)} |Dv-Dw|^{p}\operatorname{d}\! x\nonumber  \\
   &\leq&   \fint_{B_{r}(x_0)}\left(|Dv|^2+|Dw|^2\right)^{\frac{p-2}{2}} |Dv-Dw|^{2}\operatorname{d}\! x\nonumber  \\
    &\leq&
    C\left(\left[\omega(r)\right]^{\frac{2(q-p)}{q}}+r^{\frac{p(1-s)}{p-1}}\right)\left(\fint_{B_{2r}(x_0)}|Dv|^{q_0}\operatorname{d}\!x\right)^{\frac{p}{q_0}}+\left[\frac{1}{r} \operatorname{Tail}(v-(v)_{B_{2r}(x_0)};x_0, 2r) \right]^{p}.
\end{eqnarray*}
While if $1<p<2$, we utilize H\"{o}lder's inequality, the energy estimate \eqref{CZ-Dw} and Lemma \ref{Dv-highinteglem} with $q=p$, and Young's inequality to conclude that
\begin{eqnarray*}
  && \fint_{B_{r}(x_0)} |Dv-Dw|^{p}\operatorname{d}\! x\nonumber  \\
   &\leq&  C\left(\fint_{B_{r}(x_0)}\left(|Dv|^2+|Dw|^2\right)^{\frac{p-2}{2}} |Dv-Dw|^{2}\operatorname{d}\! x\right)^{\frac{p}{2}}\left(\fint_{B_{r}(x_0)} |Dv|^{p}\operatorname{d}\! x\right)^{\frac{2-p}{2}}\\
&\leq& C\left(\left[\omega(r)\right]^{\frac{p(q-p)}{q}}+r^{mp}\right)\left(\fint_{B_{2r}(x_0)} |Dv|^{q_0}\operatorname{d}\! x\right)^{\frac{p}{q_0}}+r^{-m(2-p)p'}\left[\frac{1}{r} \operatorname{Tail}(v-(v)_{B_{2r}(x_0)};x_0, 2r) \right]^{p}
\end{eqnarray*}
for any $m\in(0,1-s]$. Therefore, the proof of Lemma \ref{comparison:Dv-Dw} is completed by integrating the aforementioned two cases and applying H\"{o}lder's inequality once more.
\end{proof}

According to the proof of Lemma \ref{comparison:Dv-Dw}, if $Dv\in L^{\infty}_{\rm loc}(\Omega)$, we can re-estimate the previous display \eqref{com-vw1-1} as follows
\begin{equation*}
 \fint_{B_{r}(x_0)}[A(Dv,B_r(x_0))]^2 \left(|Dv|^2+\chi_{\{p<2\}}|Dw|^2\right)^{\frac{p}{2}}\operatorname{d}\! x
   \leq \left[\omega(r)\right]^{2}\|Dv\|_{L^\infty(B_r(x_0))}^p \,\,\mbox{for} \,\, p\geq 2,
\end{equation*}
while for the case of $p < 2$, we apply H\"{o}lder's inequality, \eqref{A-bdd} and Lemma \ref{CZest-Dw} to derive
\begin{eqnarray*}
 &&\fint_{B_{r}(x_0)}[A(Dv,B_r(x_0))]^2 \left(|Dv|^2+\chi_{\{p<2\}}|Dw|^2\right)^{\frac{p}{2}}\operatorname{d}\! x\\ &\leq& C\left(\fint_{B_{r}(x_0)}[A(Dv,B_r(x_0))]^2 \operatorname{d}\! x\right)^{\frac{p}{2}}\left(\fint_{B_{r}(x_0)} |Dv|^{\frac{2p}{2-p}}\operatorname{d}\! x\right)^{\frac{2-p}{2}} \\
   &\leq&C\left[\omega(r)\right]^{p}\|Dv\|_{L^\infty(B_r(x_0))}^p.
\end{eqnarray*}
Thus, the exponent $\sigma_1$ can be refined to
\begin{equation}\label{exponent2}
\sigma_1' := \begin{cases}\frac{2}{p}\quad& \text{if } p\geq 2\,,\\
       1 \quad &\text{if }1<p<2\,.
    \end{cases}
\end{equation}
The subsequent lemma strengthens the regularity assumption on the partial map $x\mapsto\mathcal{A}(x,\cdot)$, thereby guaranteeing the local Lipschitz continuity of solutions to the homogeneous mixed local and nonlocal elliptic equation \eqref{eq:homo-mixed}.

\begin{lemma}\label{Dv-L^infty}
Let $v\in W^{1,p}_{\rm loc}(\Omega)\cap {\mathcal L}^{p-1}_{sp}(\mathbb{R}^n)$ be a weak solution to \eqref{eq:homo-mixed} under assumptions \eqref{growth1}--\eqref{ellip} with $p > 1$, $s \in (0,1)$ and $\max\{sp'-1,0\}<a<\sigma_3$. Assume that $[\omega(r)]^{\sigma_1'}$ satisfies Dini-$\operatorname{VMO}$ regular condition \eqref{Dinivmo},
then $Dv\in L^{\infty}_{\rm loc}(\Omega)$, and there exist a positive constant $C$
depending only on $n$, $p$, $s$, $\Lambda$, $a$, $m$, $\operatorname{diam}(\Omega)$,
and a radius $\tilde{R}$ depending only on $n$, $p$, $s$, $\Lambda$, $m$, $\omega(\cdot)$, $\operatorname{diam}(\Omega)$, such that
\begin{equation}\label{v-Lip}
\|Dv\|_{L^\infty(B_r(x_0))}\leq C  \left(\fint_{B_{2r}(x_0)} |Dv|^{q_0} \operatorname{d}\! x\right)^{\frac{1}{q_0}} + \frac{Cr^{a}}{r^{p'}}\operatorname{Tail}(v-(v)_{B_{2r}(x_0)};x_0,2r)
\end{equation}
holds for the ball
$B_{2r}(x_0) \Subset \Omega$ with $r\in(0,\frac{\tilde{R}}{4}]$ and $q_0=\max\{1,p-1\}$.
\end{lemma}

\begin{proof}
Analogous to the standard approximation argument in the proof of \cite[Theorem 7.1]{KuMi}, we may assume a priori that $Dv\in C^{0}_{\rm loc}(\Omega)$ by applying \cite[Theorem 2.4]{BKL24}, \cite[Lemma 5.2]{BS23}, \cite[Lemma 4.1]{KuMinSi} and \cite[Theorem 5]{DM24}.
If $0<\rho<r$, then a combination of Lemma \ref{exdecay-Dw-Lem} and Lemma \ref{comparison:Dv-Dw} yields that
\begin{eqnarray}\label{v-Lip1}
   &&\left( \fint_{B_{\rho}(x_0)}|Dv-(Dv)_{B_{\rho}(x_0)}|^{q_0}  \operatorname{d}\!x\right)^{\frac{1}{q_0}}\nonumber\\
    &\leq& \left(\frac{r}{\rho}\right)^{\frac{n}{q_0}}\left( \fint_{B_{r}(x_0)}|Dv-Dw|^{q_0}  \operatorname{d}\!x\right)^{\frac{1}{q_0}} +\left( \fint_{B_{\rho}(x_0)}|Dw-(Dw)_{B_{\rho}(x_0)}|^{q_0}  \operatorname{d}\!x\right)^{\frac{1}{q_0}}\nonumber\\
    &\leq& C_0\left(\frac{\rho}{r}\right)^{\alpha_M}\left(\fint_{B_{r}(x_0)}|Dv-(Dv)_{B_{r}(x_0)}|^{q_0}  \operatorname{d}\!x\right)^{\frac{1}{q_0}}+
    C\left(\frac{r}{\rho}\right)^{\frac{n}{q_0}}\left(\big[\omega(r)\big]^{\sigma_1'}+r^{\sigma_2}\right)\|Dv\|_{L^\infty(B_{2r}(x_0))}\nonumber\\
    &&+C\left(\frac{r}{\rho}\right)^{\frac{n}{q_0}}\frac{ r^{\sigma_3}}{r^{p'}}\operatorname{Tail}(v-(v)_{B_{2r}(x_0)};x_0,2r)\,,
\end{eqnarray}
where the positive constant $C_0$ depends only on $n,\, p$ and $\Lambda$.
In the case of $r\leq\rho<2r$, the above inequality is trivial, thus, \eqref{v-Lip1} holds whenever $0<\rho<2r$. Let $B_i:=B_{r_i}(x_0)$ be a sequence of shrinking balls with the radius $r_i:=\frac{2r}{L^i}$, where $L$ is a large positive constant to be determined later. We denote
\begin{equation*}
  A_i:=\left( \fint_{B_i}|Dv-(Dv)_{B_i}|^{q_0}  \operatorname{d}\!x\right)^{\frac{1}{q_0}}\,\, \mbox{and}\,\, K_i:=|(Dv)_{B_i}|,
\end{equation*}
and choose $L\equiv L(n,p,\Lambda)$ large enough such that $C_0L^{-\alpha_M}\leq\frac{1}{2}$, then \eqref{v-Lip1} implies that
\begin{equation}\label{v-Lip2}
 A_{i+1}\leq\frac{1}{2}A_i+C\left(\big[\omega(r_i)\big]^{\sigma_1'}+r_i^{\sigma_2}\right)\|Dv\|_{L^\infty(B_{2r}(x_0))}
    +C\frac{ r_i^{\sigma_3}}{r_i^{p'}}\operatorname{Tail}(v-(v)_{B_i};x_0,r_i)\,.
\end{equation}
With respect to the estimate of the last tail term for $i\geq 1$, we utilize the local boundedness Lemma \ref{lem:bdd} with $\kappa=(v)_{B_0}$ and Poincar\'{e}'s inequality to deduce that
\begin{eqnarray*}
 &&\frac{1}{r_i^{p'}}\operatorname{Tail}(v-(v)_{B_i};x_0,r_i)\\
   &=& \left(\int_{\R^n\setminus B_i}\frac{|v(x)-(v)_{B_i}|^{p-1}}{|x-x_0|^{n+sp}} \operatorname{d}\!x\right)^{1/(p-1)}\\
   &\leq& \frac{1}{(2r)^{p'}}\operatorname{Tail}(v-(v)_{B_{2r}(x_0)};x_0,2r)+
   Cr^{-sp'}\sup_{B_{i}}|v(x)-(v)_{B_0}|+Cr^{1-sp'}\|Dv\|_{L^\infty(B_{2r}(x_0))}\\
   &\leq&\frac{C}{r^{p'}}\operatorname{Tail}(v-(v)_{B_{2r}(x_0)};x_0,2r)+Cr^{1-sp'}\|Dv\|_{L^\infty(B_{2r}(x_0))}\,.
\end{eqnarray*}
It is straightforward to show that $\sigma_3>\max\{sp'-1,0\}$ by virtue of $m\in(0,1-s)$ and $s\in(0,1)$. Let $\max\{sp'-1,0\}<a<\sigma_3$ and substitute the above estimate into \eqref{v-Lip2} to obtain
\begin{equation*}
  A_{i+1}\leq\frac{1}{2}A_i+C\left(\big[\omega(r_i)\big]^{\sigma_1'}+r_i^{\sigma_2}
  +r_i^{\sigma_3+1-sp'}\right)
  \|Dv\|_{L^\infty(B_{2r}(x_0))}
    +C r_i^{\sigma_3-a}\frac{ r^{a}}{r^{p'}}\operatorname{Tail}(v-(v)_{B_{2r}(x_0)};x_0,2r) \,.
\end{equation*}
Hence, it follows that
\begin{eqnarray}\label{v-Lip3}
  \sum_{i=1}^{j}A_{i} &\leq& \frac{1}{2}\sum_{i=0}^{j-1}A_i+C\|Dv\|_{L^\infty(B_{2r}(x_0))}\sum_{i=0}^{j-1}
  \left(\big[\omega(r_i)\big]^{\sigma_1'}+r_i^{\sigma_2}+r_i^{\sigma_3+1-sp'}\right)\nonumber\\
&&
    +C\frac{ r^{a}}{r^{p'}}\operatorname{Tail}(v-(v)_{B_{2r}(x_0)};x_0,2r) \sum_{i=0}^{j-1} r_i^{\sigma_3-a}\nonumber \\
   &\leq& A_0 +C\|Dv\|_{L^\infty(B_{2r}(x_0))}\sum_{i=0}^{j-1}\left(\big[\omega(r_i)\big]^{\sigma_1'}
   +r_i^{\sigma_2}+r_i^{\sigma_3+1-sp'}\right)\nonumber\\
&&
    +C\frac{ r^{a}}{r^{p'}}\operatorname{Tail}(v-(v)_{B_{2r}(x_0)};x_0,2r) \sum_{i=0}^{j-1} r_i^{\sigma_3-a}\,.
\end{eqnarray}
Applying H\"{o}lder's inequality, Poincar\'{e}'s inequality and \eqref{v-Lip3}, we derive
\begin{eqnarray}\label{v-Lip4}
  K_{j+1} &=& K_0+\sum_{i=0}^{j}(K_{i+1}-K_i)\leq K_0+CA_0+C\sum_{i=1}^{j}A_i\nonumber\\
   &\leq&C  \left(\fint_{B_{2r}(x_0)} |Dv|^{q_0} \operatorname{d}\! x\right)^{\frac{1}{q_0}}+C_1\|Dv\|_{L^\infty(B_{2r}(x_0))}\sum_{i=0}^{j-1}\left(\big[\omega(r_i)\big]^
   {\sigma_1'}+r_i^{\sigma_2}+r_i^{\sigma_3+1-sp'}\right)\nonumber\\
&&
    +C\frac{ r^{a}}{r^{p'}}\operatorname{Tail}(v-(v)_{B_{2r}(x_0)};x_0,2r) \sum_{i=0}^{j-1} r_i^{\sigma_3-a}\,.
\end{eqnarray}
Through a direct calculation, we have
\begin{equation*}
  \sum_{i=0}^{j-1}\left(\big[\omega(r_i)\big]^{\sigma_1'}+r_i^{\sigma_2}
  +r_i^{\sigma_3+1-sp'}\right)\leq C_2(n,p,\Lambda)\int_0^{4r}\left(\big[\omega(\rho)\big]^{\sigma_1'}+\rho^{\sigma_2}
  +\rho^{\sigma_3+1-sp'}\right)\frac{\operatorname{d}\! \rho}{\rho}.
\end{equation*}
Due to the Dini-$\operatorname{VMO}$ regularity of
$\big[\omega(r)\big]^{\sigma_1'}$ and $\sigma_2,\,\sigma_3+1-sp'>0$, we further restrict the radius
$\tilde{R}\equiv\tilde{R}(n,p,s,\Lambda,m,\omega(\cdot),\operatorname{diam}(\Omega))>0$
such that
\begin{equation}\label{v-Lip5}
  C_1C_2\int_0^{4r}\left(\big[\omega(\rho)\big]^{\sigma_1'}+\rho^{\sigma_2} +\rho^{\sigma_3+1-sp'}\right)\frac{\operatorname{d}\! \rho}{\rho}\leq  C_1C_2\int_0^{\tilde{R}}\left(\big[\omega(\rho)\big]^{\sigma_1'}+\rho^{\sigma_2} +\rho^{\sigma_3+1-sp'}\right)\frac{\operatorname{d}\! \rho}{\rho}\leq\frac{1}{2}
\end{equation}
for $r\leq\frac{\tilde{R}}{4}$. Similarly, $a<\sigma_3$ implies that $\displaystyle\sum_{i=0}^{j-1} r_i^{\sigma_3-a}\leq C$. Inserting \eqref{v-Lip5} into \eqref{v-Lip4}, we have
\begin{equation*}
  K_{j+1}
   \leq C  \left(\fint_{B_{2r}(x_0)} |Dv|^{q_0} \operatorname{d}\! x\right)^{\frac{1}{q_0}}+\frac{1}{2}\|Dv\|_{L^\infty(B_{2r}(x_0))}
    +C\frac{ r^{a}}{r^{p'}}\operatorname{Tail}(v-(v)_{B_{2r}(x_0)};x_0,2r).
\end{equation*}
Therefore, letting $j\rightarrow +\infty$, we conclude that
\begin{equation}\label{v-Lip6}
 |Dv(x_0)| \leq C  \left(\fint_{B_{2r}(x_0)} |Dv|^{q_0} \operatorname{d}\! x\right)^{\frac{1}{q_0}}+\frac{1}{2}\|Dv\|_{L^\infty(B_{2r}(x_0))}
    +C\frac{ r^{a}}{r^{p'}}\operatorname{Tail}(v-(v)_{B_{2r}(x_0)};x_0,2r).
\end{equation}

In the sequel, we consider concentric balls $B_r(x_0)\subset B_{\rho_1}(x_0)\subset B_{\rho_2}(x_0)\subset B_{2r}(x_0)$, which implies that $\rho_2-\rho_1\in[0,r]$. For any $x\in B_{\rho_1}(x_0)$, applying \eqref{v-Lip6} to the ball $B_{\rho_2-\rho_1}(x)$, we have
\begin{eqnarray*}
 |Dv(x)| &\leq& C  \left(\fint_{B_{\rho_2-\rho_1}(x)} |Dv|^{q_0} \operatorname{d}\! x\right)^{\frac{1}{q_0}}+\frac{1}{2}\|Dv\|_{L^\infty(B_{\rho_2-\rho_1}(x))}\nonumber\\
 &&
    +C\frac{ r^{a}}{(\rho_2-\rho_1)^{p'}}\operatorname{Tail}(v-(v)_{B_{\rho_2-\rho_1}(x)};x,\rho_2-\rho_1).
\end{eqnarray*}
Note that $B_{\rho_2-\rho_1}(x)\subset B_{\rho_2}(x_0)$ and $|y-x|>\frac{\rho_2-\rho_1}{\rho_2}|y-x_0|$ for $x\in B_{\rho_1}(x_0)$ and $y\in \R^n\setminus B_{\rho_2-\rho_1}(x)$. Combining these facts with H\"{o}lder's inequality and Poincar\'{e}'s inequality leads to
\begin{eqnarray}\label{v-Lip7}
 |Dv(x)| &\leq& \frac{1}{2}\|Dv\|_{L^\infty(B_{\rho_2}(x_0))}+\frac{C}{(\rho_2-\rho_1)^{\frac{n}{q_0}}}  \left(\int_{B_{2r}(x_0)} |Dv|^{q_0} \operatorname{d}\! x\right)^{\frac{1}{q_0}}\nonumber\\
 &&+\frac{Cr^{a+1+np'-\frac{n}{q_0}}}{(\rho_2-\rho_1)^{(n+s)p'}}  \left( \int_{B_{2r}(x_0)} |Dv|^{q_0}\operatorname{d}\! x\right)^{\frac{1}{q_0}}\nonumber\\
 &&+\frac{C}{(\rho_2-\rho_1)^{\frac{n+sp}{p-1}}}\frac{ r^{a+\frac{n+sp}{p-1}}}{r^{p'}}\operatorname{Tail}(v-(v)_{B_{2r}(x_0)};x_0,2r)\nonumber\\
 &\leq& \frac{1}{2}\|Dv\|_{L^\infty(B_{\rho_2}(x_0))}+\frac{C\left(r^{a
 +1+np'}+r^{(n+s)p'}\right)}{(\rho_2-\rho_1)^{(n+s)p'+\frac{n}{q_0}}}  \left( \int_{B_{2r}(x_0)} |Dv|^{q_0}\operatorname{d}\! x\right)^{\frac{1}{q_0}}\nonumber\\
 &&+\frac{C}{(\rho_2-\rho_1)^{\frac{n+sp}{p-1}}}\frac{ r^{a+\frac{n+sp}{p-1}}}{r^{p'}}\operatorname{Tail}(v-(v)_{B_{2r}(x_0)};x_0,2r)
\end{eqnarray}
for any $x\in B_{\rho_1}(x_0)$.
Finally, taking the essential supremum over $B_{\rho_1}(x_0)$ in \eqref{v-Lip7}, and using the iteration Lemma \ref{iterate} and $a>\max\{sp'-1,0\}$, we deduce that \eqref{v-Lip} is valid. Thus, the proof of Lemma \ref{Dv-L^infty} is complete.
\end{proof}

Consequently, if the coefficient satisfies Dini-$\operatorname{VMO}$ regularity,
then the following  modified comparison estimate between $Dv$ and $Dw$ can be obtained by combining Lemmas \ref{comparison:Dv-Dw} and \ref{Dv-L^infty}.

\begin{lemma}\label{comparison:Dv-Dw-Dinivmo}
Let $v\in W^{1,p}_{\rm loc}(\Omega)\cap {\mathcal L}^{p-1}_{sp}(\mathbb{R}^n)$ and $w\in W^{1,p}(B_{r}(x_0))$ be weak solutions to \eqref{eq:homo-mixed} and \eqref{Commodel-wv} respectively, under assumptions \eqref{growth1}--\eqref{ellip} with $p > 1$, $s \in (0,1)$ and $\max\{sp'-1,0\}<a<\sigma_3$. Assume that $[\omega(r)]^{\sigma_1'}$ satisfies Dini-$\operatorname{VMO}$ regular condition \eqref{Dinivmo}, then there exist a positive constant $C$
depending only on $n$, $p$, $s$, $\Lambda$, $m$, $a$, $\operatorname{diam}(\Omega)$,
and a radius $\tilde{R}$ depending only on $n$, $p$, $s$, $\Lambda$, $m$, $\omega(\cdot)$, $\operatorname{diam}(\Omega)$, such that
\begin{eqnarray}\label{comest-Dv-Dw-Dinivmo}
\left(\fint_{B_{r}(x_0)} |Dv-Dw|^{q_0}\operatorname{d}\! x\right)^{\frac{1}{q_0}} &\leq& C\left(\big[\omega(r)\big]^{\sigma_1'}+r^{\sigma_2}\right)\left(\fint_{B_{2r}(x_0)} |Dv|^{q_0}\operatorname{d}\! x\right)^{\frac{1}{q_0}}\nonumber\\
&&+C\frac{\left([\omega(r)\big]^{\sigma_1'}r^{a} + r^{\sigma_3}\right)}{r^{p'}}\operatorname{Tail}(v-(v)_{B_{2r}(x_0)};x_0,2r),\nonumber\\
	\end{eqnarray}
holds for the ball
$B_{2r}(x_0) \Subset \Omega$ with $r\in(0,\frac{\tilde{R}}{4}]$ and $q_0=\max\{1,p-1\}$, where the exponents $\sigma_1',\,\sigma_2$ and $\sigma_3$ have been defined in \eqref{exponent2} and \eqref{exponent1}.
\end{lemma}

\subsection{Comparison between $Du$ and $Dv$}

For a fixed ball $B_{2r}(x_0)\Subset\Omega$, we define $v\in W^{1,p}(B_{2r}(x_0))$ to be a weak solution of the homogeneous mixed local and nonlocal Dirichlet problem
\begin{equation}\label{Commodel-vu}
\left\{\begin{array}{r@{\ \ }c@{\ \ }ll}
-\operatorname{div}\left(\mathcal{A}(x, Dv)\right) +\mathcal{L}_\Phi v&=&0& \mbox{in}\ \ B_{2r}(x_0)\,, \\[0.05cm]
v&=& u & \mbox{on}\ \  \R^n\backslash B_{2r}(x_0)\,,
\end{array}\right.
\end{equation}
where the boundary datum $u\in W^{1,p}_{\rm loc}(\Omega)\cap {\mathcal L}^{p-1}_{sp}(\mathbb{R}^n)$ is a weak solution to the nonhomogeneous mixed equation \eqref{eq1}. We proceed by invoking a comparison estimate for $Du$ and $Dv$, as established in \cite[Lemma 4.2]{BS23}.

\begin{lemma}\label{comparison:Du-Dv}
Let $u\in W^{1,p}_{\rm loc}(\Omega)\cap {\mathcal L}^{p-1}_{sp}(\mathbb{R}^n)$ and $v\in W^{1,p}(B_{2r}(x_0))$ be weak solutions to \eqref{eq1} and \eqref{Commodel-vu} respectively, under assumptions \eqref{growth1}--\eqref{ellip} with $p > 2-\frac{1}{n}$, $s \in (0,1)$ and $q_0=\max\{1,p-1\}$.
Then there holds that
\begin{equation}\label{comest-Du-Dv}
	\left(\fint_{B_{2r}(x_0)} |Du-Dv|^{q_0}\operatorname{d}\! x\right)^{\frac{1}{q_0}} \leq C\left[\frac{|\mu|(B_{2r}(x_0))}{r^{n-1}}\right]^{\frac{1}{p-1}}+C\chi_{\{p<2\}}
\left[\frac{|\mu|(B_{2r}(x_0))}{r^{n-1}}\right]
\left(\fint_{B_{2r}(x_0)} |Du|\operatorname{d}\! x\right)^{2-p},
	\end{equation}
where the positive constant $C$ depends only on $n,\,p,\,s,\,\Lambda$ and $\operatorname{diam}(\Omega)$.
\end{lemma}

By incorporating  the comparison estimate \eqref{comest-Du-Dv} with the Caccioppoli type inequality for $v$ established in Lemma \ref{CCP-v}, we can derive the Caccioppoli type inequality for $u$.
\begin{lemma}\label{CCP-u}
Let $u\in W^{1,p}_{\rm loc}(\Omega)\cap {\mathcal L}^{p-1}_{sp}(\mathbb{R}^n)$ be a weak solution to \eqref{eq1} under assumptions \eqref{growth1}--\eqref{ellip} with $p > 2-\frac{1}{n}$, $s \in (0,1)$ and $q_0=\max\{1,p-1\}$.
Then there exists a positive constant $C = C(n,p,s,\Lambda,\operatorname{diam}(\Omega))$ such that
\begin{eqnarray}\label{CCPu}
\left(\fint_{B_{ \frac{r}{2}}(x_0)} |Du|^{q_0} \operatorname{d}\!x \right)^{\frac{1}{q_0}}
&\leq& \frac{C}{r}\left[\left(\fint_{B_{ r}(x_0)} |u-(u)_{B_r(x_0)}|^{q_0} \operatorname{d}\!x\right)^{\frac{1}{q_0}} +\operatorname{Tail}(u-(u)_{B_r(x_0)};x_0,r) \right]\nonumber\\
&&+C\left[\frac{|\mu|(B_{r}(x_0))}{r^{n-1}}\right]^{\frac{1}{p-1}},
\end{eqnarray}
holds whenever $B_r(x_0)\Subset\Omega$.
\end{lemma}

\begin{proof}
Without loss of generality, we may assume that $(u)_{B_r(x_0)}=0$, since $u-(u)_{B_r(x_0)}$ remains a weak solution to the equation \eqref{eq1}. Let
$\frac{1}{2}\leq \gamma <\varrho<1$ and $v$ be a weak solution to the same Dirichlet problem \eqref{Commodel-vu} in the ball $B_{\varrho r}(x_0)$ instead of $B_{2r}(x_0)$.
A combination of Lemma \ref{CCP-v} with $\kappa=0$, Lemma \ref{comparison:Du-Dv}, Poincar\'{e}'s inequality, and Young's inequality for $p<2$ yields that
 \begin{eqnarray*}
&&\left(\fint_{B_{ \gamma r}(x_0)} |Du|^{q_0} \operatorname{d}\!x \right)^{\frac{1}{q_0}}\\
   &\leq& C\left(\fint_{B_{ \varrho r}(x_0)} |Du-Dv|^{q_0} \operatorname{d}\!x \right)^{\frac{1}{q_0}} +\left(\fint_{B_{ \gamma r}(x_0)} |Dv|^{q_0} \operatorname{d}\!x \right)^{\frac{1}{q_0}},  \\
    &\leq&\frac{C}{(\varrho-\gamma)^\theta}\left(\fint_{B_{ \varrho r}(x_0)} |Du-Dv|^{q_0} \operatorname{d}\!x \right)^{\frac{1}{q_0}}+\frac{C}{(\varrho-\gamma)^\theta r} \left[\left(\fint_{B_{\varrho r}(x_0)} |u|^{q_0} \operatorname{d}\!x\right)^{\frac{1}{q_0}} +\operatorname{Tail}(u;x_0, \varrho r) \right], \\
    &\leq& \frac{C}{(\varrho-\gamma)^\theta} \left\{\left[\frac{|\mu|(B_{ r}(x_0))}{r^{n-1}}\right]^{\frac{1}{p-1}}+\chi_{\{p<2\}}
\left[\frac{|\mu|(B_{ r}(x_0))}{r^{n-1}}\right]
\left(\fint_{B_{\varrho r}(x_0)} |Du|\operatorname{d}\! x\right)^{2-p}\right\}\\
&&+\frac{C}{(\varrho-\gamma)^\theta}\frac{1}{r} \left[\left(\fint_{B_{ r}(x_0)} |u|^{q_0} \operatorname{d}\!x\right)^{\frac{1}{q_0}} +\operatorname{Tail}(u;x_0, r) \right], \\
&\leq&\frac{1}{2}\left(\fint_{B_{ \varrho r}(x_0)} |Du|^{q_0} \operatorname{d}\!x \right)^{\frac{1}{q_0}} + \frac{C}{(\varrho-\gamma)^{\theta\max\{1,\frac{1}{p-1}\}}}\left[\frac{|\mu|(B_{ r}(x_0))}{r^{n-1}}\right]^{\frac{1}{p-1}}\\
&&+\frac{C}{(\varrho-\gamma)^\theta}\frac{1}{r} \left[\left(\fint_{B_{ r}(x_0)} |u|^{q_0} \operatorname{d}\!x\right)^{\frac{1}{q_0}} +\operatorname{Tail}(u;x_0, r) \right]
 \end{eqnarray*}
for any $\frac{1}{2}\leq \gamma <\varrho<1$. Thus, it follows from the iteration Lemma \ref{iterate} that the Caccioppoli type inequality \eqref{CCPu} is valid.
\end{proof}

Next, we focus on establish the Morrey type estimates for $Du$ by using the comparison estimate \eqref{comest-Du-Dv} and the Morrey type growth Lemma \ref{Dv-morrey} for $Dv$.
We first consider the case in which the partial map $x\mapsto\mathcal{A}(x,\cdot)$ is merely measurable.

\begin{lemma}\label{Du-morrey}
Let $u\in W^{1,p}_{\rm loc}(\Omega)\cap {\mathcal L}^{p-1}_{sp}(\mathbb{R}^n)$ be a weak solution to \eqref{eq1} under assumptions \eqref{growth1}--\eqref{ellip} with $p > 2-\frac{1}{n}$, $s \in (0,1)$ and $q_0=\max\{1,p-1\}$.
Then there exists a positive constant $C$ depending only on $n,\,p,\,s,\,\Lambda$ and $\operatorname{diam}(\Omega)$ such that
\begin{eqnarray*}
&& \left(\fint_{B_{\rho}(x_0)} |Du|^{q_0} \operatorname{d}\! x\right)^{\frac{1}{q_0}}\\
&\leq& C\left( \frac{\rho}{r} \right)^{-1+\alpha_m} \left[\left(\fint_{B_{r}(x_0)} |Du|^{q_0} \operatorname{d}\! x\right)^{\frac{1}{q_0}}+\frac{1}{r}\operatorname{Tail}(u-(u)_{B_r(x_0)};x_0,r)\right]\\
&&   +C\left(\frac{r}{\rho}\right)^{\max\{\frac{n}{q_0},1-\alpha_m\}} \left\{
\left[\frac{|\mu|(B_{r}(x_0))}{r^{n-1}}\right]^{\frac{1}{p-1}}+\chi_{\{p<2\}}
\left[\frac{|\mu|(B_{r}(x_0))}{r^{n-1}}\right]
\left(\fint_{B_{r}(x_0)} |Du|\operatorname{d}\! x\right)^{2-p}\right\}
\end{eqnarray*}
holds whenever $B_{\rho}(x_0) \subset B_{r}(x_0) \Subset \Omega$, where $\alpha_m$ is the H\"{o}lder continuity exponent of $v$, as given in Lemma \ref{mixv-holder}.
\end{lemma}

\begin{proof}
Let $v$ be a weak solution to the same Dirichlet problem \eqref{Commodel-vu} in the ball $B_{r}(x_0)$ instead of $B_{2r}(x_0)$. A combination of Lemma \ref{Dv-morrey} with Lemma \ref{comparison:Du-Dv} and Poincar\'{e}'s inequality yields that
\begin{eqnarray*}
   &&  \left(\fint_{B_{\rho}(x_0)} |Du|^{q_0} \operatorname{d}\! x\right)^{\frac{1}{q_0}}\\
   &\leq &  \left(\fint_{B_{\rho}(x_0)} |Dv|^{q_0} \operatorname{d}\! x\right)^{\frac{1}{q_0}}+\left(\frac{r}{\rho}\right)^{\frac{n}{q_0}} \left(\fint_{B_{r}(x_0)} |Du-Dv|^{q_0} \operatorname{d}\! x\right)^{\frac{1}{q_0}} \\
   &\leq& C\left( \frac{\rho}{r} \right)^{-1+\alpha_m} \left(\fint_{B_{r}(x_0)} |Du|^{q_0} \operatorname{d}\! x\right)^{\frac{1}{q_0}}+C \left[\left( \frac{\rho}{r} \right)^{-1+\alpha_m} +\left(\frac{r}{\rho}\right)^{\frac{n}{q_0}}\right] \left(\fint_{B_{r}(x_0)} |Du-Dv|^{q_0} \operatorname{d}\! x\right)^{\frac{1}{q_0}}\\
   &&+C  \left( \frac{\rho}{r} \right)^{-1+\alpha_m} \frac{1}{r}\operatorname{Tail}(v-(v)_{B_r(x_0)};x_0,r)\\
    &\leq& C\left( \frac{\rho}{r} \right)^{-1+\alpha_m} \left(\fint_{B_{r}(x_0)} |Du|^{q_0} \operatorname{d}\! x\right)^{\frac{1}{q_0}} +C\left(\frac{r}{\rho}\right)^{\max\{\frac{n}{q_0},1-\alpha_m\}} \left(\fint_{B_{r}(x_0)} |Du-Dv|^{q_0} \operatorname{d}\! x\right)^{\frac{1}{q_0}}\\
&&    +C  \left( \frac{\rho}{r} \right)^{-1+\alpha_m} \frac{1}{r}\operatorname{Tail}(u-(u)_{B_r(x_0)};x_0,r)
 +C\left(\frac{r}{\rho}\right)^{1-\alpha_m}r^{(1-s)p'} \left(\fint_{B_{r}(x_0)} |Du-Dv|^{q_0} \operatorname{d}\! x\right)^{\frac{1}{q_0}}\\
&\leq& C\left( \frac{\rho}{r} \right)^{-1+\alpha_m} \left[\left(\fint_{B_{r}(x_0)} |Du|^{q_0} \operatorname{d}\! x\right)^{\frac{1}{q_0}}+\frac{1}{r}\operatorname{Tail}(u-(u)_{B_r(x_0)};x_0,r)\right]\\
&&   +C\left(\frac{r}{\rho}\right)^{\max\{\frac{n}{q_0},1-\alpha_m\}} \left\{
\left[\frac{|\mu|(B_{r}(x_0))}{r^{n-1}}\right]^{\frac{1}{p-1}}+\chi_{\{p<2\}}
\left[\frac{|\mu|(B_{r}(x_0))}{r^{n-1}}\right]
\left(\fint_{B_{r}(x_0)} |Du|\operatorname{d}\! x\right)^{2-p}\right\}.
\end{eqnarray*}
Thus, we complete the proof of Lemma \ref{Du-morrey}.
\end{proof}

If we employ the Morrey type growth estimate \eqref{decay-Dw} for $Dw$ instead of \eqref{morreyDv} for $Dv$, and combine the comparison estimates \eqref{comest-Dv-Dw} for $Dv$ and $Dw$ with \eqref{comest-Du-Dv} for $Du$ and $Dv$, then the following Morrey type estimate for $Du$ is treated in a completely similar manner to that in the proof of Lemma \ref{Du-morrey}.

\begin{lemma}\label{Du-morrey2-3}
Let $u\in W^{1,p}_{\rm loc}(\Omega)\cap {\mathcal L}^{p-1}_{sp}(\mathbb{R}^n)$ be a weak solution to \eqref{eq1} under assumptions \eqref{growth1}--\eqref{ellip} with $p > 2-\frac{1}{n}$, $s \in (0,1)$ and $q_0=\max\{1,p-1\}$.
Then there exists a positive constant $C$ depending only on $n,\,p,\,s,\,\Lambda$ and $\operatorname{diam}(\Omega)$ such that
\begin{eqnarray}\label{Du-morrey2}
&& \left(\fint_{B_{\rho}(x_0)} |Du|^{q_0} \operatorname{d}\! x\right)^{\frac{1}{q_0}}\nonumber\\
&\leq& C\left(\fint_{B_{2r}(x_0)} |Du|^{q_0} \operatorname{d}\! x\right)^{\frac{1}{q_0}}+C\left(\frac{r}{\rho}\right)^{\frac{n}{q_0}}
\left(\big[\omega(r)\big]^{\sigma_1}+r^{\sigma_2}\right)\left(\fint_{B_{2r}(x_0)} |Du|^{q_0} \operatorname{d}\! x\right)^{\frac{1}{q_0}}\nonumber\\
&&+C\left(\frac{r}{\rho}\right)^{\frac{n}{q_0}}\frac{r^{\sigma_3}}{r^{p'}}
\operatorname{Tail}(u-(u)_{B_{2r}(x_0)};x_0,2r)\nonumber\\
&&   +C\left(\frac{r}{\rho}\right)^{\frac{n}{q_0}} \left\{
\left[\frac{|\mu|(B_{2r}(x_0))}{r^{n-1}}\right]^{\frac{1}{p-1}}+\chi_{\{p<2\}}
\left[\frac{|\mu|(B_{2r}(x_0))}{r^{n-1}}\right]
\left(\fint_{B_{2r}(x_0)} |Du|\operatorname{d}\! x\right)^{2-p}\right\}
\end{eqnarray}
holds whenever $B_{\rho}(x_0)\subset B_r(x_0) \subset B_{2r}(x_0) \Subset \Omega$, where the exponents $\sigma_1,\,\sigma_2$ and $\sigma_3$ have been defined in \eqref{exponent1}.
%Moreover, if $\big[\omega(r)\big]^{\sigma_1'}$ is assumed to be
%Dini-$\operatorname{VMO}$ regular, then there exist a positive constant $C$
%depending only on $n$, $p$, $s$, $\Lambda$, $\operatorname{diam}(\Omega)$,
%and a radius $\tilde{R}$ depending only on $n$, $p$, $s$, $\Lambda$, $m$, $\omega(\cdot)$, $\operatorname{diam}(\Omega)$ for $p\geq2$ and additionally on $\varsigma$ for $p\in (2-\frac{1}{n},2)$, such that
%\begin{eqnarray}\label{Du-morrey3}
%&& \left(\fint_{B_{\rho}(x_0)} |Du|^{q_0} \operatorname{d}\! x\right)^{\frac{1}{q_0}}\nonumber\\
%&\leq& C\left(\fint_{B_{2r}(x_0)} |Du|^{q_0} \operatorname{d}\! x\right)^{\frac{1}{q_0}}+C\left(\frac{r}{\rho}\right)^{\frac{n}{q_0}}
%\left(\big[\omega(r)\big]^{\sigma_1'}+r^{\sigma_2}\right)\left(\fint_{B_{2r}(x_0)} |Du|^{q_0} \operatorname{d}\! x\right)^{\frac{1}{q_0}}\nonumber\\
%&&+C\left(\frac{r}{\rho}\right)^{\frac{n}{q_0}}\frac{\left([\omega(r)\big]^{\sigma_1'}r^{\max\{0,sp'-1\}} + r^{\sigma_3}\right)}{r^{p'}}
%\operatorname{Tail}(u-(u)_{B_{2r}(x_0)};x_0,2r)\nonumber\\
%&&   +C\left(\frac{r}{\rho}\right)^{\frac{n}{q_0}} \left\{
%\left[\frac{|\mu|(B_{2r}(x_0))}{r^{n-1}}\right]^{\frac{1}{p-1}}+\chi_{\{p<2\}}
%\left[\frac{|\mu|(B_{2r}(x_0))}{r^{n-1}}\right]
%\left(\fint_{B_{2r}(x_0)} |Du|\operatorname{d}\! x\right)^{2-p}\right\}
%\end{eqnarray}
%holds for the ball
%$B_{\rho}(x_0)\subset B_r(x_0)  \subset B_{2r}(x_0) \Subset \Omega$ with the radius $r\in(0,\frac{\tilde{R}}{4}]$, where the exponents $\sigma_1'$ is given by \eqref{exponent2}.
\end{lemma}

We conclude this section by presenting a Campanato type decay estimate for the gradient of solutions to the equation \eqref{eq1} with a Dini-$\operatorname{VMO}$ regular coefficient. Adopting the comparison scheme  \eqref{comest-Dv-Dw-Dinivmo} for $Dv$ and $Dw$, we replace the Morrey type estimate with the Campanato type decay estimate \eqref{exdecay-Dw} to handle $Dw$.

\begin{lemma}\label{Lem:Du-camp}
Let $u\in W^{1,p}_{\rm loc}(\Omega)\cap {\mathcal L}^{p-1}_{sp}(\mathbb{R}^n)$ be a weak solution to \eqref{eq1} under assumptions \eqref{growth1}--\eqref{ellip} with $p > 2-\frac{1}{n}$, $s \in (0,1)$, $\max\{sp'-1,0\}<a<\sigma_3$ and $q_0=\max\{1,p-1\}$.
Assume that $[\omega(r)]^{\sigma_1'}$ satisfies Dini-$\operatorname{VMO}$ regular condition \eqref{Dinivmo}, then there exist a positive constant $C$
depending only on $n$, $p$, $s$, $\Lambda$, $m$, $a$, $\operatorname{diam}(\Omega)$,
and a radius $\tilde{R}$ depending only on $n$, $p$, $s$, $\Lambda$, $m$, $\omega(\cdot)$, $\operatorname{diam}(\Omega)$, such that
\begin{eqnarray}\label{Du-camp}
&& \left(\fint_{B_{\rho}(x_0)} |Du-(Du)_{B_{\rho}(x_0)}|^{q_0} \operatorname{d}\! x\right)^{\frac{1}{q_0}}\nonumber\\
&\leq& C\left(\frac{\rho}{r}\right)^{\alpha_M}\left(\fint_{B_{2r}(x_0)} |Du-(Du)_{B_{2r}(x_0)}|^{q_0} \operatorname{d}\! x\right)^{\frac{1}{q_0}}+C\left(\frac{r}{\rho}\right)^{\frac{n}{q_0}}
\left(\big[\omega(r)\big]^{\sigma_1'}+r^{\sigma_2}\right)\left(\fint_{B_{2r}(x_0)} |Du|^{q_0} \operatorname{d}\! x\right)^{\frac{1}{q_0}}\nonumber\\
&&+C\left(\frac{r}{\rho}\right)^{\frac{n}{q_0}}\frac{\left([\omega(r)\big]^{\sigma_1'}r^{a} + r^{\sigma_3}\right)}{r^{p'}}
\operatorname{Tail}(u-(u)_{B_{2r}(x_0)};x_0,2r)\nonumber\\
&&   +C\left(\frac{r}{\rho}\right)^{\frac{n}{q_0}} \left\{
\left[\frac{|\mu|(B_{2r}(x_0))}{r^{n-1}}\right]^{\frac{1}{p-1}}+\chi_{\{p<2\}}
\left[\frac{|\mu|(B_{2r}(x_0))}{r^{n-1}}\right]
\left(\fint_{B_{2r}(x_0)} |Du|\operatorname{d}\! x\right)^{2-p}\right\}
\end{eqnarray}
holds for the ball
$B_{\rho}(x_0) \subset B_{2r}(x_0) \Subset \Omega$ with $r\in(0,\frac{\tilde{R}}{4}]$, where the exponents $\sigma_1',\,\sigma_2$ and $\sigma_3$ have been defined in \eqref{exponent2} and \eqref{exponent1}, respectively.
\end{lemma}

\section{Pointwise estimates of nonlocal fractional maximal functions}\label{section4}

In this section, we complete the proof of pointwise estimates for truncated nonlocal fractional maximal functions of the solution and its gradient for the mixed local and nonlocal equation \eqref{eq1}, as established in Theorem \ref{Maxmeasure} - Theorem \ref{MaxDiniholder}.
We omit the proof of Theorem \ref{Maxmeasure} here, since it
follows the same lines as the proof of Theorem \ref{MaxsmallBMO}, except for the use of Lemma \ref{Du-morrey} instead of Lemma \ref{Du-morrey2-3}.

\begin{proof}[Proof of Theorem \ref{MaxsmallBMO}] In terms of \eqref{Sharpm-M}, it suffices to prove that
\begin{eqnarray}\label{Maxpoint-Measurablesimple}
  &&  \mathfrak{M}_{1-\alpha,q_0}(Du)(x,R)
   \leq  C\left[ {\bf M}_{p-\alpha(p-1)}(\mu)(x,R)\right]^{\frac{1}{p-1}} \nonumber\\
  &&\qquad\qquad\qquad+CR^{1-\alpha} \left[\left(\fint_{B_{R}(x)} |Du|^{q_0} \operatorname{d}\! \xi\right)^{\frac{1}{q_0}}+R^{-1-\sigma}\operatorname{Tail}(u-(u)_{B_R(x)};x,R)\right]\,.
  \end{eqnarray}
Let $B_{\rho}(x)\subset B_{\frac{r}{2}}(x) \subset B_{r}(x) \subset B_R(x)$ be a set of concentric balls. Multiplying both sides of \eqref{Du-morrey2} by $\rho^{1-\alpha}$, and invoking Lemma \ref{v-holder} and Lemma \ref{comparison:Du-Dv} to bound the tail term in a similar manner to \cite[Lemma 4.12]{CSYZ24}, we derive
\begin{eqnarray*}
&& \rho^{1-\alpha}\left(\fint_{B_{\rho}(x)} |Du|^{q_0} \operatorname{d}\! \xi\right)^{\frac{1}{q_0}}+\rho^{-\sigma-\alpha} \operatorname{Tail}(u-(u)_{B_{\rho}(x)};x,\rho) \nonumber\\
&\leq& c_1\left(\frac{\rho}{r}\right)^{1-\alpha}r^{1-\alpha}\left(\fint_{B_{r}(x)} |Du|^{q_0} \operatorname{d}\! \xi\right)^{\frac{1}{q_0}}
+c_2\left(\frac{\rho}{r}\right)^{p'-\alpha-\sigma}r^{-\sigma-\alpha}
\operatorname{Tail}(u-(u)_{B_{r}(x)};x,r)
\nonumber\\
&&+c_3\left[\left(\frac{r}{\rho}\right)^{\frac{n}{q_0}-1+\alpha}
\left(\big[\omega(r)\big]^{\sigma_1}+r^{\sigma_2}\right)
+\left(\frac{r}{\rho}\right)^{\frac{n}{p-1}+\alpha+\sigma}r^{(1-s)p'-\sigma}\right]r^{1-\alpha}\left(\fint_{B_{r}(x)} |Du|^{q_0} \operatorname{d}\! \xi\right)^{\frac{1}{q_0}}\nonumber\\
&&+c_4\left[\left(\frac{r}{\rho}\right)^{\frac{n}{q_0}-1+\alpha}r^{\sigma_3+\sigma-\frac{1}{p-1}}
+\left(\frac{r}{\rho}\right)^{\alpha+\sigma-\alpha_m}
r^{(1-s)p'}\right]r^{-\sigma-\alpha}\operatorname{Tail}(u-(u)_{B_{r}(x)};x,r)\nonumber\\
&&+C\left[\left(\frac{r}{\rho}\right)^{\frac{n}{q_0}-1+\alpha}
+\left(\frac{r}{\rho}\right)^{\frac{n}{p-1}+\alpha+\sigma}r^{(1-s)p'-\sigma}\right]
\left[\frac{|\mu|(B_{r}(x))}{r^{n-p+\alpha(p-1)}}\right]^{\frac{1}{p-1}}
\nonumber\\
&&  +C\chi_{\{p<2\}}\left(\frac{r}{\rho}\right)^{\frac{n}{q_0}-1+\alpha}
\left[\frac{|\mu|(B_{r}(x))}{r^{n-p+\alpha(p-1)}}\right]
\left(r^{1-\alpha}\fint_{B_{r}(x)} |Du|\operatorname{d}\! \xi\right)^{2-p}\,.
\end{eqnarray*}
Given that $\sigma<\frac{1}{p-1}$, we now set $\rho=\frac{r}{L}$, where $L\equiv L(n,p,s,\Lambda,\tilde{\alpha},\sigma,\operatorname{diam}(\Omega))\geq2$ is chosen to be sufficiently large such that
$$c_1\left(\frac{1}{L}\right)^{1-\alpha}\leq c_1\left(\frac{1}{L}\right)^{1-\tilde{\alpha}}\leq\frac{1}{8}$$
and
$$c_2\left(\frac{1}{L}\right)^{p'-\alpha-\sigma}\leq c_2\left(\frac{1}{L}\right)^{\frac{1}{p-1}-\sigma}\leq \frac{1}{8}.$$
Due to the fact that $\omega(r)$ is small $\operatorname{BMO}$ regular, the exponent $q$ can be chosen arbitrarily in the interval $(p,+\infty)$, as stated in Remark \ref{q-range}. Thus, we may take $q = 2p$.
Since the exponents $\sigma_1,\,\sigma_2>0$ and $\frac{1}{p-1}-\sigma_3<\sigma<(1-s)p'$,
we further select positive parameters
$\bar{R}\equiv\bar{R}(n,p,s,\Lambda,\tilde{\alpha},m,\sigma,\omega(\cdot),\operatorname{diam}(\Omega))$
and $\delta\equiv\delta(n,p,s,\Lambda,\tilde{\alpha},\sigma,\operatorname{diam}(\Omega))$
small enough to ensure that
\begin{eqnarray*}
   && c_3\left[L^{\frac{n}{q_0}-1+\alpha}
\left(\big[\omega(r)\big]^{\sigma_1}+r^{\sigma_2}\right)+L^{\frac{n}{p-1}+\alpha+\sigma}r^{(1-s)p'-\sigma}\right] \\
   &\leq&  c_3\left[L^{\frac{n}{q_0}}
\left(\delta^{\sigma_1}+\bar{R}^{\sigma_2}\right)+L^{\frac{n}{p-1}+\tilde{\alpha}+\sigma}\bar{R}^{(1-s)p'-\sigma}\right]\leq \frac{1}{8},
\end{eqnarray*}
as well as
$$c_4\left[L^{\frac{n}{q_0}-1+\alpha}r^{\sigma_3+\sigma-\frac{1}{p-1}}
+L^{\alpha+\sigma-\alpha_m}
r^{(1-s)p'}\right]\leq c_4\left[L^{\frac{n}{q_0}-1+\tilde{\alpha}}\bar{R}^{\sigma_3+\sigma-\frac{1}{p-1}}
+L^{\tilde{\alpha}+\sigma-\alpha_m}
\bar{R}^{(1-s)p'}\right]\leq \frac{1}{8}.$$
Note that introducing the exponent $\sigma$ guarantees that the constant preceding the tail term can be sufficiently small.

Based on the above analysis and applying Young's inequality, we deduce that
\begin{eqnarray}\label{Du-maxBMO-1}
&& \rho^{1-\alpha}\left(\fint_{B_{\rho}(x)} |Du|^{q_0} \operatorname{d}\! \xi\right)^{\frac{1}{q_0}}+\rho^{-\sigma-\alpha} \operatorname{Tail}(u-(u)_{B_{\rho}(x)};x,\rho) \nonumber\\
&\leq&\frac{1}{2}\left[ r^{1-\alpha}\left(\fint_{B_{r}(x)} |Du|^{q_0} \operatorname{d}\! \xi\right)^{\frac{1}{q_0}}+r^{-\sigma-\alpha} \operatorname{Tail}(u-(u)_{B_{r}(x)};x,r)\right]+C
\left[\frac{|\mu|(B_{r}(x))}{r^{n-p+\alpha(p-1)}}\right]^{\frac{1}{p-1}}
\nonumber\\
&\leq& \frac{1}{2}\mathfrak{M}_{1-\alpha,q_0}(Du)(x,R)
+C\left[ {\bf M}_{p-\alpha(p-1)}(\mu)(x,R)\right]^{\frac{1}{p-1}},
\end{eqnarray}
where the positive constant $C\equiv C(n,p,s,\Lambda,\tilde{\alpha},m,\sigma,\omega(\cdot),\operatorname{diam}(\Omega))$.
Due to $\rho=\frac{r}{L}$ and the arbitrariness of $0<r\leq R$, it follows from
\eqref{Du-maxBMO-1} that
\begin{eqnarray}\label{Du-maxBMO-2}
 && \sup_{0<\rho\leq \frac{R}{L}}\rho^{1-\alpha}\left(\fint_{B_{\rho}(x)} |Du|^{q_0} \operatorname{d}\! \xi\right)^{\frac{1}{q_0}}+\rho^{-\sigma-\alpha} \operatorname{Tail}(u-(u)_{B_{\rho}(x)};x,\rho)\nonumber\\
 &\leq& \frac{1}{2}\mathfrak{M}_{1-\alpha,q_0}(Du)(x,R)
+C\left[ {\bf M}_{p-\alpha(p-1)}(\mu)(x,R)\right]^{\frac{1}{p-1}}.
\end{eqnarray}
On the other hand, applying Poincar\'{e}'s inequality, H\"{o}lder' inequality and $\sigma<(1-s)p'$, we proceed to calculate
\begin{eqnarray}\label{Du-maxBMO-3}
    && \sup_{\frac{R}{L}<\rho\leq R}\rho^{1-\alpha}\left(\fint_{B_{\rho}(x)} |Du|^{q_0} \operatorname{d}\! \xi\right)^{\frac{1}{q_0}}+\rho^{-\sigma-\alpha} \operatorname{Tail}(u-(u)_{B_{\rho}(x)};x,\rho)\nonumber \\
   &\leq& C\left(R^{1-\alpha}+R^{(1-s)p'-\sigma+1-\alpha}\right)\left(\fint_{B_{R}(x)} |Du|^{q_0} \operatorname{d}\! \xi\right)^{\frac{1}{q_0}}+CR^{-\sigma-\alpha} \operatorname{Tail}(u-(u)_{B_{R}(x)};x,R)\nonumber \\
   &\leq& CR^{1-\alpha}\left(\fint_{B_{R}(x)} |Du|^{q_0} \operatorname{d}\! \xi\right)^{\frac{1}{q_0}}+CR^{-\sigma-\alpha} \operatorname{Tail}(u-(u)_{B_{R}(x)};x,R).
\end{eqnarray}
Thus, a combination of \eqref{Du-maxBMO-2} and \eqref{Du-maxBMO-3} yields that \eqref{Maxpoint-Measurablesimple} holds uniformly in $\alpha\in[0,\tilde{\alpha}]$ for a positive constant $C\equiv C(n,p,s,\Lambda,\tilde{\alpha},m,\sigma,\omega(\cdot),\operatorname{diam}(\Omega))$ and a radius $R\leq \bar{R}$ as determined above.

The last step is to remove the restriction on radius $R\leq \bar{R}\equiv\bar{R}(n,p,s,\Lambda,\tilde{\alpha},m,\sigma,\omega(\cdot),\operatorname{diam}(\Omega))$. In the case of $R>\bar{R}$, we apply \eqref{Maxpoint-Measurablesimple} with $R=\bar{R}$ and adopt the estimate strategy from \eqref{Du-maxBMO-3} to derive that
\begin{eqnarray*}
     &&\mathfrak{M}_{1-\alpha,q_0}(Du)(x,R)\\
     &\leq&\mathfrak{M}_{1-\alpha,q_0}(Du)(x,\bar{R}) +\sup_{\bar{R}<r\leq R}r^{1-\alpha}\left(\fint_{B_{r}(x)} |Du|^{q_0} \operatorname{d}\! \xi\right)^{\frac{1}{q_0}}+r^{-\sigma-\alpha} \operatorname{Tail}(u-(u)_{B_{r}(x)};x,r)\\
   &\leq& C\left[ {\bf M}_{p-\alpha(p-1)}(\mu)(x,\bar{R})\right]^{\frac{1}{p-1}}
+C\left(\frac{R}{\bar{R}}\right)^{\tilde{\alpha}+\sigma}R^{-\sigma-\alpha}\operatorname{Tail}(u-(u)_{B_{R}(x)};x,R)\\
&&+C\left[\left(\frac{R}{\bar{R}}\right)^{\frac{n}{q_0}}R^{1-\alpha}+\left(\frac{R}{\bar{R}}\right)^{\frac{n+sp}{p-1}+\tilde{\alpha}+\sigma}R^{(1-s)p'-\sigma+1-\alpha}\right]\left(\fint_{B_{R}(x)} |Du|^{q_0} \operatorname{d}\! \xi\right)^{\frac{1}{q_0}}\\
&\leq&C\left[ {\bf M}_{p-\alpha(p-1)}(\mu)(x,R)\right]^{\frac{1}{p-1}}
 +CR^{1-\alpha} \left[\left(\fint_{B_{R}(x)} |Du|^{q_0} \operatorname{d}\! \xi\right)^{\frac{1}{q_0}}+R^{-1-\sigma}\operatorname{Tail}(u-(u)_{B_R(x)};x,R)\right]\,,
\end{eqnarray*}
where the positive constant $C$ depends only on $n,\,p,\,s,\,\Lambda,\,\tilde{\alpha},\,m,\,\sigma,\,\omega(\cdot),\,\operatorname{diam}(\Omega)$. Hence, the proof of \eqref{Maxpoint-SmBMO} is complete.
\end{proof}

The next objective is to demonstrate that the pointwise maximum estimate established in Theorem \ref{MaxDinivmo} holds when the coefficient satisfies Dini-$\operatorname{VMO}$ regularity.

\begin{proof}[Proof of Theorem \ref{MaxDinivmo}]
Let  $q_0=\max\{1,p-1\}$ and $\mathcal{B}_i:=B_{R_i}(x)$ be a sequence of shrinking balls with the radius $R_i:=\frac{R}{(2L)^{i}}$ for $i\in \mathbb{N}$ and $R\leq \frac{\tilde{R}}{2}$, where $L\geq 1$ is a large constant to be determined later and $\tilde{R}$ is given by Lemma \ref{Dv-L^infty}.
We denote $G$ as a fixed vector in $\R^n$, and
define
\begin{equation*}
  \tilde{A}_i:=\left( \fint_{\mathcal{B}_i}|Du-(Du)_{\mathcal{B}_i}|^{q_0}  \operatorname{d}\!\xi\right)^{\frac{1}{q_0}}\,\, \mbox{and}\,\, \tilde{K}_i:=|(Du)_{\mathcal{B}_i}-G|.
\end{equation*}
By virtue of Lemma \ref{Lem:Du-camp} with $\mathcal{B}_{i+1}\subset B_{\frac{R_i}{2}}(x)\subset \mathcal{B}_{i}$ for any $i\in\mathbb{N}$, we have
\begin{eqnarray}\label{max-diniVMO-A_i}
\tilde{A}_{i+1}
&\leq& c_5\left(\frac{1}{2L}\right)^{\alpha_M}\tilde{A}_i+c_6(2L)^{\frac{n}{q_0}}
\left(\big[\omega(R_i)\big]^{\sigma_1'}+R_i^{\sigma_2}\right)\left(\tilde{A}_i+\tilde{K}_i+|G|\right)\nonumber\\
&&+CL^{\frac{n}{q_0}}\frac{\left([\omega(R_i)\big]^{\sigma_1'}R_i^{a} + R_i^{\sigma_3}\right)}{R_i^{p'}}
\operatorname{Tail}(u-(u)_{\mathcal{B}_{i}};x,R_i)\nonumber\\
&&   +CL^{\frac{n}{q_0}} \left\{
\left[\frac{|\mu|(\mathcal{B}_{i})}{R_i^{n-1}}\right]^{\frac{1}{p-1}}+\chi_{\{p<2\}}
\left[\frac{|\mu|(\mathcal{B}_{i})}{R_i^{n-1}}\right]
\left(\fint_{\mathcal{B}_{i}} |Du|\operatorname{d}\! \xi\right)^{2-p}\right\}.
\end{eqnarray}
We choose $L_1\equiv L_1(n,p,s,\Lambda,m,a,\operatorname{diam}(\Omega))$ sufficiently large such that
$c_5\left(\frac{1}{2L}\right)^{\alpha_M}\leq\frac{1}{16}$ for $L\geq L_1$, and then select a positive radius $\bar{R}_1\equiv\bar{R}_1(n,p,s,\Lambda,\omega(\cdot),m,a,\operatorname{diam}(\Omega))$ small enough to ensure
$$c_6(2L)^{\frac{n}{q_0}}
\left(\big[\omega(R_i)\big]^{\sigma_1'}+R_i^{\sigma_2}\right)\leq c_6(2L)^{\frac{n}{q_0}}
\left(\big[\omega(\bar{R}_1)\big]^{\sigma_1'}+\bar{R}_1^{\sigma_2}\right)\leq \frac{1}{16}$$
for $R \leq \min\{\bar{R}_1,\frac{\tilde{R}}{2}\}$. From these selections, we sum up \eqref{max-diniVMO-A_i} over $i$ to derive
\begin{eqnarray}\label{max-diniVMO-1}
\sum_{i=1}^j\tilde{A}_{i}
&\leq& \tilde{A}_0+C\sum_{i=0}^{j-1}
\left(\big[\omega(R_i)\big]^{\sigma_1'}+R_i^{\sigma_2}\right)\left(\tilde{K}_i+|G|\right)\nonumber\\
&&+C\sum_{i=0}^{j-1}\frac{\left([\omega(R_i)\big]^{\sigma_1'}R_i^{a} + R_i^{\sigma_3}\right)}{R_i^{p'}}
\operatorname{Tail}(u-(u)_{\mathcal{B}_{i}};x,R_i)\nonumber\\
&&   +C\sum_{i=0}^{j-1} \left\{
\left[\frac{|\mu|(\mathcal{B}_{i})}{R_i^{n-1}}\right]^{\frac{1}{p-1}}+\chi_{\{p<2\}}
\left[\frac{|\mu|(\mathcal{B}_{i})}{R_i^{n-1}}\right]
\left(\fint_{\mathcal{B}_{i}} |Du|\operatorname{d}\! \xi\right)^{2-p}\right\}
\end{eqnarray}
for any $j\geq 1$. Applying H\"{o}lder's inequality, we compute
\begin{equation}\label{max-diniVMO-2}
\tilde{K}_{0}+\tilde{K}_1\leq \left[1+(2L)^{\frac{n}{q_0}}\right]\left(\fint_{B_{R}(x)} |Du-G|^{q_0}\operatorname{d}\! \xi\right)^{\frac{1}{q_0}}
\end{equation}
Then a combination of \eqref{max-diniVMO-1} and \eqref{max-diniVMO-2} yields that
\begin{eqnarray}\label{max-diniVMO-2-1}
  \tilde{K}_{j+1} &\leq&(2L)^{\frac{n}{q_0}}\sum_{i=0}^j\tilde{A}_{i}+\tilde{K}_0  \nonumber \\
   &\leq&C\left( \fint_{B_R(x)}|Du-(Du)_{B_R(x)}|^{q_0}+|Du-G|^{q_0}  \operatorname{d}\!\xi\right)^{\frac{1}{q_0}}+C\sum_{i=0}^{j}
\left(\big[\omega(R_i)\big]^{\sigma_1'}+R_i^{\sigma_2}\right)\left(\tilde{K}_i+|G|\right)\nonumber\\
&&+C\sum_{i=0}^{j}\frac{\left([\omega(R_i)\big]^{\sigma_1'}R_i^{a} + R_i^{\sigma_3}\right)}{R_i^{p'}}
\operatorname{Tail}(u-(u)_{\mathcal{B}_{i}};x,R_i)\nonumber\\
&&   +C\sum_{i=0}^{j} \left\{
\left[\frac{|\mu|(\mathcal{B}_{i})}{R_i^{n-1}}\right]^{\frac{1}{p-1}}+\chi_{\{p<2\}}
\left[\frac{|\mu|(\mathcal{B}_{i})}{R_i^{n-1}}\right]
\left(\fint_{\mathcal{B}_{i}} |Du|\operatorname{d}\! \xi\right)^{2-p}\right\}
\end{eqnarray}
whenever $j\in\mathbb{N}$. We now specify $G=\bf{0}$ and multiply both sides of the previous inequality by $R_j^{1-\alpha}$ to derive
\begin{eqnarray}\label{max-diniVMO-3}
  R_{j+1}^{1-\alpha}\tilde{K}_{j+1}
   &\leq&CR^{1-\alpha}\left( \fint_{B_R(x)}|Du|^{q_0}\operatorname{d}\!\xi\right)^{\frac{1}{q_0}}+C\sum_{i=0}^{j}
\left(\big[\omega(R_i)\big]^{\sigma_1'}+R_i^{\sigma_2}\right)R_i^{1-\alpha}\tilde{K}_i\nonumber\\
&&+C\sum_{i=0}^{j}\left([\omega(R_i)\big]^{\sigma_1'}R_i^{a+1+\sigma-p'} + R_i^{\sigma_3+1+\sigma-p'}\right)R_i^{-\sigma-\alpha}
\operatorname{Tail}(u-(u)_{\mathcal{B}_{i}};x,R_i)\nonumber\\
&& + \frac{\varepsilon}{16}
\mathfrak{M}_{1-\alpha,q_0}(Du)(x,R) +C\sum_{i=0}^{j}
\left[\frac{|\mu|(\mathcal{B}_{i})}{R_i^{n-p+\alpha(p-1)}}\right]^{\frac{1}{p-1}},
\end{eqnarray}
by using the fact that $R_{j+1}\leq R_i$ for every $i=0,...,j$ and Young's inequality with the parameter $\varepsilon\in (0,1)$ to be determined later. We divide the estimate of the last term in \eqref{max-diniVMO-3} into the superquadratic and subquadratic cases. If $p\leq 2$, then $\frac{1}{p-1}\geq 1$ and a direct calculation shows that
\begin{equation}\label{max-diniVMO-potentialp<=2}
  \sum_{i=0}^{j}
\left[\frac{|\mu|(\mathcal{B}_{i})}{R_i^{n-p+\alpha(p-1)}}\right]^{\frac{1}{p-1}}\leq
\left[\sum_{i=0}^{\infty}\frac{|\mu|(\mathcal{B}_{i})}{R_i^{n-p+\alpha(p-1)}}\right]^{\frac{1}{p-1}}
\leq C\left[{\bf{I}}^{\mu}_{p-\alpha(p-1)}(x,2R)\right]^{\frac{1}{p-1}}
\end{equation}
is valid for any $j\in\mathbb{N}$.
While for $p> 2$, we can compute that
\begin{equation}\label{max-diniVMO-potentialp>2}
  \sum_{i=0}^{\infty}
\left[\frac{|\mu|(\mathcal{B}_{i})}{R_i^{n-p+\alpha(p-1)}}\right]^{\frac{1}{p-1}}
\leq C{\bf{W}}^{\mu}_{1-\frac{\alpha (p-1)}{p},p}(x,2R),
\end{equation}
where the positive constant $C\equiv C(n,p,s,\Lambda,m,a,\operatorname{diam}(\Omega))$.

We first pay attention to the subquadratic case $2-\frac{1}{n}<p\leq 2$,
and claim that the following uniform estimate
\begin{equation}\label{max-diniVMO-4}
    R_{j+1}^{1-\alpha}\tilde{K}_{j+1}\leq CR^{1-\alpha} \left(\fint_{B_{R}(x)} |Du|^{q_0} \operatorname{d}\! \xi\right)^{\frac{1}{q_0}}
   +\frac{\varepsilon}{8}
\mathfrak{M}_{1-\alpha,q_0}(Du)(x,R)+ C\left[{\bf{I}}^{\mu}_{p-\alpha(p-1)}(x,2R)\right]^{\frac{1}{p-1}}=:M
\end{equation}
holds whenever $j\in\mathbb{N}$, where the positive constant $C\equiv C(n,p,s,\Lambda,m,a,\operatorname{diam}(\Omega))$ and $q_0=1$.
The proof proceeds by induction, starting with the case $j=0$, for which \eqref{max-diniVMO-4} is immediately established by using \eqref{max-diniVMO-2} with $G=0$. Assuming that $R_{i}^{1-\alpha}\tilde{K}_{i}\leq M$ is valid for every $i=1,...,j$, then we apply \eqref{max-diniVMO-3} and \eqref{max-diniVMO-potentialp<=2} to deduce that
\begin{eqnarray}\label{max-diniVMO-5}
  R_{j+1}^{1-\alpha}\tilde{K}_{j+1}
   &\leq&CR^{1-\alpha}\left( \fint_{B_R(x)}|Du|^{q_0}\operatorname{d}\!\xi\right)^{\frac{1}{q_0}}+c_7M\sum_{i=0}^{j}
\left(\omega(R_i)+R_i^{\sigma_2}\right)\nonumber\\
&&+c_8\mathfrak{M}_{1-\alpha,q_0}(Du)(x,R)\sum_{i=0}^{j}\left(\omega(R_i)R_i^{a+1+\sigma-p'} + R_i^{\sigma_3+1+\sigma-p'}\right)\nonumber\\
&& + \frac{\varepsilon}{16}
\mathfrak{M}_{1-\alpha,q_0}(Du)(x,R) +C\left[{\bf{I}}^{\mu}_{p-\alpha(p-1)}(x,2R)\right]^{\frac{1}{p-1}}.
\end{eqnarray}
Since $\big[\omega(r)\big]^{\varsigma}$ is Dini-$\operatorname{VMO}$ regular,
and in conjunction with the conditions $\alpha\leq 1$, $\sigma_2,\,a+\sigma-\frac{1}{p-1}>0$ and $a<\sigma_3$, we further restrict the radius $\bar{R}_2\equiv\bar{R}_2(n,p,s,\Lambda,m,a,\sigma,\omega(\cdot),\operatorname{diam}(\Omega),\varepsilon)>0$ small enough such that
\begin{equation*}
  c_7\sum_{i=0}^{j}
\left(\omega(R_i)+R_i^{\sigma_2}\right)\leq C\int_0^{2R}\left(\omega(\rho)+\rho^{\sigma_2}
  \right)\frac{\operatorname{d}\! \rho}{\rho}\leq C\int_0^{2\bar{R}_2}\left(\omega(\rho)+\rho^{\sigma_2}
  \right)\frac{\operatorname{d}\! \rho}{\rho}\leq\frac{\varepsilon}{32}
\end{equation*}
and
\begin{eqnarray*}
 c_8\sum_{i=0}^{j}\left(\omega(R_i)R_i^{a+1+\sigma-p'} + R_i^{\sigma_3+1+\sigma-p'}\right)
 \leq C\int_0^{2\bar{R}_2}\left(\omega(\rho) +\rho^{\sigma_3+\sigma-\frac{1}{p-1}}\right)\frac{\operatorname{d}\! \rho}{\rho}\leq\frac{\varepsilon}{16}
\end{eqnarray*}
hold whenever $R\leq \min\{\bar{R}_1,\bar{R}_2,\frac{\tilde{R}}{2}\}$. Substituting the above estimates into \eqref{max-diniVMO-5}, we prove the validity of the assertion \eqref{max-diniVMO-4}.

Furthermore, combining \eqref{max-diniVMO-A_i} with \eqref{max-diniVMO-4}, and using an analogous argument to that in the proof of \eqref{max-diniVMO-4}, we obtain
\begin{eqnarray*}
   R_{j+1}^{1-\alpha} \left(\fint_{\mathcal{B}_{j+1}} |Du|^{q_0} \operatorname{d}\! \xi\right)^{\frac{1}{q_0}}
   & \leq &R_{j+1}^{1-\alpha} \tilde{A}_{j+1} +R_{j+1}^{1-\alpha}\tilde{K}_{j+1}\nonumber\\
&\leq& \frac{1}{8}R_{j}^{1-\alpha} \tilde{A}_{j} +  M\nonumber\\
&\leq& \frac{1}{8}R_{j}^{1-\alpha} \left(\fint_{\mathcal{B}_{j}} |Du|^{q_0} \operatorname{d}\! \xi\right)^{\frac{1}{q_0}}+M.
\end{eqnarray*}
Thus, employing induction once again, we readily verify that
\begin{equation}\label{max-diniVMO-6}
  R_{j}^{1-\alpha} \left(\fint_{\mathcal{B}_{j}} |Du|^{q_0} \operatorname{d}\! \xi\right)^{\frac{1}{q_0}}\leq 2M
\end{equation}
is true for every $j\in \mathbb{N}$.
The next task is to validate the following uniform estimate for the nonlocal tail term
\begin{equation}\label{max-diniVMO-7}
  R_{j}^{-\sigma-\alpha}
\operatorname{Tail}(u-(u)_{\mathcal{B}_{j}};x,R_{j})\leq C  R^{-\sigma-\alpha}
\operatorname{Tail}(u-(u)_{B_{R}(x)};x,R)+M,
\end{equation}
whenever $j\in \mathbb{N}$. We proceed by utilizing \cite[Lemma 4.12]{CSYZ24} and \eqref{max-diniVMO-potentialp<=2} to derive
\begin{eqnarray*}
   &&  R_{j+1}^{-\sigma-\alpha}
\operatorname{Tail}(u-(u)_{\mathcal{B}_{j+1}};x,R_{j+1}) \\
   &\leq& c_9\left[\left(\frac{1}{2L}\right)^{p'-\alpha-\sigma}+(2L)^{1-\alpha_m+\sigma}R_j^{p'(1-s)}\right] R_{j}^{-\sigma-\alpha}
\operatorname{Tail}(u-(u)_{\mathcal{B}_{j}};x,R_{j})\\
&&+c_{10} R_j^{(1-s)p'-\sigma} \left[ R_{j}^{1-\alpha} \left(\fint_{\mathcal{B}_{j}} |Du|^{q_0} \operatorname{d}\! \xi\right)^{\frac{1}{q_0}}+\left[{\bf{I}}^{\mu}_{p-\alpha(p-1)}(x,2R)\right]^{\frac{1}{p-1}}\right].
\end{eqnarray*}
Since $p'-1-\sigma>0$ and $(1-s)p'-\sigma>0$, we further
choose $L_2\equiv L_2(n,p,s,\Lambda,\sigma)$ sufficiently large such that
$c_9\left(\frac{1}{2L}\right)^{p'-\alpha-\sigma}\leq c_9\left(\frac{1}{2L}\right)^{p'-1-\sigma}\leq\frac{1}{16}$ for $L\geq \max\{L_1, L_2\}$, and ultimately specify a positive radius $\bar{R}_3\equiv\bar{R}_3(n,p,s,\Lambda,\sigma,\operatorname{diam}(\Omega))$ small enough so that
$$c_9 (2L)^{1-\alpha_m+\sigma}R_j^{p'(1-s)}+c_{10} R_j^{(1-s)p'-\sigma}\leq c_9 (2L)^{1-\alpha_m+\sigma}\bar{R}_3^{p'(1-s)}+c_{10}\bar{R}_3^{(1-s)p'-\sigma}\leq\frac{1}{16} $$
for $R\leq \min\{\bar{R}_1,\bar{R}_2,\bar{R}_3,\frac{\tilde{R}}{2}\}$.
Based on the above selections and utilizing \eqref{max-diniVMO-6}, we have
\begin{equation*}
  R_{j+1}^{-\sigma-\alpha}
\operatorname{Tail}(u-(u)_{\mathcal{B}_{j+1}};x,R_{j+1})\leq \frac{1}{8}R_{j}^{-\sigma-\alpha}
\operatorname{Tail}(u-(u)_{\mathcal{B}_{j}};x,R_{j})+\frac{M}{4}.
\end{equation*}
Then we conclude that the uniform estimate \eqref{max-diniVMO-7} holds by induction.

We observe that there exists an integer $j\in \mathbb{N}$ such that $R_{j+1}<r\leq R_j$ for every $r\in(0, R]$. Therefore, by applying Poincar\'{e}'s inequality and H\"{o}lder's inequality together with the uniform estimates \eqref{max-diniVMO-6} and \eqref{max-diniVMO-7}, we arrive at
\begin{eqnarray*}
   && r^{1-\alpha}\left(\fint_{B_{r}(x)} |Du|^{q_0} \operatorname{d}\! \xi\right)^{\frac{1}{q_0}}+r^{-\sigma-\alpha} \operatorname{Tail}(u-(u)_{B_{r}(x)};x,r) \nonumber\\
   &\leq& c_{11} R_j^{1-\alpha}\left(\fint_{\mathcal{B}_j} |Du|^{q_0} \operatorname{d}\! \xi\right)^{\frac{1}{q_0}} + c_{11} R_{j}^{-\sigma-\alpha}
\operatorname{Tail}(u-(u)_{\mathcal{B}_{j}};x,R_{j}) \nonumber\\
&\leq&CR^{1-\alpha} \left(\fint_{B_{R}(x)} |Du|^{q_0} \operatorname{d}\! \xi\right)^{\frac{1}{q_0}}+CR^{-\sigma-\alpha}
\operatorname{Tail}(u-(u)_{B_{R}(x)};x,R) \nonumber\\
&& + C\left[{\bf{I}}^{\mu}_{p-\alpha(p-1)}(x,2R)\right]^{\frac{1}{p-1}}
+\frac{c_{11}\varepsilon}{2}
\mathfrak{M}_{1-\alpha,q_0}(Du)(x,R),
\end{eqnarray*}
where the positive constant $c_{11}$ depends only on $n,\,p,\,s,\,\Lambda,\,m,\,a,\,\sigma$ and $\operatorname{diam}(\Omega)$. We now choose $\varepsilon\equiv\varepsilon(n,p,s,\Lambda,m,a,\operatorname{diam}(\Omega)) \in (0,1)$ sufficiently small such that $\frac{c_{11}\varepsilon}{2}\leq \frac{1}{2}$, and apply the arbitrariness of $r\in(0, R]$ to obtain
\begin{eqnarray}\label{max-diniVMO-8}
 \mathfrak{M}_{1-\alpha,q_0}(Du)(x,R)&\leq&CR^{1-\alpha} \left(\fint_{B_{R}(x)} |Du|^{q_0} \operatorname{d}\! \xi\right)^{\frac{1}{q_0}}+CR^{-\sigma-\alpha}
\operatorname{Tail}(u-(u)_{B_{R}(x)};x,R) \nonumber\\
&& + C\left[{\bf{I}}^{\mu}_{p-\alpha(p-1)}(x,2R)\right]^{\frac{1}{p-1}}
+\frac{1}{2}
\mathfrak{M}_{1-\alpha,q_0}(Du)(x,R),
\end{eqnarray}
where the positive constant $C\equiv C(n,p,s,\Lambda,m,a,\sigma,\operatorname{diam}(\Omega))$.
The last term on the right side of \eqref{max-diniVMO-8} can be absorbed into the left side, and combined with the fact that
\begin{eqnarray*}
    && \sup_{R<r\leq 2R}r^{1-\alpha}\left(\fint_{B_{r}(x)} |Du|^{q_0} \operatorname{d}\! \xi\right)^{\frac{1}{q_0}}+r^{-\sigma-\alpha} \operatorname{Tail}(u-(u)_{B_{r}(x)};x,r)\nonumber \\
   &\leq& CR^{1-\alpha}\left(\fint_{B_{2R}(x)} |Du|^{q_0} \operatorname{d}\! \xi\right)^{\frac{1}{q_0}}+CR^{-\sigma-\alpha} \operatorname{Tail}(u-(u)_{B_{2R}(x)};x,2R),
\end{eqnarray*}
we conclude that \eqref{Maxpoint-Dinivmosub} holds whenever $R\leq \min\{\bar{R}_1,\bar{R}_2,\bar{R}_3,\frac{\tilde{R}}{2}\}$.
Following the last step of the proof of Theorem \ref{MaxsmallBMO} exactly, we can remove this restriction on $R$. Consequently, the pointwise maximal function estimate \eqref{Maxpoint-Dinivmosub} holds uniformly for all $\alpha\in[0,1]$, with a positive constant $C$ depending only on $n,\,p,\,s,\,\Lambda,\,m,\,a,\,\sigma,\,\omega(\cdot)$ and $\operatorname{diam}(\Omega)$.

Finally, the superquadratic case $p > 2$ remains to be proven.
We adopt a similar approach to that used for $2-\frac{1}{n}<p\leq 2$, but now with
$q_0=p-1$, and apply \eqref{max-diniVMO-potentialp>2} instead of \eqref{max-diniVMO-potentialp<=2}, which leads to the validity of \eqref{Maxpoint-Dinivmosup} for a positive constant $C\equiv C(n,p,s,\Lambda,m,a,\sigma,\omega(\cdot),\operatorname{diam}(\Omega))$. Hence, the proof of Theorem \ref{MaxDinivmo} is complete.
\end{proof}

This section concludes with a proof of pointwise estimates for the nonlocal fractional sharp maximal function of the gradient under the weaker assumption \eqref{Dinisup} on the partial map $x\mapsto\mathcal{A}(x,\cdot)$ than the Dini-H\"{o}lder regular condition given in \eqref{Diniholder}.

\begin{proof}[Proof of Theorem \ref{MaxDiniholder}]
We maintain the notation introduced in the proof of Theorem \ref{MaxDinivmo}, and define
$$H_i:=R_{i}^{-\alpha}\left( \fint_{\mathcal{B}_i}|Du-(Du)_{\mathcal{B}_i}|^{q_0}  \operatorname{d}\!\xi\right)^{\frac{1}{q_0}}+R_{i}^{-\alpha-\sigma-1}
\operatorname{Tail}(u-(u)_{\mathcal{B}_{i}};x,R_{i}).$$
Let $2-\frac{1}{n}<p\leq2$, then $\sigma_1'=1$. It is easy to verify that the condition \eqref{Dinisup} implies that
$\omega(r)$ is Dini-$\operatorname{VMO}$ regular. In this subquadratic case, Theorem \ref{MaxDinivmo} and its proof process apply,
then multiplying both sides of \eqref{max-diniVMO-A_i} by $R_{i+1}^{-\alpha}$ yields that
\begin{eqnarray}\label{max-diniVMO-A_i2}
R_{i+1}^{-\alpha}\tilde{A}_{i+1}
&\leq& c_5\left(\frac{1}{2L}\right)^{\alpha_M-\alpha}R_{i}^{-\alpha}\tilde{A}_i+c_6(2L)^{\frac{n}{q_0}+\alpha}
\left(\frac{\omega(R_i)}{R_i^{\alpha}}+R_i^{\sigma_2-\alpha}\right)\left( \fint_{\mathcal{B}_i}|Du|^{q_0}  \operatorname{d}\!\xi\right)^{\frac{1}{q_0}}\nonumber\\
&&+CL^{\frac{n}{q_0}+\alpha}\left[\frac{\omega(R_i)}{R_i^\alpha}R_i^{a-\frac{1}{p-1}+\sigma} + R_i^{\sigma_3+\sigma-\frac{1}{p-1}-\alpha}\right]
R_i^{-1-\sigma}
\operatorname{Tail}(u-(u)_{\mathcal{B}_{i}};x,R_i)\nonumber\\
&&   +CL^{\frac{n}{q_0}+\alpha} \left[
\left[R_i^{\alpha(2-p)}\frac{|\mu|(\mathcal{B}_{i})}{R_i^{n-1+\alpha}}\right]^{\frac{1}{p-1}}+\chi_{\{p<2\}}
\left[\frac{|\mu|(\mathcal{B}_{i})}{R_i^{n-1+\alpha}}\right]
\left(\fint_{\mathcal{B}_{i}} |Du|\operatorname{d}\! \xi\right)^{2-p}\right].
\end{eqnarray}
We choose $L\equiv L(n,p,s,\Lambda,m,a,\tilde{\alpha},\operatorname{diam}(\Omega))$ sufficiently large such that
$ c_5\left(\frac{1}{2L}\right)^{\alpha_M-\tilde{\alpha}}\leq\frac{1}{8}$, and restrict the positive radius $R\leq1$. Then applying \eqref{Maxpoint-Dinivmosub} with $\alpha=1$, together with Young's inequality and under the conditions imposed by Theorem \ref{MaxDiniholder}, we deduce that
\begin{eqnarray}\label{max-diniholder-1}
R_{i+1}^{-\alpha}\tilde{A}_{i+1}   &\leq& \frac{1}{8}R_{i}^{-\alpha}\tilde{A}_{i} +CR^{-\alpha} \left[\left(\fint_{B_{R}(x)} |Du|^{q_0} \operatorname{d}\! \xi\right)^{\frac{1}{q_0}}+R^{-1-\sigma}\operatorname{Tail}(u-(u)_{B_R(x)};x,R)\right]
\nonumber\\
&&+ C\left[{\bf{I}}^{\mu}_{1}(x,R)\right]^{\frac{1}{p-1}}+C\left[ {\bf M}_{1-\alpha}(\mu)(x, R)\right]^{\frac{1}{p-1}},
\end{eqnarray}
where the positive constant $C$ depending only on $n$, $p$, $s$, $\Lambda$, $m$, $a$, $\sigma$, $\tilde{\alpha}$, $\omega(\cdot)$, $\operatorname{diam}(\Omega)$ and $C_\omega$.
With respect to the estimate of the tail term, we employ \cite[Lemma 4.12]{CSYZ24} to derive
\begin{eqnarray*}
   &&  R_{i+1}^{-\alpha-\sigma-1}
\operatorname{Tail}(u-(u)_{\mathcal{B}_{i+1}};x,R_{i+1}) \\
   &\leq& c_9\left[\left(\frac{1}{2L}\right)^{\frac{1}{p-1}-\alpha-\sigma}+(2L)^{1-\alpha_m+\sigma+\alpha}R_i^{p'(1-s)}\right] R_{i}^{-\alpha-\sigma-1}
\operatorname{Tail}(u-(u)_{\mathcal{B}_{i}};x,R_{i})\\
&&+C R_i^{(1-s)p'-\sigma-\alpha} \left(\fint_{\mathcal{B}_{i}} |Du|^{q_0} \operatorname{d}\! \xi\right)^{\frac{1}{q_0}}+C R_i^{(1-s)p'-\sigma} \left[\frac{|\mu|(\mathcal{B}_{i})}{R_i^{n-1+\alpha}}\right]^{\frac{1}{p-1}}
\end{eqnarray*}
Due to the fact that $\alpha\leq\tilde{\alpha} <\frac{1}{p-1}-\sigma$, we further
choose $L\equiv L(n,p,s,\Lambda,\sigma,\tilde{\alpha})$ sufficiently large such that
$c_9\left(\frac{1}{2L}\right)^{\frac{1}{p-1}-\tilde{\alpha}-\sigma}\leq\frac{1}{16}$, and select a positive radius $\bar{R}\equiv\bar{R}(n,p,s,\Lambda,\sigma,\tilde{\alpha})\leq 1$ small enough to ensure that
$$c_9(2L)^{1-\alpha_m+\sigma+\alpha}R_i^{p'(1-s)}\leq c_9 (2L)^{1-\alpha_m+\sigma+\tilde{\alpha}}\bar{R}^{p'(1-s)}\leq\frac{1}{16}\,\,\,\mbox{for}\,\, R_i\leq\bar R.$$
Subsequently, using \eqref{Maxpoint-Dinivmosub} with $\alpha=1$ once again, we  calculate the following directly
\begin{eqnarray}\label{max-diniholder-2}
&&R_{i+1}^{-\alpha-\sigma-1}
\operatorname{Tail}(u-(u)_{\mathcal{B}_{i+1}};x,R_{i+1})\nonumber\\
&\leq& CR^{-\alpha} \left[\left(\fint_{B_{R}(x)} |Du|^{q_0} \operatorname{d}\! \xi\right)^{\frac{1}{q_0}}+R^{-1-\sigma}\operatorname{Tail}(u-(u)_{B_R(x)};x,R)\right]
\nonumber\\
&&+\frac{1}{8}R_{i}^{-\alpha-\sigma-1}
\operatorname{Tail}(u-(u)_{\mathcal{B}_{i}};x,R_{i})+ C\left[{\bf{I}}^{\mu}_{1}(x,R)\right]^{\frac{1}{p-1}}+C\left[ {\bf M}_{1-\alpha}(\mu)(x, R)\right]^{\frac{1}{p-1}},
\end{eqnarray}
where the positive constant $C$ depending only on $n$, $p$, $s$, $\Lambda$, $m$, $a$, $\sigma$, $\tilde{\alpha}$, $\omega(\cdot)$ and $\operatorname{diam}(\Omega)$.
Thus, a combination of \eqref{max-diniholder-1} and \eqref{max-diniholder-2} yields that
\begin{eqnarray*}
  H_{i+1}&\leq& \frac{1}{8}H_i+CR^{-\alpha} \left[\left(\fint_{B_{R}(x)} |Du|^{q_0} \operatorname{d}\! \xi\right)^{\frac{1}{q_0}}+R^{-1-\sigma}\operatorname{Tail}(u-(u)_{B_R(x)};x,R)\right]\\
  &&+ C\left[{\bf{I}}^{\mu}_{1}(x,R)\right]^{\frac{1}{p-1}}+C\left[ {\bf M}_{1-\alpha}(\mu)(x, R)\right]^{\frac{1}{p-1}}.
\end{eqnarray*}
We proceed by induction to obtain
\begin{eqnarray*}
  H_{i}&\leq& CR^{-\alpha} \left[\left(\fint_{B_{R}(x)} |Du|^{q_0} \operatorname{d}\! \xi\right)^{\frac{1}{q_0}}+R^{-1-\sigma}\operatorname{Tail}(u-(u)_{B_R(x)};x,R)\right]\\
  &&+ C\left[{\bf{I}}^{\mu}_{1}(x,R)\right]^{\frac{1}{p-1}}+C\left[ {\bf M}_{1-\alpha}(\mu)(x, R)\right]^{\frac{1}{p-1}}
\end{eqnarray*}
holds whenever $i\in \mathbb{N}$.
Finally, repeating the last part of the proof of Theorem \ref{MaxDinivmo}, we conclude that \eqref{MaxDiniholdersub} is valid.

For the superquadratic case $p>2$, we need to apply \eqref{Maxpoint-Dinivmosup} instead of \eqref{Maxpoint-Dinivmosub} and observe that
$$\left[\frac{|\mu|(\mathcal{B}_{i})}{R_i^{n-1+\alpha(p-1)}}\right]^{\frac{1}{p-1}}\leq C(n,p)\left[ {\bf M}_{1-\alpha(p-1)}(\mu)(x, R)\right]^{\frac{1}{p-1}}.$$
Thus, we can demonstrate that the estimate \eqref{MaxDiniholdersup} is correct by treating this case in a virtually similar manner as above.
This completes the proof of Theorem \ref{MaxDiniholder}.
\end{proof}

\section{Universal potential estimates}\label{section5}

We are now in a position to prove our main results regarding the universal potential estimates for the solution and its gradient of the mixed local and nonlocal nonlinear elliptic equation \eqref{eq1}.

\subsection{Oscillation estimates for solutions}

In this subsection, we present the proofs for Theorem \ref{OscPotenSolMea}, \ref{OscPotenSolBMO} and \ref{OscPotenSolDiniVMO}, which deal with oscillation estimates for solutions and establish a connection between the classical De Giorgi–Nash–Moser theory and nonlinear potential estimates.

\begin{proof}[Proof of Theorem \ref{OscPotenSolMea}]
Let $x,\,y\in B_{\frac{R}{8}}(x_0)$ and $r\leq \frac{R}{2}$. We define two sequences of shrinking balls, denoted by $\mathcal{B}_i:=B_{r_i}(x)$ and $\tilde{\mathcal{B}}_i:=B_{\frac{r_i}{2}}(x)$,
and set
\begin{equation*}
  E_i:=\left( \fint_{\mathcal{B}_i}|u-(u)_{\mathcal{B}_i}|^{q_0}  \operatorname{d}\!\xi\right)^{\frac{1}{q_0}}+\operatorname{Tail}(u-(u)_{\mathcal{B}_i};x,r_i)\,\, \mbox{and}\,\, \mathcal{K}_i:=|(u)_{\mathcal{B}_i}|,
\end{equation*}
where the radius $r_i:=\frac{r}{(4L)^{i}}$ and the constant $L\geq 1$ will be determined later.
According to the selection of the ball sequences,
it is obvious that $\mathcal{B}_{i+1}\subset\tilde{\mathcal{B}}_i\subset\mathcal{B}_i$ for all $i\in \mathbb{N}$.
Applying Poincar\'{e}'s inequality and combining Lemma \ref{Du-morrey} with \cite[Lemma 4.12]{CSYZ24}, Young's inequality and the Caccioppoli type inequality \eqref{CCPu} for $u$, we obtain
\begin{eqnarray}\label{OscSol-1-1}
  E_{i+1} &\leq& C\left( \frac{r_{i+1}}{r_i} \right)^{\alpha_m}r_i \left[\left(\fint_{\tilde{\mathcal{B}}_i} |Du|^{q_0} \operatorname{d}\! x\right)^{\frac{1}{q_0}}+\frac{1}{r_i}\operatorname{Tail}(u-(u)_{\tilde{\mathcal{B}}_i};x,\frac{r_i}{2})\right]\nonumber\\
  &&+
  C\left( \frac{r_{i}}{r_{i+1}} \right)^{\frac{n}{p-1}}r_i^{(1-s)p'+1} \left(\fint_{\tilde{\mathcal{B}}_i} |Du|^{q_0} \operatorname{d}\! x\right)^{\frac{1}{q_0}}\nonumber\\
&&   +C\left(\frac{r_i}{r_{i+1}}\right)^{\max\{\frac{n}{p-1},1-\alpha_m\}} \left\{
\left[\frac{|\mu|(\mathcal{B}_i)}{r_i^{n-p}}\right]^{\frac{1}{p-1}}+\chi_{\{p<2\}}
\left[\frac{|\mu|(\mathcal{B}_i)}{r_i^{n-p}}\right]
\left(r_i\fint_{\tilde{\mathcal{B}}_i} |Du|\operatorname{d}\! x\right)^{2-p}\right\}\nonumber\\
&\leq& C_0\left[\left( \frac{1}{4L} \right)^{\alpha_m}+\epsilon+\left( 4L \right)^{\frac{n}{p-1}}r^{(1-s)p'} \right] E_i
  +C_{L,\epsilon}
\left[\frac{|\mu|(\mathcal{B}_i)}{r_i^{n-p}}\right]^{\frac{1}{p-1}}.
\end{eqnarray}
We now choose $L\equiv L(n,p,s,\Lambda,\operatorname{diam}(\Omega))$ large enough
and $\epsilon\equiv \epsilon(n,p,s,\Lambda,\operatorname{diam}(\Omega))$ sufficiently small such that $C_0\left[\left( \frac{1}{4L} \right)^{\alpha_m}+\epsilon\right] \leq\frac{1}{4}$. And then selecting a radius $R_0\equiv R_0(n,p,s,\Lambda,\operatorname{diam}(\Omega))$
small enough to guarantee that
$ C_0\left( 4L \right)^{\frac{n}{p-1}}r^{(1-s)p'}\leq C_0\left( 4L \right)^{\frac{n}{p-1}}(\frac{R}{2})^{(1-s)p'} \leq \frac{1}{4}$
for $R\leq R_0$.
Next, we proceed exactly as in \eqref{max-diniVMO-1}-\eqref{max-diniVMO-2-1} to derive
\begin{eqnarray}\label{OscSol-1-2}
  \mathcal{K}_{j+1} &\leq& C\left( \fint_{B_r(x)}|u|^{q_0}  \operatorname{d}\!\xi\right)^{\frac{1}{q_0}}+C\operatorname{Tail}(u-(u)_{B_r(x)};x,r) +C r^{\alpha} \sum_{i=0}^{j-1} \left[\frac{|\mu|(\mathcal{B}_i)}{r_i^{n-p+\alpha(p-1)}}\right]^{\frac{1}{p-1}} \nonumber \\
   &\leq& C\left( \fint_{B_r(x)}|u|^{q_0}  \operatorname{d}\!\xi\right)^{\frac{1}{q_0}}+C\operatorname{Tail}(u-(u)_{B_r(x)};x,r) +Cr^\alpha {\bf{W}}^{\mu}_{1-\frac{\alpha(p-1)}{p},p}(x,r)
\end{eqnarray}
for all $j\in \mathbb{N}$, where the positive constant $C\equiv C(n,p,s,\Lambda,\operatorname{diam}(\Omega))$.
Let $j\rightarrow +\infty$ and note that if $u$ solves \eqref{eq1}, then $u-g$ remains a solution to the same equation for any constant $g$, thus we have
\begin{equation}\label{OscSol-1}
  |u(x)-g|\leq C \left( \fint_{B_r(x)}|u-g|^{q_0}  \operatorname{d}\!\xi\right)^{\frac{1}{q_0}}+C\operatorname{Tail}(u-(u)_{B_r(x)};x,r) +Cr^\alpha {\bf{W}}^{\mu}_{1-\frac{\alpha(p-1)}{p},p}(x,r)
\end{equation}
for almost everywhere $x\in B_{\frac{R}{8}}(x_0)$.
Similarly, replacing point $x$ with point $y$ leads to
\begin{equation}\label{OscSol-2}
  |u(y)-g|\leq C \left( \fint_{B_r(y)}|u-g|^{q_0}  \operatorname{d}\!\xi\right)^{\frac{1}{q_0}}+C\operatorname{Tail}(u-(u)_{B_r(y)};y,r) +Cr^\alpha {\bf{W}}^{\mu}_{1-\frac{\alpha(p-1)}{p},p}(y,r)
\end{equation}
for almost everywhere $y\in B_{\frac{R}{8}}(x_0)$.

Let $r:=\frac{|x-y|}{2}$, it is straightforward to verify that $B_r(y)\subset B_{3r}(x)\subset B_{\frac{3R}{8}}(x)\subset B_{\frac{3R}{4}}(x)\subset B_{R}(x_0)$ and $B_{\frac{3R}{4}}(y)\subset B_{R}(x_0)$.
Fixing $g:=(u)_{B_{3r}(x)}$, incorporating \eqref{OscSol-1} and \eqref{OscSol-2}, and utilizing Theorem \ref{Maxmeasure} along with the Caccioppoli type inequality \eqref{CCPu}, we deduce that
\begin{eqnarray}\label{OscSol-3}
&& |u(x)-u(y)|\nonumber\\
 &\leq & Cr^\alpha\mathfrak{M}^{\sharp}_{\alpha,q_0}(u)(x,\frac{3R}{8})+
 +Cr^\alpha\left[{\bf{W}}^{\mu}_{1-\frac{\alpha(p-1)}{p},p}(x,R)
 +{\bf{W}}^{\mu}_{1-\frac{\alpha(p-1)}{p},p}(y,R)\right]\nonumber\\
&\leq& Cr^\alpha\left[ {\bf M}_{p-\alpha(p-1)}(\mu)(x,\frac{3R}{8})\right]^{\frac{1}{p-1}}+Cr^\alpha\left[{\bf{W}}^{\mu}_{1-\frac{\alpha(p-1)}{p},p}(x,R)
 +{\bf{W}}^{\mu}_{1-\frac{\alpha(p-1)}{p},p}(y,R)\right]\nonumber\\
&&+C\left(\frac{r}{R}\right)^{\alpha}\left[R \left(\fint_{B_{\frac{3R}{8}}(x)} |Du|^{q_0} \operatorname{d}\! \xi\right)^{\frac{1}{q_0}}+\operatorname{Tail}(u-(u)_{B_\frac{3R}{8}(x)};x,R)\right]\nonumber\\
&\leq& Cr^\alpha\left[ {\bf M}_{p-\alpha(p-1)}(\mu)(x,\frac{3R}{4})\right]^{\frac{1}{p-1}}+Cr^\alpha\left[{\bf{W}}^{\mu}_{1-\frac{\alpha(p-1)}{p},p}(x,R)
 +{\bf{W}}^{\mu}_{1-\frac{\alpha(p-1)}{p},p}(y,R)\right]\nonumber\\
&&+C\left(\frac{r}{R}\right)^{\alpha}\left[ \left(\fint_{B_{R}(x_0)} |u|^{q_0} \operatorname{d}\! x\right)^{\frac{1}{q_0}}+\operatorname{Tail}(u-(u)_{B_R(x_0)};x_0,R)\right]
\end{eqnarray}
is valid for almost everywhere $x,\,y\in B_{\frac{R}{8}}(x_0)$, where the positive constant $C\equiv C(n,p,s,\Lambda, \tilde{\alpha}, \operatorname{diam}(\Omega))$.
As established in Lemma \ref{maxf-potential}, we have
\begin{equation}\label{OscSol-4}
  \left[ {\bf M}_{p-\alpha(p-1)}(\mu)(x,\frac{3R}{4})\right]^{\frac{1}{p-1}}\leq C(n,p,\tilde{\alpha}){\bf{W}}^{\mu}_{1-\frac{\alpha(p-1)}{p},p}(x,R).
\end{equation}
Finally, we insert \eqref{OscSol-4} into \eqref{OscSol-3}, and therefore
complete the proof of Theorem \ref{OscPotenSolMea}.
\end{proof}

With Theorem \ref{OscPotenSolMea} at our disposal, we turn to proving Corollary \ref{PotenmaxMea}, which concerns a pointwise estimate for the truncated Hardy–Littlewood maximal function of solutions.

\begin{proof}[Proof of Corollary \ref{PotenmaxMea}]
A direct calculation combined with H\"{o}lder's inequality shows that
\begin{eqnarray*}
  {\bf M}(u)(x,R) \leq {\bf M}(u)(x,\frac{R}{2})+C(n)\left(\fint_{B_{R}(x)} |u|^{q_0} \operatorname{d}\! \xi\right)^{\frac{1}{q_0}},
\end{eqnarray*}
then it suffices to prove that
\begin{equation}\label{PotenmaxMea-est}
  {\bf M}(u)(x,\frac{R}{2})
   \leq C \left[\left(\fint_{B_{R}(x)} |u|^{q_0} \operatorname{d}\! \xi\right)^{\frac{1}{q_0}}+\operatorname{Tail}(u-(u)_{B_R(x)};x,R)\right]
   +C{\bf{W}}^{\mu}_{1,p}(x,R).
\end{equation}
For any $0<\rho\leq\frac{R}{2}$, with the notation used in the proof of Theorem \ref{OscPotenSolMea}, there exists an $i\in\mathbb{N}$ such that $r_{i+1}<\rho\leq r_i$, and then
\begin{equation*}
  \fint_{B_{\rho}(x)} |u| \operatorname{d}\! \xi\leq (4L)^n\left(E_i+\mathcal{K}_i\right).
\end{equation*}
Hence, applying \eqref{OscSol-1-1} and \eqref{OscSol-1-2}, and taking into account the arbitrariness of $\rho$, we conclude that \eqref{PotenmaxMea-est} is valid,  thereby completing the proof of Corollary \ref{PotenmaxMea}.
\end{proof}

In order to increase the size of the exponent $\tilde{\alpha}$, we further assume that the coefficient is Dini-$\operatorname{VMO}$ regular, thereby providing a proof of Theorem \ref{OscPotenSolBMO}.

\begin{proof}[Proof of Theorem \ref{OscPotenSolBMO}]
Without loss of generality, we may assume that $\alpha\in [\frac{\alpha_m}{2},\tilde{\alpha}]$. In fact, when $\alpha\in[0,\frac{\alpha_m}{2}]$ is restricted, \eqref{OscPotenSolBMO-est} is a slight modification of Theorem \ref{OscPotenSolMea}. At this point, we replace the definition of $E_i$ by
$$E_i:=\left( \fint_{\mathcal{B}_i}|u-(u)_{\mathcal{B}_i}|^{q_0}  \operatorname{d}\!\xi\right)^{\frac{1}{q_0}}.$$
Then similar to the argument in \eqref{OscSol-1-1}, without using \cite[Lemma 4.12]{CSYZ24}, we obtain
\begin{eqnarray*}
  E_{i+1}
&\leq& C_0\left[\left( \frac{1}{4L} \right)^{\alpha_m}+\epsilon \right] E_i
  +C_{L,\epsilon}
\left[\frac{|\mu|(\mathcal{B}_i)}{r_i^{n-p}}\right]^{\frac{1}{p-1}}+C_L\operatorname{Tail}(u-(u)_{\mathcal{B}_i};x,r_i).
\end{eqnarray*}
Selecting $L\equiv L(n,p,s,\Lambda,\operatorname{diam}(\Omega))$ large enough
and $\epsilon\equiv\epsilon(n,p,s,\Lambda,\operatorname{diam}(\Omega))$ small enough such that $C_0\left[\left( \frac{1}{4L} \right)^{\alpha_m}+\epsilon\right] \leq\frac{1}{2}$. Here, we do not need to impose any constraints on the radius $R$, and proceed exactly as in \eqref{max-diniVMO-1}-\eqref{max-diniVMO-2-1}, together with Theorem \ref{MaxsmallBMO} and \eqref{CCPu} to derive
\begin{eqnarray*}
  \mathcal{K}_{j+1} &\leq& C\left( \fint_{B_r(x)}|u|^{q_0}  \operatorname{d}\!\xi\right)^{\frac{1}{q_0}}+Cr^\alpha \sum_{i=0}^{j-1} r_i^\sigma r_i^{-\alpha-\sigma}\operatorname{Tail}(u-(u)_{B_r(x)};x,r) +C r^{\alpha} \sum_{i=0}^{j-1} \left[\frac{|\mu|(\mathcal{B}_i)}{r_i^{n-p+\alpha(p-1)}}\right]^{\frac{1}{p-1}} \nonumber \\
   &\leq& C\left( \fint_{B_r(x)}|u|^{q_0}  \operatorname{d}\!\xi\right)^{\frac{1}{q_0}}+Cr^{-\sigma}\operatorname{Tail}(u-(u)_{B_r(x)};x,r) +Cr^\alpha\left[ {\bf M}_{p-\alpha(p-1)}(\mu)(x,R)\right]^{\frac{1}{p-1}}\\
   &&+Cr^\alpha {\bf{W}}^{\mu}_{1-\frac{\alpha(p-1)}{p},p}(x,r)
\end{eqnarray*}
for all $j\in \mathbb{N}$. Subsequently,  in analogy to the proof of Theorem \ref{OscPotenSolMea}, it can be shown that \eqref{OscPotenSolBMO-est} holds uniformly in $\alpha\in[0,\frac{\alpha_m}{2}]$.

Alternatively, if $\alpha\in [\frac{\alpha_m}{2},\tilde{\alpha}]$, then a combination of \cite[Theorem 2.5]{DS84} and H\"{o}lder's inequality with Theorem \ref{MaxsmallBMO} yields that
\begin{eqnarray*}
  |u(x)-u(y)|
   &\leq& \frac{C}{\alpha_m} \left[\mathfrak{M}^{\sharp}_{\alpha,q_0}(u)(x,\frac{3R}{8})+\mathfrak{M}^{\sharp}_{\alpha,q_0}(u)(y,\frac{3R}{8})\right] |x-y|^\alpha\\
  &\leq& \frac{C}{\alpha_m}\left[ {\bf M}_{p-\alpha(p-1)}(\mu)(x,\frac{3R}{8})+{\bf M}_{p-\alpha(p-1)}(\mu)(y,\frac{3R}{8})\right]^{\frac{1}{p-1}}|x-y|^\alpha\nonumber\\
   &&+\frac{C}{\alpha_m}\left(\frac{|x-y|}{R}\right)^\alpha \left[R\left(\fint_{B_{\frac{3R}{8}}(x)} |Du|^{q_0} \operatorname{d}\! \xi\right)^{\frac{1}{q_0}}+R\left(\fint_{B_{\frac{3R}{8}}(y)} |Du|^{q_0} \operatorname{d}\! \xi\right)^{\frac{1}{q_0}}\right.\\
&&\left. +R^{-\sigma}\operatorname{Tail}(u-(u)_{B_\frac{3R}{8}(x)};x,\frac{3R}{8})  +R^{-\sigma}\operatorname{Tail}(u-(u)_{B_\frac{3R}{8}(y)};y,\frac{3R}{8})\right]
\end{eqnarray*}
for almost everywhere $x,\,y \in B_{\frac{R}{8}}(x_0)$.
Finally, we conclude that Theorem \ref{OscPotenSolBMO} holds by employing the Caccioppoli-type inequalities \eqref{CCPu} and \eqref{OscSol-4} in an analysis analogous to that presented in \eqref{OscSol-3}.
\end{proof}

Ultimately, we impose the Dini-$\operatorname{VMO}$ regularity condition on the partial map $x\mapsto\mathcal{A}(x,\cdot)$ to prove the oscillation estimates for solutions where the exponent $\tilde{\alpha}$ can equal $1$.

\begin{proof}[Proof of Theorem \ref{OscPotenSolDiniVMO}]
The proof proceeds similarly to that of Theorem \ref{OscPotenSolBMO}, but here we apply Theorem \ref{MaxDinivmo} instead of Theorem \ref{MaxsmallBMO}.
In particular, for the case of $2-\frac{1}{n}<p<2$, we further utilize the following fact that
\begin{eqnarray*}
\left[ {\bf M}_{p-\alpha(p-1)}(\mu)(x,\frac{3R}{8})\right]^{\frac{1}{p-1}}\leq C(n,p){\bf{W}}^{\mu}_{1-\frac{\alpha(p-1)}{p},p}(x,\frac{R}{2})\leq C(n,p) \left[{\bf{I}}^{\mu}_{p-\alpha(p-1)}(x,R)\right]^{\frac{1}{p-1}}.
\end{eqnarray*}
Therefore, we can complete the proof of Theorem \ref{OscPotenSolDiniVMO}.
\end{proof}

\subsection{Oscillation estimates for gradient}

In the final subsection, we focus on proving the oscillation estimate for the gradient, as stated in Theorem \ref{OscPotenGrad}.

\begin{proof}[Proof of Theorem \ref{OscPotenGrad}]
We define a sequence of shrinking balls, denoted by $\mathcal{B}_i:=B_{r_i}(x)$, with radius $r_i:=\frac{r}{(2L)^{i}}$ for $i\in \mathbb{N}$ and $r\leq \frac{R}{2}$, such that $B_r(x)\subset B_R(x_0)$ for all $x\in B_{\frac{R}{4}}(x_0)$.
In the case of $2-\frac{1}{n}<p\leq2$, we apply \eqref{Maxpoint-Dinivmosub} with $\alpha=1$, as established in Theorem \ref{MaxDinivmo}, to derive
\begin{eqnarray}\label{Oscgrad-1}
  &&\left(\fint_{\mathcal{B}_i}|Du|^{q_0}  \operatorname{d}\!\xi\right)^{\frac{1}{q_0}}
+r_i^{-1-\sigma}
\operatorname{Tail}(u-(u)_{\mathcal{B}_{i}};x,r_i)\nonumber\\
&\leq&C\left[\left(\fint_{B_{r}(x)} |Du|^{q_0} \operatorname{d}\! \xi\right)^{\frac{1}{q_0}}+r^{-1-\sigma}\operatorname{Tail}(u-(u)_{B_r(x)};x,r)
+ \left[{\bf{I}}^{\mu}_{1}(x,r)\right]^{\frac{1}{p-1}}\right]:=M(x,r).
\end{eqnarray}
With the notation used in the proof of Theorem \ref{MaxDinivmo}, it follows from \eqref{max-diniVMO-2-1} and \eqref{Oscgrad-1} that
\begin{eqnarray}\label{Oscgrad-1-1}
  \tilde{K}_{j+1}
   &\leq&C\left( \fint_{B_r(x)}|Du-(Du)_{B_r(x)}|^{q_0}+|Du-G|^{q_0}  \operatorname{d}\!\xi\right)^{\frac{1}{q_0}}\nonumber\\
   &&+Cr^\alpha\sum_{i=0}^{j}
\left(\frac{\omega(r_i)}{r_i^\alpha}+r_i^{\sigma_2-\alpha}\right)\left(M(x,r)+|G|\right)
+Cr^\alpha\sum_{i=0}^{j}\left(\frac{\omega(r_i)}{r_i^{\alpha}}
+r_i^{\sigma_3+\sigma-\frac{1}{p-1}-\alpha}\right)M(x,r)\nonumber\\
&&   +Cr^\alpha\sum_{i=0}^{j} \left\{
\left[\frac{|\mu|(\mathcal{B}_{i})}{r_i^{n-1+\alpha(p-1)}}\right]^{\frac{1}{p-1}}+\chi_{\{p<2\}}
\left[\frac{|\mu|(\mathcal{B}_{i})}{r_i^{n-1+\alpha}}\right]
\left[M(x,r)\right]^{2-p}\right\}
\end{eqnarray}
for any $j\in\mathbb{N}$. According to the conditions assumed in Theorem \ref{OscPotenGrad}, a direct calculation shows that
\begin{equation*}
  \sum_{i=0}^{j}
\left(\frac{\omega(r_i)}{r_i^\alpha}+r_i^{\sigma_2-\alpha}+r_i^{\sigma_3+\sigma-\frac{1}{p-1}-\alpha}\right)
\leq C\int_0^{2r}\left(\frac{\omega(\rho)}{\rho^{\tilde{\alpha}}}
+\rho^{\sigma_2-\tilde{\alpha}}+\rho^{\sigma_3+\sigma-\frac{1}{p-1}-\tilde{\alpha}}
  \right)\frac{\operatorname{d}\! \rho}{\rho}\leq C
\end{equation*}
and
\begin{equation*}
  \sum_{i=0}^{j}
\left[\frac{|\mu|(\mathcal{B}_{i})}{r_i^{n-1+\alpha(p-1)}}\right]^{\frac{1}{p-1}}\leq C\sum_{i=0}^{j}\left[
\frac{|\mu|(\mathcal{B}_{i})}{r_i^{n-1+\alpha}}\right]^{\frac{1}{p-1}} \leq C\left[\sum_{i=0}^{j}
\frac{|\mu|(\mathcal{B}_{i})}{r_i^{n-1+\alpha}}\right]^{\frac{1}{p-1}} \leq C\left[{\bf{I}}^{\mu}_{1-\alpha}(x,2r)\right]^{\frac{1}{p-1}}
\end{equation*}
hold for a positive constant $C$ depending only on $n$, $p$, $s$, $\Lambda$, $m$, $a$, $\sigma$, $\tilde{\alpha}$ and $\operatorname{diam}(\Omega)$.
Substituting the last two estimates into \eqref{Oscgrad-1-1} and
letting $j\rightarrow +\infty$, we conclude from Young's inequality and the Lebesgue differentiation theorem that
\begin{eqnarray}\label{Oscgrad-2}
 |Du(x)-G|&=& \lim_{j\rightarrow +\infty} \tilde{K}_{j+1}\nonumber\\
 &\leq & C\left( \fint_{B_r(x)}|Du-(Du)_{B_r(x)}|^{q_0}+|Du-G|^{q_0}  \operatorname{d}\!\xi\right)^{\frac{1}{q_0}}\nonumber\\
 &&+Cr^\alpha\left[{\bf{I}}^{\mu}_{1-\alpha}(x,R)\right]^{\frac{1}{p-1}}+Cr^\alpha
  \left(M(x,r)+|G|\right)
\end{eqnarray}
for almost everywhere $x\in B_{\frac{R}{4}}(x_0)$, where the positive constant $C\equiv C(n,p,s,\Lambda, m, a, \sigma, \tilde{\alpha},\omega(\cdot),\operatorname{diam}(\Omega))$.
Repeating the same steps as above for the point $y\in B_{\frac{R}{4}}(x_0)$ to obtain
\begin{eqnarray}\label{Oscgrad-3}
 |Du(y)-G|
 &\leq & C\left( \fint_{B_r(y)}|Du-(Du)_{B_r(y)}|^{q_0}+|Du-G|^{q_0}  \operatorname{d}\!\xi\right)^{\frac{1}{q_0}}\nonumber\\
 &&+Cr^\alpha\left[{\bf{I}}^{\mu}_{1-\alpha}(y,R)\right]^{\frac{1}{p-1}}+Cr^\alpha
  \left(M(y,r)+|G|\right).
\end{eqnarray}
Let $r:=\frac{|x-y|}{2}$, it is straightforward to verify that $B_r(y)\subset B_{3r}(x)\subset B_{\frac{3R}{4}}(x)\subset B_{R}(x_0)$ and $B_{\frac{3R}{4}}(y)\subset B_{R}(x_0)$.
We now fix $G:=(Du)_{B_{3r}(x)}$, and apply the condition $\sigma<p'(1-s)$ and the estimate \eqref{Maxpoint-Dinivmosub} with $\alpha=1$ again to deduce that
\begin{eqnarray}\label{Oscgrad-4}
  &&M(x,r)+M(y,r)+|(Du)_{B_{3r}(x)}|\nonumber \\
   &\leq& C \left(\fint_{B_{\frac{3R}{4}}(x)} |Du|^{q_0} \operatorname{d}\! \xi\right)^{\frac{1}{q_0}}+CR^{-1-\sigma}\operatorname{Tail}(u-(u)_{B_\frac{3R}{4}(x)};x,\frac{3R}{4})
+C \left[{\bf{I}}^{\mu}_{1}(x,\frac{3R}{4})\right]^{\frac{1}{p-1}}\nonumber\\
&& +C \left(\fint_{B_{\frac{3R}{4}}(y)} |Du|^{q_0} \operatorname{d}\! \xi\right)^{\frac{1}{q_0}}+CR^{-1-\sigma}\operatorname{Tail}(u-(u)_{B_\frac{3R}{4}(y)};y,\frac{3R}{4})
+C \left[{\bf{I}}^{\mu}_{1}(y,\frac{3R}{4})\right]^{\frac{1}{p-1}} \nonumber\\
&\leq& CR^{-\alpha}\left[ \left(\fint_{B_{R}(x_0)} |Du|^{q_0} \operatorname{d}\! x\right)^{\frac{1}{q_0}}+R^{-1-\sigma}\operatorname{Tail}(u-(u)_{B_R(x_0)};x_0,R)\right]\nonumber\\
&&+C\left[{\bf{I}}^{\mu}_{1-\alpha}(x,R)+{\bf{I}}^{\mu}_{1-\alpha}(y,R)\right]^{\frac{1}{p-1}}.
\end{eqnarray}
Combining \eqref{Oscgrad-2}-\eqref{Oscgrad-4}, and employing
the pointwise estimate \eqref{MaxDiniholdersub} for the nonlocal fractional sharp
maximal function of the gradient, we arrive at
\begin{eqnarray}\label{Oscgrad-5}
&& |Du(x)-Du(y)|\nonumber\\
 &\leq & Cr^\alpha\mathfrak{M}^{\sharp}_{\alpha,q_0}(Du)(x,\frac{3R}{4})+
 +Cr^\alpha\left[{\bf{I}}^{\mu}_{1-\alpha}(x,R)+{\bf{I}}^{\mu}_{1-\alpha}(y,R)\right]^{\frac{1}{p-1}}\nonumber\\
 &&+C\left(\frac{r}{R}\right)^{\alpha}\left[ \left(\fint_{B_{R}(x_0)} |Du|^{q_0} \operatorname{d}\! x\right)^{\frac{1}{q_0}}+R^{-1-\sigma}\operatorname{Tail}(u-(u)_{B_R(x_0)};x_0,R)\right]\nonumber\\
&\leq& Cr^\alpha\left[ {\bf M}_{1-\alpha}(\mu)(x,\frac{3R}{4})\right]^{\frac{1}{p-1}}+Cr^\alpha\left[{\bf{I}}^{\mu}_{1-\alpha}(x,R)
+{\bf{I}}^{\mu}_{1-\alpha}(y,R)\right]^{\frac{1}{p-1}}\nonumber\\
&&+C\left(\frac{r}{R}\right)^{\alpha}\left[ \left(\fint_{B_{R}(x_0)} |Du|^{q_0} \operatorname{d}\! x\right)^{\frac{1}{q_0}}+R^{-1-\sigma}\operatorname{Tail}(u-(u)_{B_R(x_0)};x_0,R)\right]
\end{eqnarray}
for almost all $x,\,y\in B_{\frac{R}{4}}(x_0)$, where the positive constant $C\equiv C(n,p,s,\Lambda, m, a, \sigma, \tilde{\alpha},\omega(\cdot),\operatorname{diam}(\Omega))$.
Finally, in terms of Lemma \ref{maxf-potential}, we utilize the fact that
$${\bf M}_{1-\alpha}(\mu)(x,\frac{3R}{4})\leq C(n,\tilde{\alpha})\,{\bf{I}}^{\mu}_{1-\alpha}(x,R)$$
to bound the first term on the right side of \eqref{Oscgrad-5}.
Therefore, we conclude that the result \eqref{OscPotenGradsub} is valid.

In the sequel, we consider the case where $p> 2$.
Treating it in exactly the same manner as $p\leq2$,
but this time $\sigma_1'=\frac{2}{p}$ and $q_0=p-1$, and using \eqref{Maxpoint-Dinivmosup} instead of
\eqref{Maxpoint-Dinivmosub} and \eqref{MaxDiniholdersup} instead of \eqref{MaxDiniholdersub}.
Thus, we redefine $M(x,r)$ as follows
\begin{equation*}
   M(x,r):=C\left[\left(\fint_{B_{r}(x)} |Du|^{q_0} \operatorname{d}\! \xi\right)^{\frac{1}{q_0}}+r^{-1-\sigma}\operatorname{Tail}(u-(u)_{B_r(x)};x,r)
+ {\bf{W}}^{\mu}_{\frac{1}{p},p}(x,r)\right],
\end{equation*}
and recalculate
\begin{equation*}
  \sum_{i=0}^{j}
\left[\frac{|\mu|(\mathcal{B}_{i})}{r_i^{n-1+\alpha(p-1)}}\right]^{\frac{1}{p-1}}\leq C{\bf{W}}^{\mu}_{1-\frac{(1+\alpha)(p-1)}{p},p}(x,R).
\end{equation*}
Furthermore, combining the fact ${\bf{W}}^{\mu}_{\frac{1}{p},p}(x,\frac{3R}{4})\leq C{\bf{W}}^{\mu}_{1-\frac{(1+\alpha)(p-1)}{p},p}(x,R)$ with the inequality
$$\left[{\bf M}_{1-\alpha(p-1)}(\mu)(x,\frac{3R}{4})\right]^{\frac{1}{p-1}}\leq C(n,p,\tilde{\alpha})\,{\bf{W}}^{\mu}_{1-\frac{(1+\alpha)(p-1)}{p},p}(x,R)$$
established in Lemma \ref{maxf-potential}, we deduce that the estimate \eqref{OscPotenGradsup} holds uniformly in $\alpha\in[0,\tilde{\alpha}]$ for a positive constant $C\equiv C(n,p,s,\Lambda, m, a, \sigma, \tilde{\alpha},\omega(\cdot),\operatorname{diam}(\Omega))$. Therefore, the proof of Theorem \ref{OscPotenGrad} is complete.
\end{proof}

\section*{Acknowledgments} This work is supported by the National Natural Science Foundation of China (NSFC Grant No.12571103 and No.12401122), Young Scientific and Technological Talents (Level Three) in Tianjin, the Fundamental Research Funds for the Central Universities (Grant No. 2682024CX028), and Sichuan Natural Science Foundation Youth Fund Project (Grant No.2025ZNSFSC0799).

\section*{Data availability} Data sharing is not applicable to this article as obviously no datasets were generated or analyzed during the current study.

\section*{Conflict of interest} Author states no conflict of interest.

\bibliography{bibliography}

\begin{thebibliography}{99}
\bibitem{BM}\label{BM} L.~Beck and G.~Mingione, \textit{Lipschitz Bounds and Nonuniform Ellipticity}, Commun. Pure Appl. Math. 73 (2020), 944--1033.

\bibitem{BDVV22} \label{BDVV22} S. Biagi, S. Dipierro, E. Valdinoci and E. Vecchi, \textit{Mixed local and nonlocal elliptic operators: regularity and maximum principles}, Comm. Partial Differential Equations 47 (2022), 585--629.

\bibitem{BVDV21} \label{BVDV21} S. Biagi, E. Vecchi, S. Dipierro and E. Valdinoci, \textit{Semilinear elliptic equations involving mixed local and nonlocal operators}, Proc. Roy. Soc. Edinburgh Sect. A 151 (2021), 1611--1641.

\bibitem{BdC13} \label{BdC13} D. Blazevski and D. del-Castillo-Negrete, \textit{Local and nonlocal anisotropic transport in reversed shear magnetic fields: shearless Cantori and nondiffusive transport}, Phys. Rev. E 87 (2013), 063106.

\bibitem{BKL24} \label{BKL24} S. Byun, D. Kumar and H. Lee, \textit{Global gradient estimates for the mixed local and nonlocal problems with measurable nonlinearities}, Calc. Var. Partial Differential Equations 63 (2024), Paper No. 27, 48 pp.

\bibitem{BS23} \label{BS23} S. Byun and K. Song, \textit{Mixed local and nonlocal equations with measure data}, Calc. Var. Partial Differential Equations 62 (2023), Paper No. 14, 35 pp.

\bibitem{CiSc} \label{CiSc} A.~Cianchi and S.~Schwarzacher, \textit{Potential estimates for the p-Laplace system with data in divergence form}, J. Differential Equations 265 (2018), 478--499.

\bibitem{CSYZ24} \label{CSYZ24} I. Chlebicka, K. Song, Y. Youn and A. Zatorska-Goldstein, \textit{Riesz potential estimates for mixed local-nonlocal problems with measure data}, arXiv:2401.04549 (2024).

\bibitem{CV17} \label{CV17} D. Cao and I. Verbitsky, \textit{Nonlinear elliptic equations and intrinsic potentials of Wolff type}, J. Funct. Anal. 272 (2017), 112-165.

\bibitem{DKM14} \label{DKM14} P. Daskalopoulos, T. Kuusi and G.~Mingione, \textit{
 Borderline estimates for fully nonlinear elliptic equations}, Comm. Partial Differential Equations 39 (2014), 574--590.

\bibitem{DKLN24} \label{DKLN24} L. Diening, K. Kim, H. Lee and S. Nowak, \textit{ Nonlinear nonlocal potential theory at the gradient level}, J. Eur. Math. Soc. (2025), DOI 10.4171/JEMS/1706.

\bibitem{DuMi2} \label{DuMi2} F.~Duzaar and G.~Mingione, \textit{Gradient estimates via linear and nonlinear potentials}, J. Funct. Anal. 259 (2010), 2961--2998.

\bibitem{DM24} \label{DM24} C.~De Filippis and G.~Mingione, \textit{Gradient regularity in mixed local and nonlocal problems}, Math. Ann. 388 (2024), 261--328.

\bibitem{DPV23} \label{DPV23} S. Dipierro, E. Proietti Lippi and E. Valdinoci, \textit{(Non)local logistic equations with Neumann conditions}, Ann. Inst. H. Poincar\'{e} C Anal. Non Lin\'{e}aire 40 (2023), 1093--1166.

\bibitem{DS84} \label{DS84} R.A. DeVore and R.C. Sharpley,  \textit{Maximal functions measuring smoothness}, Mem. Amer. Math. Soc., 47 (1984), 293.

\bibitem{DV21} \label{DV21} S. Dipierro and E. Valdinoci, \textit{Description of an ecological niche for a mixed local/nonlocal dispersal: an evolution equation and a new Neumann condition arising from the superposition of Brownian and L\'{e}vy processes}, Phys. A, 575 (2021), 126052.

\bibitem{DZ24} \label{DZ24} H. Dong and H. Zhu, \textit{Gradient estimates for singular $p$-Laplace type equations with measure data}, J. Eur. Math. Soc. 26 (2024), 3939-3985.

\bibitem{EG07} \label{EG07} J. Epstein, D. Goedecke, F. Yu, R. Morris, D. Wagener and G. Bobashev, \textit{Controlling pandemic flu: the value of international air travel restrictions}, PLoS ONE 2 (2007), e401.

\bibitem{FSZ22} \label{FSZ22} Y. Fang, B. Shang and C. Zhang, \textit{Regularity theory for mixed local and nonlocal parabolic  $p$-Laplace equations}, J. Geom. Anal. 32 (2022), Paper No. 22, 33 pp.

\bibitem{GK22} \label{GK22} P. Garain and J. Kinnunen, \textit{On the regularity theory for mixed local and nonlocal quasilinear elliptic equations}, Trans. Amer. Math. Soc. 375 (2022), 5393--5423.

\bibitem{GK24} \label{GK24} P. Garain and J. Kinnunen, \textit{On the regularity theory for mixed local and nonlocal quasilinear parabolic equations}, Ann. Sc. Norm. Super. Pisa Cl. Sci. 25 (2024), 495--540.

\bibitem{Giu} \label{Giu} E.~Giusti, \textit{Direct Methods in the Calculus of Variations},  World Scientific, Singapore (2003).

\bibitem{KiMa} \label{KiMa} T.~Kilpel\"{a}inen and J.~Mal\'{y}, \textit{Degenerate elliptic equations with measure data and nonlinear potentials}, Ann. Scuola Norm. Sup. Pisa Cl. Sci. 19 (1992), 591--613.

\bibitem{KiMa2} \label{KiMa2} T.~Kilpel\"{a}inen and J.~Mal\'{y}, \textit{The Wiener test and potential estimates for quasilinear elliptic equations}, Acta Math. 172 (1994), 137--161.

\bibitem{KuMi} \label{KuMi} T.~Kuusi and G.~Mingione, \textit{Universal potential estimates}, J. Funct. Anal. 262 (2012), 4205--4269.

\bibitem{KM13} \label{KM13} T.~Kuusi and G.~Mingione, \textit{Linear potentials in nonlinear potential theory}, Arch. Ration. Mech. Anal. 207 (2013), 215--246.

\bibitem{KM14-1} \label{KM14-1} T.~Kuusi and G.~Mingione, \textit{Riesz potentials and nonlinear parabolic equations}, Arch. Ration. Mech. Anal. 212 (2014), 727--780.

\bibitem{KuMi3} \label{KuMi3} T.~Kuusi and G.~Mingione, \textit{Vectorial nonlinear potential theory}, J. Eur. Math. Soc. 20 (2018), 929--1004.

\bibitem{KuMinSi} \label{KuMinSi} T.~Kuusi, G.~Mingione and Y.~Sire, \textit{Nonlocal equations with measure data}, Comm. Math. Phys. 337 (2015), 1317--1368.

\bibitem{KNS22} \label{KNS22} T. Kuusi, S. Nowak and Y. Sire, \textit{Gradient regularity and first-order potential estimates for a class of nonlocal equations}, to appear in Amer. J. Math..

\bibitem{L06} \label{L06} P. Lindqvist, \textit{Notes on the $p$-Laplace equation}, Univ. Jyv\"{a}skyl\"{a}, Report 102, (2006).

\bibitem{MXZ25} \label{MXZ25} L. Ma, Q. Xiong and Z. Zhang, \textit{Pointwise and oscillation estimates via Riesz potentials for mixed local and nonlocal parabolic equations}, J. Geom. Anal. 35 (2025), Paper No. 191, 40 pp.

\bibitem{MZ} \label{MZ} L. Ma and Z. Zhang, \textit{Wolff type potential estimates for stationary stokes systems with Din-BMO coefficients}, Commun. Contemp. Math. 23 (2021), Paper No. 2050064, 24 pp.

\bibitem{MZ24} \label{MZ24} L. Ma and Z. Zhang, \textit{Partial regularity and nonlinear potential estimates for Stokes systems with super-quadratic growth}, Sci. China Math. 67 (2024), 1525--1554.

\bibitem{MZZ} \label{MZZ} L. Ma, Z. Zhang and F. Zhou, \textit{Nonlinear Potential Estimates for Generalized Stokes System}, Mediterr. J. Math. 19 (2022), Paper No. 212, 38 pp.

\bibitem{Min10} \label{Min10} G. Mingione, \textit{Nonlinear aspects of Calder\'{o}n-Zygmund theory}, Jahresber. Dtsch. Math.-Ver. 112 (2010), 159-191.

\bibitem{Min} \label{Min} G. Mingione, \textit{Gradient potential estimates}, J. Eur. Math. Soc. 13 (2011), 459--486.

\bibitem{N23} \label{N23} K. Nakamura, \textit{Harnack's estimate for a mixed local-nonlocal doubly nonlinear parabolic equation}, Calc. Var. Partial Differential Equations 62 (2023), Paper No. 40, 45 pp.


\bibitem{NOS24} \label{NOS24} Q. Nguyen, J. Ok and K. Song, \textit{Wolff potentials and nonlocal equations of Lane-Emden type}, arXiv:2405.11747 (2024).

\bibitem{PV08} \label{PV08} N. Phuc and I. Verbitsky, \textit{Quasilinear and Hessian equations of Lane-Emden type}, Ann. of Math. 168 (2008), 859-914.

\bibitem{PV09} \label{PV09} N. Phuc and I. Verbitsky, \textit{Singular quasilinear and Hessian equations and inequalities}, J. Funct. Anal. 256 (2009), 1875-1906.

\bibitem{TrWa} \label{TrWa} N.S.~Trudinger and X.J.~Wang, \textit{On the weak continuity of elliptic operators and applications to potential theory}, Amer. J. Math. 124 (2002), 369--410.


\bibitem{XZM} \label{XZM} Q.~Xiong, Z.~Zhang and L.~Ma, \textit{Gradient potential estimates in elliptic obstacle problems with Orlicz growth}, Calc. Var. Partial Differential Equations 61 (2022), Paper No. 83, 33 pp.


\end{thebibliography}

\end{document}